%% file: rAA1.tex
\documentclass[article,onefignum,onetabnum]{siamart190516}

\input{rAA1-SHARED}

\begin{document}


\setcounter{page}{1}
\maketitle

\begin{abstract}
	Anderson acceleration (AA) is widely used for accelerating the convergence of an underlying fixed-point iteration $\bm{x}_{k+1} = \bm{q}( \bm{x}_{k} )$, $k = 0, 1, \ldots$, with $\bm{x}_k \in \mathbb{R}^n$, $\bm{q} \colon \mathbb{R}^n \to \mathbb{R}^n$.
Despite AA's widespread use, relatively little is understood theoretically about the extent to which it may accelerate the underlying fixed-point iteration.
To this end, we analyze a restarted variant of AA with a restart size of one, a method closely related to GMRES(1). 
We consider the case of $\bm{q}( \bm{x} ) = M \bm{x} + \bm{b}$ with matrix $M \in \mathbb{R}^{n \times n}$ either symmetric or skew symmetric.
For both classes of $M$ we compute the worst-case root-average asymptotic convergence factor of the AA method, partially relying on conjecture in the symmetric setting, proving that it is strictly smaller than that of the underlying fixed-point iteration.
For symmetric $M$, we show that the AA residual iteration corresponds to a fixed-point iteration for solving an eigenvector-dependent nonlinear eigenvalue problem (NEPv), and we show how this can result in the convergence factor strongly depending on the initial iterate, which we quantify exactly in certain special cases.
Conversely, for skew-symmetric $M$ we show that the AA residual iteration is closely related to a power iteration for $M$, and how this results in the convergence factor being independent of the initial iterate.
Supporting numerical results are given, which also indicate the theory is applicable to the more general setting of nonlinear $\bm{q}$ with Jacobian at the fixed point that is symmetric or skew symmetric.
\end{abstract}

\begin{keywords}
	Anderson acceleration, AA, Krylov methods, GMRES, fixed-point iteration, asymptotic convergence, eigenvector nonlinearity, NEPv, multigrid
\end{keywords}

%
\section{Introduction}
\label{sec:introduction}
Consider the problem of finding the fixed point $\bm{x} \in \mathbb{R}^n$ of the function ${\bm{q} \colon \mathbb{R}^n \to \mathbb{R}^n}$. That is, solving the system of equations
\begin{align} \label{eq:FP-prob}
\bm{x} = \bm{q}(\bm{x}).
\end{align}
The simplest iterative method for approximating $\bm{x}$ is the fixed-point iteration
\begin{align} \label{eq:FPI} 
\bm{x}_{k+1} = \bm{q}(\bm{x}_{k}), \quad k = 0, 1, \ldots,
\end{align}
starting from an initial iterate $\bm{x}_0 \approx \bm{x}$. 

Iterations of the type \eqref{eq:FPI} are ubiquitous throughout computational science, even if they are not always presented in such abstract form.
In many applications the convergence rate of \eqref{eq:FPI} may be undesirably slow, and so it is commonplace to accelerate it.
One such widely used method for this is Anderson acceleration (AA) \cite{Anderson1965}:
\begin{align} \label{eq:AA} \tag{AA}
\bm{x}_{k+1} 
= 
\bm{q}(\bm{x}_k)
+
\sum \limits_{i = 1}^{m_k} \beta_k^{(i)}  \big( \bm{q}(\bm{x}_k) - \bm{q}(\bm{x}_{k-i}) \big), 
\quad k = 0, 1, \ldots
\end{align} 
Defining the residual of \eqref{eq:FP-prob} at the point $\bm{x}_k \approx \bm{x}$ as
\begin{align} \label{eq:r-def}
\bm{r}_k = \bm{x}_k - \bm{q}(\bm{x}_k),
\end{align}
the extrapolation coefficients in \eqref{eq:AA} are found by solving the least squares problem
\begin{align} \label{eq:beta-opt-def}
\bm{\beta}_{k} 
=
\big( \beta_k^{(1)}, \ldots, \beta_{k}^{(m_k)} \big)
=
\argmin \limits_{\bm{\beta} \in \mathbb{R}^{m_k}}
\frac{1}{2}
\left\Vert 
\bm{r}_k
+
\sum \limits_{i = 1}^{m_k} \beta_k^{(i)} \big( \bm{r}_k - \bm{r}_{k-i} \big)
\right\Vert^2
\end{align}
at each iteration $k$.
Here $\Vert \cdot \Vert$ is the $\ell^2$-norm.
One hopes that the sequence of iterates \eqref{eq:AA} converges more quickly to the solution of \eqref{eq:FP-prob} than \eqref{eq:FPI} does.
The AA variant \eqref{eq:AA} is consistent with those considered in \cite{Anderson1965,Walker_Ni_2011,Toth_Kelley_2015}, but we note that other variants exist,
including some based on nonlinear Krylov methods \cite{DeSterck-2012-NGMRES,Washio-Oosterlee-1997}.

Commonly used in practice is windowed AA, denoted as AA$(m)$, where $\bm{x}_{k+1}$ is computed from a sliding window of $m+1$ previous iterates, $m_k = \min (m, k)$.
Little is known about the asymptotic convergence behavior of AA$(m)$, even in the simplest setting of affine $\bm{q}$ and small $m$.
Toth \& Kelley \cite{Toth_Kelley_2015} first proved convergence of the method some 50 years after it was initially proposed \cite{Anderson1965}, wherein they showed that if $\bm{q}(\bm{x}) = M \bm{x} + \bm{b}$, then convergence of AA$(m)$ is, in a certain sense, ``no worse'' than that of the underlying iteration \eqref{eq:FPI}.
Certain local results for AA$(m)$ are known, showing that on a given iteration AA may provide acceleration \cite{Evans-etal-2020,DeSterck-etal-2024}.
However, only very recently was it shown by Garner et al. \cite{Garner-etal-2023} that AA$(m)$ does asymptotically converge faster than the underlying method \eqref{eq:FPI}, at least when $\bm{q}( \bm{x} ) = M \bm{x} + \bm{b}$ with $M$ a symmetric matrix, or when $\bm{q}(\bm{x})$ has a symmetric Jacobian at the fixed point.

Despite the progress made in \cite{Garner-etal-2023}, much remains unclear about the convergence of AA.
For example, its convergence factor has been empirically observed to depend on the initial iterate, sometimes strongly, and theoretically this is not well understood \cite{DeSterck-He-2022-AA-conv,Chen-Vuik-2023,DeSterck-etal-2024}. This phenomenon means that standard techniques used for analyzing related stationary linear iterations, such as those in \cite{DeSterck-He-2021-stationary-AA,Hong-Yavneh-2021}, are not applicable for analyzing it.
Moreover, while \cite{Garner-etal-2023} bounded the asymptotic convergence factor of AA$(m)$, it is not totally clear how tight this bound is, and it remains unclear whether it is possible to compute exactly the convergence factor of AA in certain special cases.
The current paper investigates both of these questions in the context of a restarted AA method, providing exact expressions for its asymptotic convergence factor in certain cases, including its exact dependence on the initial iterate.
To the best of our knowledge, these are the first results of their kind for any AA method.

Restarted AA is in some respects simpler than windowed AA, but they are similar enough that ultimately our analysis may have implications for analyzing windowed AA too.
Nonetheless, restarted AA-like methods are certainly used in practice \cite{Fang_Saad_2009,
Banerjee-etal-2016,
Henderson-etal-2019,
Ouyang-etal-2024,
He-2022-gda,
Pasini2019,
Both-etal-2019}, including to solve linear systems \cite{
Pratapa-Suryanarayana-2015,
Pratapa-etal-2016-AR,
Suryanarayana-etal-2019-AR}.
The AA iteration analyzed here is the same as that analyzed by Both et al. \cite{Both-etal-2019}, who proved several convergence results, and some results here can be seen as extensions of those results.
The remainder of this paper is organized as follows.
\Cref{sec:AA1} presents simplifying assumptions considered in this paper and the develops the associated AA residual propagator.
\Cref{sec:con-and-R-dif} introduces convergence notions and analyzes the differentiability of the residual propagator.
\Cref{sec:symm,sec:skew} analyze theoretically and numerically the convergence of the AA method when the iteration function is $\bm{q}( \bm{x} ) = M \bm{x} + \bm{b}$ with matrix $M$ symmetric and skew symmetric, respectively.
Concluding statements are given in \cref{sec:conclusion}.
The MATLAB code used to generate the numerical results in this paper is available at \url{https://github.com/okrzysik/rAA1-paper-code} (v1.0.0).

%
\section{Preliminaries}
\label{sec:AA1}

\cref{sec:ass} details the simplifying assumptions used in this paper, and \cref{sec:res-prop} develops the associated residual propagation operator.

%
\subsection{Simplifying assumptions}
\label{sec:ass}

We introduce three assumptions leading to a simplified version of the general iteration \eqref{eq:AA}.
Our goal is to consider an iteration that is simple enough that we can understand its convergence behavior, yet has convergence properties resembling those of the general iteration \eqref{eq:AA}.

\underline{\textit{A1. Affine iteration function $\bm{q}(\bm{x})$.}}
We suppose that
\begin{align} \label{eq:q-affine-def}
\bm{q}(\bm{x}) = M \bm{x} + \bm{b},
\end{align}
with known $M \in \mathbb{R}^{n \times n}$ and $\bm{b} \in \mathbb{R}^n$. 
Under this assumption, the underling fixed-point iteration \eqref{eq:FPI} becomes
\begin{align} \tag{UFPI} \label{eq:PI-iter}
\bm{x}_{k+1}
=
\bm{q}(\bm{x}_k)
=
M \bm{x}_k + \bm{b},
\quad
k = 0, 1, \ldots
\end{align}
Furthermore, with $\bm{q}$ as in \eqref{eq:q-affine-def}, \eqref{eq:FP-prob} is equivalent to a linear system of equations:
\begin{align} \label{eq:A-def}
\bm{x} = \bm{q}(\bm{x})
\quad
\Longleftrightarrow
\quad
A \bm{x} = \bm{b}, 
\quad
A := I - M.
\end{align}
We assume that $A$ is invertible, and $M \neq 0$.
With $\bm{q}$ as in \eqref{eq:q-affine-def}, \eqref{eq:AA} is as an iterative method for solving $A \bm{x} = \bm{b}$, and is closely related to GMRES for doing so \cite{Saad-Schultz-1986,Walker_Ni_2011}.
To make the analysis more tractable, we further assume: \underline{\textit{A1(i)}}. $M$ is symmetric (see \cref{sec:symm}) or \underline{\textit{A1(ii)}} $M$ is skew symmetric (see \cref{sec:skew}).
The method we analyze is closely related to GMRES(1), for which the convergence for non-normal matrices $A$ can be very complicated. For example, Embree \cite{Embree2003} shows cases where the convergence factor of GMRES(1) exhibits fractal-like dependence on the initial iterate when  $A$ is non-symmetric.
The assumption of affine $\bm{q}$ is relaxed in \cref{sec:nonlin}.

\underline{\textit{A2. Memory parameter $m = 1$.}} 
AA($m$) can exhibit extremely complicated convergence behaviour for small $m$, even for $m = 1$ \cite{DeSterck-He-2022-AA-conv,DeSterck-etal-2024}.
We also note that small $m$ is not without practical value; for example, $m = 1$ or $m = 2$ is a common choice in electronic structure computations \cite{Toth_Kelley_2015}.
Moreover, there are many real-world problems for which small $m$ may provide significant acceleration \cite{Nataj-He-2024,Pollock-Rebholz-2021-noncontractive}.

\underline{\textit{A3. Restarting.}} While windowing is perhaps used more frequently in practice, restarted AA-like iterations are also widely used \cite{Fang_Saad_2009,Banerjee-etal-2016,Henderson-etal-2019,Pratapa-Suryanarayana-2015,Ouyang-etal-2024,He-2022-gda,Pasini2019,Suryanarayana-etal-2019-AR,Pratapa-etal-2016-AR,Both-etal-2019}.
Moreover, it remains unclear whether windowing is in fact more efficient than restarting, see, e.g., numerical tests in \cite{DeSterck-He-2022-AA-conv,Garner-etal-2023}.
At each restart we discard all previous information akin to restarts used in GMRES($m$), i.e., we take $m_k = k \, \mathrm{mod} \, 2$ in \eqref{eq:AA}.

Applying A1, A2, and A3 to \eqref{eq:AA} results in the following ``rAA(1)'' iteration:
\begin{subequations}  \label{eq:rAA1-iter}
\begin{align}
\bm{x}_{k+1} 
&= 
M \bm{x}_k  + \bm{b}, 
&&k = 0, 2, 4, \ldots,
\\
\bm{x}_{k+1} 
&= 
M \bm{x}_k + \beta_k M (\bm{x}_k - \bm{x}_{k-1}) + \bm{b}, 
&&k = 1, 3, 5, \ldots,
\\
\beta_{k} &= \argmin \limits_{\beta \in \mathbb{R}} 
\left\Vert
\bm{r}_k + \beta \big( \bm{r}_k - \bm{r}_{k-1} \big)
\right\Vert^2.
\end{align}
\end{subequations}
We again note that this is equivalent to the restarted AA iteration analyzed in \cite{Both-etal-2019}.

\subsection{Residual propagation}

\label{sec:res-prop}

To analyze convergence of \eqref{eq:PI-iter} and \eqref{eq:rAA1-iter} we study their associated residual iterations.
Imposing $\bm{q}$ is as in \eqref{eq:q-affine-def}, the $k$th residual vector \eqref{eq:r-def} is
$
\bm{r}_k
=
(I - M) \bm{x}_k - \bm{b}
=
A \bm{x}_k - \bm{b}.
$
As such, the $k$th iterate can be written $\bm{x}_k = A^{-1} (\bm{r}_k + \bm{b})$.
Plugging into \eqref{eq:PI-iter} gives the residual iteration:
\begin{align} 
\label{eq:PI-iter-res}
\tag{UFPI-res}
\bm{r}_{k+1}
=
M \bm{r}_k,
\quad
k = 0, 1, \ldots
\end{align}
Similarly, plugging $\bm{x}_k = A^{-1} (\bm{r}_k + \bm{b})$ into \eqref{eq:rAA1-iter} gives the rAA(1) residual iteration:
\begin{subequations} \label{eq:rAA1-res}
\begin{align}
\bm{r}_{k+1} \label{eq:rAA1-res-a}
&= 
M \bm{r}_k, 
&&k = 0, 2, 4, \ldots,
\\
\bm{r}_{k+1} \label{eq:rAA1-res-b}
&= 
M \big[ \bm{r}_k + \beta_k (\bm{r}_k - \bm{r}_{k-1}) \big], 
&&k = 1, 3, 5, \ldots
\end{align}
\end{subequations}
We can write \eqref{eq:rAA1-res} as a two-step iteration by eliminating the $k = 0, 2, 4, \ldots$ updates resulting in
$
\bm{r}_{k+1} 
=
M [ \bm{r}_k + \beta_k (\bm{r}_k - \bm{r}_{k-1}) ] 
=
M [ I - (1 + \beta_k) A  ] \bm{r}_{k-1},
k = 1, 3, 5, \ldots,
$
where 
$
1 + \beta_k 
= 
1 +
\argmin_{\beta \in \mathbb{R}} 
\left\Vert
M \bm{r}_{k-1} - \beta A \bm{r}_{k-1}
\right\Vert^2
=
\frac{\la A\bm{r}_{k-1}, \bm{r}_{k-1} \ra}{\la A \bm{r}_{k-1}, A \bm{r}_{k-1} \ra}$.
As such, \textit{every second} rAA(1) residual \eqref{eq:rAA1-iter} satisfies
\begin{align} \label{eq:rAA1-two-step-iter-res} \tag{rAA-res}
\bm{r}_{k+1} = {\cal R}( \bm{r}_{k-1} ), \quad k = 1, 3, 5, \ldots
\end{align}
with ${\cal R} : \mathbb{R}^n \setminus \{ \bm{0} \} \to \mathbb{R}^n$ the  \textit{two-step} residual propagation operator given by
\begin{subequations}
\begin{align} \label{eq:calR-def}
{\cal R}(\bm{v}) 
&:= 
M \big[ I - \alpha(\bm{v}) A \big] \bm{v}, 
\\
\intertext{
where $\alpha : \mathbb{R}^n \setminus \{ \bm{0} \} \to \mathbb{R}$ is defined by}
\label{eq:alpha-def}
\alpha( \bm{v} ) 
&:= 
\frac{\la A\bm{v}, \bm{v} \ra}{\la A \bm{v}, A \bm{v} \ra}
=
\frac{\la (A\bm{v}), A^{-1} (A\bm{v}) \ra}{\la (A\bm{v}), (A\bm{v}) \ra}
=
\frac{\la (A\bm{v}), Y (A\bm{v}) \ra}{\la (A\bm{v}), (A\bm{v}) \ra},
\end{align}
\end{subequations}
with $Y := \tfrac{1}{2} \big(A^{-1} + A^{-\top} \big)$ the symmetric part of $A^{-1}$.
That is, $\alpha( \bm{v} )$ is the Rayleigh quotient between the vector $A \bm{v}$ and the symmetric part of $A^{-1}$. Note that this Rayleigh quotient not being defined at zero is consistent with the least squares problem for $\beta_k$ being rank deficient, which occurs if and only if the iteration \eqref{eq:rAA1-iter} produces the exact solution, $\bm{x}_{k-1} = \bm{x}$.

\textit{Relationship between rAA(1) and GMRES(1):} 
When applied to solve $A \bm{x} = \bm{b}$, GMRES(1) residuals $\bm{r}_{k}^{\mathrm G}$ satisfy the one-step iteration
$\bm{r}_{k+1}^{\mathrm G} = [I - \alpha(\bm{r}_{k}^{\mathrm G}) A] \bm{r}_{k}^{\mathrm G}$, $k = 0, 1, \ldots$; see \cite[Section 5.3.2]{Saad2003}.\footnote{The so-called ``minimal residual iteration'' from \cite[Section 5.3.2]{Saad2003} is GMRES(1).} 
Conversely, rAA(1) residuals $\bm{r}_{k}^{\mathrm A}$ satisfy the two-step iteration $\bm{r}_{k+1}^{\mathrm A} = M  [I - \alpha(\bm{r}_{k-1}^{\mathrm A}) A] \bm{r}_{k-1}^{\mathrm A}$, $k = 1, 3, \ldots$.
Thus, rAA(1) computes $\bm{r}_{k+1}^{\mathrm A}$ from $\bm{r}_{k-1}^{\mathrm A}$ by taking a GMRES(1) step on $\bm{r}_{k-1}^{\mathrm A}$, and then applying $M$ to the resulting vector.
In other words, two consecutive rAA(1) steps consist of taking a GMRES(1) step on the current residual followed by applying the fixed-point iteration matrix $M$.

\section{Convergence notions, and differentiability of residual propagator}
\label{sec:con-and-R-dif}

To characterize the convergence of \eqref{eq:rAA1-iter} we must quantify what we mean by convergence. 
Here we consider root-convergence of the residuals; see \cite[Section 9.2]{Ortega2000}.
\begin{definition}[Root convergence]
\label{def:conv-fac}
Consider an iterative method producing the sequence of residual vectors $\{  \bm{r}_k \} \subset \mathbb{R}^n $ converging to zero. Then the number
\begin{align}
\varrho( \bm{r}_0 ) := \limsup \limits_{k \to \infty} \varrho_k( \bm{r}_0 ),
\quad
\mathrm{where}
\quad
\varrho_k( \bm{r}_0 )
:=
\Vert \bm{r}_{k} \Vert^{1/k},
\end{align}
is the root-convergence factor of the method.
If $\varrho( \bm{r}_0 ) \in (0, 1)$, we say that the method converges r-linearly with r-linear convergence factor $\varrho( \bm{r}_0 )$.
We call
\begin{align}
\varrho^{\rm{worst}} 
:= 
\max_{\bm{r}_0} \varrho( \bm{r}_0 ) 
\end{align}
the worst-case root-convergence factor of the method.
\end{definition}
We also recall the following classic result which states that if the iteration function of a fixed-point iteration is differentiable at the underlying fixed point, then the worst-case r-linear convergence factor is determined by the spectral radius of the Jacobian of the iteration function at the fixed point (see \cite[Theorems 10.1.3 \& 10.1.4]{Ortega2000}).
\begin{theorem}[Ostrowski] \label{thm:ostrowski}
Suppose that ${\cal A} \colon D \subset \mathbb{R}^n \to \mathbb{R}^n$ has a fixed-point $\bm{x}_*$ that is an interior point of $D$ and is differentiable at $\bm{x}_*$. If the spectral radius of ${\cal A}'( \bm{x}_* )$ satisfies $ \rho({\cal A}'( \bm{x}_* )) \in (0, 1)$, then the fixed-point iteration $\bm{x}_{k+1} = {\cal A}( \bm{x}_k )$, $k = 0, 1, \ldots,$ converges r-linearly with r-linear convergence factor $\rho({\cal A}'( \bm{x}_* ))$.

\end{theorem}

A straightforward application of \cref{thm:ostrowski} shows that the r-linear convergence factor of the residual iteration \eqref{eq:PI-iter-res} is
\begin{align} \label{eq:rho-PI-wc}
\varrho^{\rm{worst}}_{\rm{FP}} = \rho(M).
\end{align}
Furthermore, for \eqref{eq:PI-iter-res} the initial-residual-specific convergence factor $\varrho_{\rm FP}( \bm{r}_0 )$ depends on the initial residual $\bm{r}_0$ only when $\bm{r}_0$ is parallel to an eigenvector of $M$, and otherwise it does not---the iteration \eqref{eq:PI-iter-res} is nothing other than a non-normalized power iteration for computing the dominant eigenvector of $M$.

Now consider the rAA(1) residual iteration \eqref{eq:rAA1-two-step-iter-res}.
Immediately we encounter the issue that the iteration function ${\cal R}$ is not defined at $\bm{r}_{k-1} = \bm{0}$, and, thus, $\bm{r}_{k-1} = \bm{0}$ cannot possibly be a fixed-point of $\bm{r}_{k+1} = {\cal R}( \bm{r}_{k-1} )$.
Recalling that zero is not in the domain of ${\cal R}$ due to rank deficiency of the least squares problem in \eqref{eq:rAA1-iter}, it is possible to extend the domain of ${\cal R}$ by choosing one of the infinite least squares solutions when zero is encountered. For example, in the context of the more general iteration \eqref{eq:AA}, De Sterck \& He \cite{DeSterck-He-2022-AA-conv} advocate using the pseudo inverse when rank deficiency is encountered.
However, none of these infinite solutions is particularly meaningful in the current setting since this rank deficiency is synonymous with the exact solution being found, at which point the iteration should be halted; note that rank deficiency in \eqref{eq:beta-opt-def} for the general iteration \eqref{eq:AA} can occur for reasons other than the exact solution being reached \cite{Pollock-Rebholz-2023,DeSterck-He-2022-AA-conv}.
Thus, while we could modify ${\cal R}$ to include zero in its domain, such a modification does not seem to be helpful and would not be reflective of how the iteration should be used in practice (i.e., halting at the exact solution).
This begs the question: What does the spectral radius of the Jacobian of ${\cal R}(\bm{v})$ as $\bm{v} \to \bm{0}$ tell us about convergence?
To this end we present \cref{lem:R-cont-dif} below which requires the following definition.

\begin{definition}[Directional derivative]
\label{def:direct-deriv}
Let $F \colon U \subset \mathbb{R}^n \to \mathbb{R}^m$ be a function on the open set $U$. We call $\mathfrak{D} F(\bm{x}, \bm{d})$ defined by
\begin{align}
\mathfrak{D} F(\bm{x}, \bm{d})
=
\lim \limits_{h \to 0^+} \frac{F(\bm{x} + h \bm{d}) - F(\bm{x})}{h}
\end{align}
the directional derivative of $F$ at $\bm{x}$ in the direction $\bm{d}$ if the limit exists.
If $F$ is differentiable at $\bm{x}$ with Jacobian $F'(\bm{x})$, then $\mathfrak{D} F(\bm{x}, \bm{d}) = F'( \bm{x} ) \bm{d}$.

\end{definition}

\begin{lemma}[Continuity and differentiability of ${\cal R}$]
\label{lem:R-cont-dif}
The following properties hold for the residual propagator $
{\cal R}( \bm{v} ) := M[I - \alpha(\bm{v}) A] \bm{v}
$ defined in \eqref{eq:calR-def}.

\begin{enumerate}

\item While ${\cal R}(\bm{v})$ is not defined at $\bm{v} = \bm{0}$, its limit exists there:
$
\lim_{ \bm{v} \to \bm{0} }
{\cal R}( \bm{v} )
=
\bm{0}.
$

\item 

\begin{enumerate}

\item The directional derivative of ${\cal R}$ at $\bm{v} \neq \bm{0}$ in the direction $\bm{d}$ is 
\begin{align} \label{eq:skew-DR}
\mathfrak{D} {\cal R}(\bm{v}, \bm{d})
=
M [I - \alpha( \bm{v} ) A] \bm{d}
- 
\la \alpha'( \bm{v} ), \bm{d} \ra  M A \bm{v},
\end{align}
where the column vector
\begin{align} \label{eq:alpha-prime}
\alpha'( \bm{v} ) = \frac{2}{\Vert A \bm{v} \Vert^2} A^\top
\big[Y - \alpha( \bm{v} ) I \big] A \bm{v} \in \mathbb{R}^n
\end{align}
is the gradient of $\alpha( \bm{v} )$ in \eqref{eq:alpha-def}, with $Y$ the symmetric part of $A^{-1}$.

\item The function ${\cal R}(\bm{v})$ is differentiable at $\bm{v} \neq \bm{0}$ if and only if $A \bm{v}$ is an eigenvector of $Y$, with Jacobian ${\cal R}'(\bm{v}) = M [I - \alpha( \bm{v} ) A]$. 

\item The limit $\lim_{ \bm{v} \to \bm{0} } \mathfrak{D} {\cal R}(\bm{v}, \bm{d})$ does not exist. 

\end{enumerate}

\end{enumerate}

\end{lemma}

Before proving these properties some commentary is warranted.
From property 1, the fact that $\lim_{ \bm{v} \to \bm{0} }
{\cal R}( \bm{v} )
=
\bm{0}$ is consistent with the fixed-point iteration $\bm{v} \gets {\cal R}( \bm{v} )$ converging to $\bm{v} = \bm{0}$ despite $\bm{v} = \bm{0}$ not being a fixed point of ${\cal R}$.
Interestingly property 2(b) states that the iteration function is not differentiable at all points $\bm{v} \neq \bm{0}$, but only at certain special points.
This suggests that the Jacobian of ${\cal R}(\bm{v})$ may not exist as $\bm{v} \to \bm{0}$, and indeed, property 2(c) tells us that even the limit of the directional derivative of ${\cal R}(\bm{v})$ does not exist as $\bm{v} \to \bm{0}$.
Unfortunately, this dashes our hopes of assessing convergence by considering the Jacobian of ${\cal R}(\bm{v})$ as $\bm{v} \to \bm{0}$.

De Sterck \& He \cite{DeSterck-He-2022-AA-conv} computed the directional derivative of the AA$(m)$ propagator, and, while they ultimately found that it was not differentiable at the fixed point, they pondered whether convergence information is contained in this directional derivative.
In \cref{sec:skew} we show this is the case when rAA(1) is applied to skew symmetric $M$.

\begin{proof}
1. From \eqref{eq:calR-def} and \eqref{eq:alpha-def} it is clear that the domain of ${\cal R}$ is equal to that of $\alpha$, and $\alpha(\bm{v})$ is defined everywhere except at $\bm{v} = \bm{0}$. Note that the limit 
$
\lim_{ \bm{v} \to \bm{0} } \alpha( \bm{v} ) = \lim_{ \bm{w} \to \bm{0} } \la \bm{w}, Y \bm{w} \ra / \la \bm{w}, \bm{w} \ra 
$
does not exist because this value depends on how $\bm{v} \to \bm{0}$, recalling that $\alpha(\bm{v})$ can attain any value in the range of the Rayleigh quotient of $Y$ depending on how $\bm{v} \to \bm{0}$.
Despite this, the limit $\lim_{ \bm{v} \to \bm{0} } {\cal R}( \bm{v} )$ does exist as we now demonstrate.
For any $\bm{v} \neq \bm{0}$ we have $0 \leq \Vert M[I - \alpha(\bm{v}) A] \bm{v} \Vert \leq X \cdot \Vert \bm{v} \Vert$, with $X = \max_{\bm{v} \neq \bm{0}} \Vert M[I - \alpha(\bm{v}) A] \Vert$. 
Taking limits on all sides gives
$0 
\leq 
\lim_{ \bm{v} \to \bm{0} } \Vert {\cal R}( \bm{v} ) \Vert
\leq 
X \cdot \lim_{ \bm{v} \to \bm{0} } \Vert \bm{v} \Vert = 0$. 
By the squeeze theorem and the norm property that 
$\Vert \bm{z} \Vert = 0$ if and only if $\bm{z} = \bm{0}$ we have $\lim_{ \bm{v} \to \bm{0} }
{\cal R}( \bm{v} ) = \bm{0}$.\\
2(a). Recall \cref{def:direct-deriv}. For any $h > 0$ and $\bm{v} \neq \bm{0}$ we have
\begin{align*}
&\frac{{\cal R}(\bm{v} + h \bm{d}) -  {\cal R}(\bm{v})}{h}
=
\frac{M[I - \alpha(\bm{v} + h \bm{d}) A] ( \bm{v} + h \bm{d} ) - M[I - \alpha(\bm{v}) A] \bm{v}}{h},
\\
&=
\frac{
M[I - \alpha( \bm{v} + h \bm{d} ) A] h \bm{d}
+
[ \alpha(\bm{v}) - \alpha( \bm{v} + h \bm{d} ) ] M A \bm{v}}{h},
\\
&=
M[I - \alpha( \bm{v} + h \bm{d} ) A] \bm{d}
-
[ \la \alpha'( \bm{v} ), \bm{d} \ra   + {\cal O}(h) ] M A \bm{v}.
\end{align*}
Taking $h \to 0$ gives the directional derivative in \eqref{eq:skew-DR}.
To compute the derivative $\alpha'( \bm{v} )$ apply the quotient rule to the right-most expression for $\alpha(\bm{v})$ in \eqref{eq:alpha-def}.\\
2(b). For $\bm{v}$ where the scalar $\la \alpha'( \bm{v} ), \bm{d} \ra $ in \eqref{eq:skew-DR} is non-zero, $\mathfrak{D} {\cal R}(\bm{v}, \bm{d})$ cannot be written as a matrix-vector product ${\cal R}'(\bm{v}) \bm{d}$, meaning that ${\cal R}$ is not differentiable there, as per \cref{def:direct-deriv}.
However, for $\bm{v}$ where $\la \alpha'( \bm{v} ), \bm{d} \ra $ vanishes, the directional derivative can be written as a matrix-vector product ${\cal R}'(\bm{v}) \bm{d}$, and hence ${\cal R}$ is differentiable at such $\bm{v}$.
From \eqref{eq:alpha-prime}, note that $\alpha'( \bm{v} )$ vanishes for $\bm{v} \neq \bm{0}$ if and only if $Y(A \bm{v}) = \alpha(\bm{v}) (A \bm{v}) = \tfrac{\la (A \bm{v}), Y (A \bm{v}) \ra}{\la (A \bm{v}) ,(A \bm{v}) \ra} (A \bm{v})$, i.e, $A \bm{v}$ is an eigenvector of $Y$. \\
2(c). From point 1., $\lim_{ \bm{v} \to \bm{0} }  \alpha( \bm{v} )$ does not exist. 
As such, $\lim_{ \bm{v} \to \bm{0} } \mathfrak{D} {\cal R}(\bm{v}, \bm{d})
=
\lim_{ \bm{v} \to \bm{0} }
\big\{
M [I - \alpha( \bm{v} ) A] \bm{d}
- 
\big( \la \alpha'( \bm{v} ), \bm{d} \ra M A \bm{v} \big\}$ 
cannot exist.
\end{proof}

\section{Convergence of rAA(1) for symmetric $M$}
\label{sec:symm}

Given the previous discussion, we now pursue an alternative path to \cref{thm:ostrowski} for studying convergence of rAA(1) in the case where $M$ is a symmetric matrix.
To this end we introduce some new notation.
Let $X  \colon \mathbb{R}^n \setminus \{ \bm{0} \} \to \mathbb{R}^{n \times n}$ be the vector-dependent symmetric matrix
\begin{align} \label{eq:X-def}
X( \bm{v} ) = M^2[I - \alpha( {\cal R}(\bm{v}) ) A][I - \alpha( \bm{v} )A],
\end{align}
where ${\cal R}$ is the two-step residual propagator given in \eqref{eq:calR-def}, and notice that $X$ is defined so that $X( \bm{v} ) \bm{v} = {\cal R}({\cal R}( \bm{v} ))$. 
That is, every fourth rAA(1) residual is given by the fixed-point iteration
\begin{align} \label{eq:X-fp}
\bm{r}_{4(k+1)} = X( \bm{r}_{4k} ) \bm{r}_{4 k}, \quad k = 0, 1, \ldots
\end{align}
If the rAA(1) method converges, then clearly iterates of \eqref{eq:X-fp} limit to the zero vector. However, the relevant question for computing the asymptotic convergence factor of the method is \textit{how} exactly these vectors limit to zero.
Consider by way of analogy the residual iteration \eqref{eq:PI-iter-res}: If the method converges, then the iterates limit to the zero vector, but, importantly, they do so by limiting to the dominant eigenvector of the matrix $M$, resulting in the convergence factor \eqref{eq:rho-PI-wc}.

The remainder of this section is structured as follows. \Cref{sec:symm-eig} considers in further detail the iteration \eqref{eq:X-fp}. These results are then used to compute convergence factors in \cref{sec:symm-conv}, with numerical results following in \cref{sec:symm-num-res,sec:nonlin}.
%

\subsection{An eigenvector-dependent nonlinear eigenvalue problem}
\label{sec:symm-eig}

Here we consider in more detail the fixed-point iteration \eqref{eq:X-fp}.
To this end, let us consider the following problem involving the symmetric matrix $X$ from \eqref{eq:X-def}:
\begin{align} \label{eq:NEPv1}
X( \bm{z} ) \bm{z} = \lambda \bm{z}.
\end{align}
This is what is known as an eigenvector-dependent nonlinear eigenvalue problem, commonly abbreviated as a NEPv or NEPv1 (specifying single eigenvector nonlinearity).
NEPv's arise in, for example, quantum chemistry \cite{Cai-etal-2018}, Rayleigh quotient optimization \cite{Bai-etal-2018}, and $p$-Laplacian eigenproblems \cite{Upadhyaya-etal-2021}, and have seen increasing attention over recent years.
However, we are unaware of their application previously for analyzing iterative methods (at least methods not directly associated with solving NEPv's), and it is interesting to consider whether they have applications for analyzing convergence of Krylov-type methods more generally.

Observe that any vector $\bm{z}$ satisfying \eqref{eq:NEPv1} is a ``fixed point'' of \eqref{eq:X-fp}, in the sense that if $\bm{r}_{4k} = \bm{z}$, then $\bm{r}_{4(k+1)} = X( \bm{z} ) \bm{z} = \lambda \bm{z} = \lambda \bm{r}_{4k}$.
This is analogous to how an eigenvector $\bm{y}$ of $M$ satisfying $M \bm{y} = m \bm{y}$ is a fixed point of the non-normalized power iteration $\bm{y}_{k+1} = M \bm{y}_k$, $k = 0, 1, \ldots$.
The following theorem characterizes (at least some of) the solutions to the NEPv1 \eqref{eq:NEPv1}.
\begin{theorem}
\label{thm:eigvec}
Let ${\cal R}$ be the two-step residual propagator \eqref{eq:calR-def}, and $X$ the four-step residual propagation matrix \eqref{eq:X-def}.
Let $0 \neq c_i, c_j \in \mathbb{R}$, and define $\bm{z}_{ij}(\varepsilon)$ as
\begin{align} \label{eq:z-def-2}
\bm{z}_{ij}( \varepsilon ) := c_i \bm{v}_i + c_j \bm{v}_j 
= 
c_i ( \bm{v}_i + \varepsilon \bm{v}_j),
\quad
\varepsilon := \frac{c_j}{c_i},
\end{align}
with $\bm{v}_i, \bm{v}_j, i \neq j,$ two orthogonal eigenvectors of $M$ with eigenvalues $m_i$ and $m_j$, respectively.
Then, if $m_i = m_j$, 
\begin{align} \label{eq:Rz-zero}
{\cal R}(\bm{z}_{ij}( \varepsilon )) = \bm{0}.
\end{align}
Otherwise, if $m_i \neq m_j$, then $\bm{z}_{ij}(\varepsilon)$ solves the NEPv1 \eqref{eq:NEPv1} with corresponding eigenvalue $\lambda_{ij}( \varepsilon ) \in \mathbb{R}$ given by
\begin{align} \label{eq:lambda}
\lambda_{ij}( \varepsilon ) 
=
\frac{(m_j - m_i)^2}{(m_i-1)^2 +  \varepsilon^2 (m_j - 1 )^2}
\,
\frac{(m_i m_j)^2}{\displaystyle m_i^2 + \frac{m_j^2}{\varepsilon^2}   }.
\end{align}

\end{theorem}

\begin{proof}
See \cref{app:proof-thm-eigvec}.
\end{proof}

Without further study, it remains unclear whether solutions to \eqref{eq:NEPv1} exist other than those characterized in \cref{thm:eigvec}.
Based on numerical evidence (see Supplementary Materials Section \ref{SMsec:conj-num-res}),
we suspect that iterates of \eqref{eq:X-fp} converge to solutions of the NEPv1: 
\begin{conjecture} \label{conj:X-fp-conv}
For any initial iterate $\bm{r}_0$, iterates of the fixed-point iteration \eqref{eq:X-fp} converge to a solution of the NEPv1 \eqref{eq:NEPv1}.
\end{conjecture}

NEPv's are typically solved numerically with so-called self-consistent field (SCF) iteration \cite{Cai-etal-2018,Upadhyaya-etal-2021}, for which the convergence analysis is not straightforward \cite{Cai-etal-2018}. 
Interestingly, the fixed-point iteration \eqref{eq:X-fp} can be interpreted as a particular instance of an inexact SCF iteration.\footnote{Given iterates $(\lambda_k, \bm{z}_k)$ approximately solving the NEPv1 \eqref{eq:NEPv1}, a single SCF iteration would compute new iterates $(\lambda_{k+1}, \bm{z}_{k+1})$ as one of the eigenpairs of the lagged, linear eigenvalue problem $X( \bm{z}_k ) \bm{z}_{k+1} = \lambda_{k+1} \bm{z}_{k+1}$.
The fixed-point iteration \eqref{eq:X-fp} is of this form. That is, given an $\bm{r}_{4k}$, \eqref{eq:X-fp} produces a new iterate $\bm{r}_{4(k+1)} = X( \bm{r}_{4k} ) \bm{r}_{4k}$ which corresponds approximately to an eigenvector of $X( \bm{r}_{4k} )$ computed with a single non-normalized power iteration initialized with $\bm{r}_{4k}$.}
As such, proving whether iterates of \eqref{eq:X-fp} converge to a solution of \eqref{eq:NEPv1}, and to which solution they converge, is likely to be rather complicated, particularly because, as per \cref{thm:eigvec}, \eqref{eq:NEPv1} has an infinite number of solutions.
While our numerical results indicate that iterates of \eqref{eq:X-fp} converge to a solution of \eqref{eq:NEPv1}, the particular solution can depend strongly on $\bm{r}_0$, which is perhaps unsurprising given the discussion above.

A consequence of \cref{thm:eigvec} is that if the residual iteration  \eqref{eq:rAA1-two-step-iter-res} is initialized with a linear combination of two eigenvectors of $M$, then the resulting residual vectors will be four-periodic---periodic in the sense that $\bm{r}_4 = \lambda \bm{r}_0$, $\bm{r}_8 = \lambda \bm{r}_4$, $\bm{r}_{12} = \lambda \bm{r}_8$, etc. for some constant $\lambda$.
We note that Both et al. \cite{Both-etal-2019} made this same observation, albeit with errors rather than residuals, wherein they used this to bound the convergence factor of rAA(1) under certain circumstances.\footnote{Specifically, what 
\cite[Lemma 4]{Both-etal-2019} shows is that if at iteration $k$ the rAA(1) error vector satisfies $\bm{e}_{k} \in \mathrm{span} \{\bm{v}_i, \bm{v}_j \} $, then $\bm{e}_{k+4} \leq c \bm{e}_{k}$, where, in our notation, $c$ is equivalent to $\max_{\varepsilon \in \mathbb{R}} \lambda_{ij}(\varepsilon)$.}
Notice that if \cref{conj:X-fp-conv} holds, then asymptotically rAA(1) residual vectors will be four-periodic, with convergence factor possibly depending on $\bm{r}_0$.
We remark that this result has a close connection with a result from Baker et al. \cite{Baker_etal_2005} who showed that residuals produced by GMRES($n-1$) on a symmetric or skew-symmetric matrix $A \in \mathbb{R}^{n \times n}$ are periodic over two restarts.
Vecharynski \& Langou \cite{Vecharynski_Langou_2010} also show related results concerning periodicity of restarted GMRES residuals.

The analysis in this section was motivated by numerical tests indicating that rAA(1) residuals often tend asymptotically to be four periodic.
More generally, our numerical tests indicate that asymptotically residuals of rAA($m$) are approximately $2(m+1)$-periodic; see some examples in Supplementary Materials Section \ref{SMsec:rAAm}.
This numerical observation is similar to that made by Baker et al. \cite{Baker_etal_2005} that GMRES($m$) residuals often appear to be $2m$-periodic.
In fact, \cref{thm:eigvec} can be generalized to show that linear combinations of $m+1$ orthogonal eigenvectors of $M$ are nonlinear eigenvectors of the matrix characterizing $2(m+1)$ residual iterations of rAA($m$).
However, we do not consider this further here since our focus is on rAA(1).

\subsection{Convergence factor}
\label{sec:symm-conv}

Here we use the results from the previous section to compute the convergence factor of rAA(1).
First we begin with the following theorem which completely characterises the convergence factor of rAA(1) for $M \in \mathbb{R}^{2 \times 2}$.
\begin{theorem}
\label{thm:2x2}
Consider the rAA(1) iteration \eqref{eq:rAA1-two-step-iter-res} applied to a symmetric matrix ${M \in \mathbb{R}^{2 \times 2}}$. 
Let $\bm{r}_0 = c_1 \bm{v}_1 + c_2 \bm{v}_2 $ be the initial residual with $c_1, c_2 \in \mathbb{R}$, and with $\bm{v}_1, \bm{v}_2$ two orthogonal eigenvectors of $M$ with eigenvalues $m_1$ and $m_2$, respectively.
Let $\varrho_{\rm AA}(\bm{r}_0)$ be the $\bm{r}_0$-specific root-convergence factor of the iteration (see \cref{def:conv-fac}).
Suppose that $c_1 = 0$ or $c_2 = 0$. Then,
\begin{align} \label{eq:rAA1-rho-special}
\varrho_{\rm AA}( \bm{r}_0 ) = 0.
\end{align}
Otherwise, suppose that $c_1, c_2 \neq 0$ and let $\varepsilon := c_2 / c_1$, so that $\bm{r}_0 \propto \bm{v}_1 + \varepsilon \bm{v}_2$.
Then,
\begin{align} \label{eq:rAA1-rho-general}
\varrho_{\rm AA}( \bm{r}_0 )
=
\left[
\frac{(m_2 - m_1)^2}{(m_1-1)^2 +  \varepsilon^2 (m_2 - 1)^2}
\,
\frac{(m_1 m_2)^2}{\displaystyle m_1^2 + \frac{m_2^2}{\varepsilon^2}   }
\right]^{1/4}.
\end{align}

\end{theorem}

\begin{figure}[b!]
\centerline{
\includegraphics[width=1\textwidth]{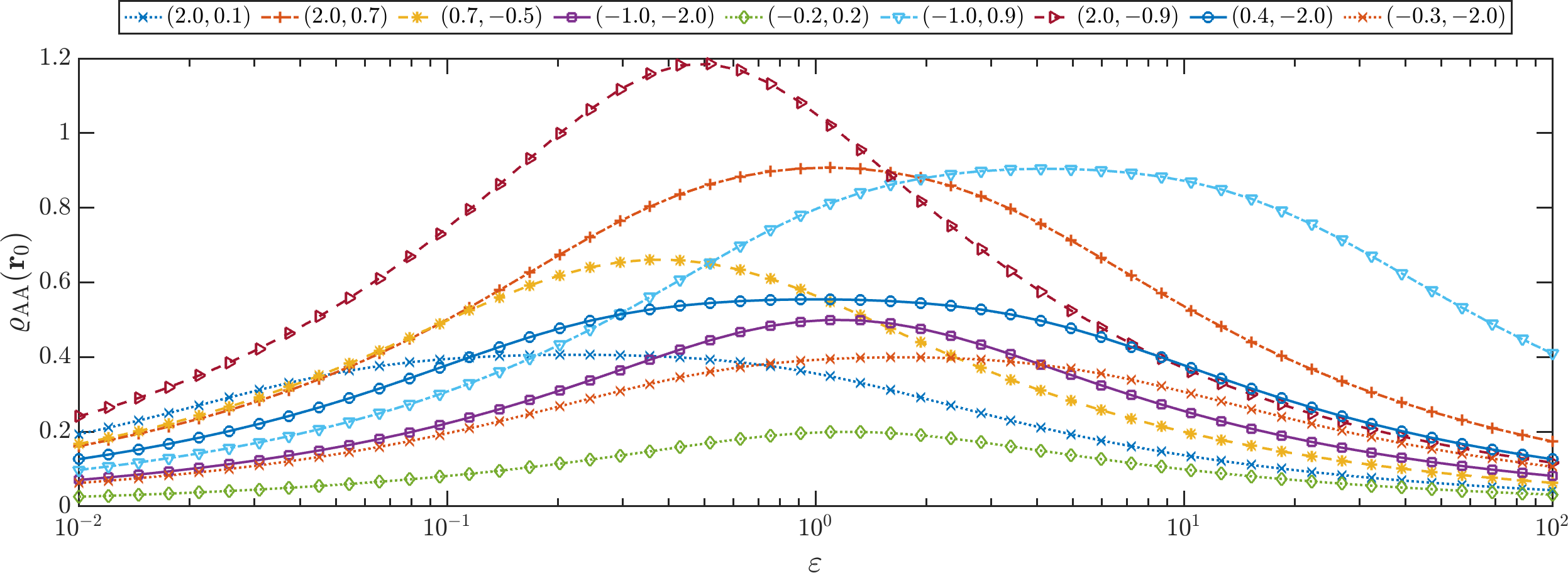}
}
\vspace{1ex}
\centerline{
\includegraphics[width = 0.475\textwidth]{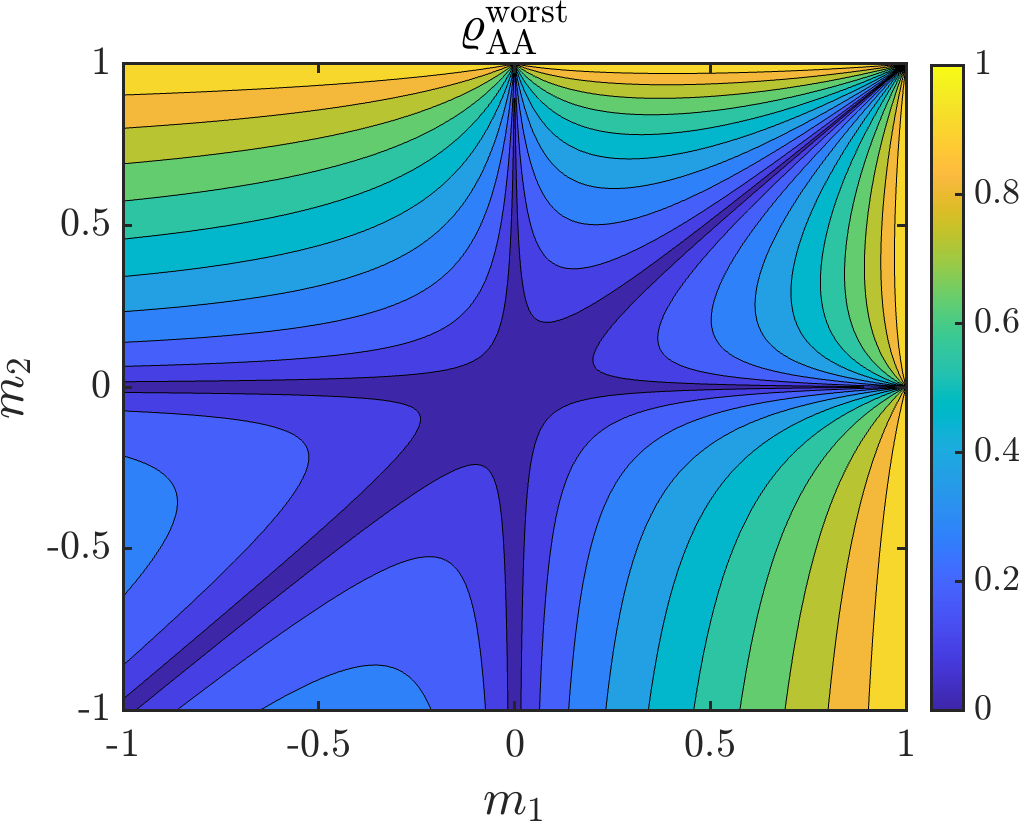}
\hspace{2ex}
\includegraphics[width = 0.475\textwidth]{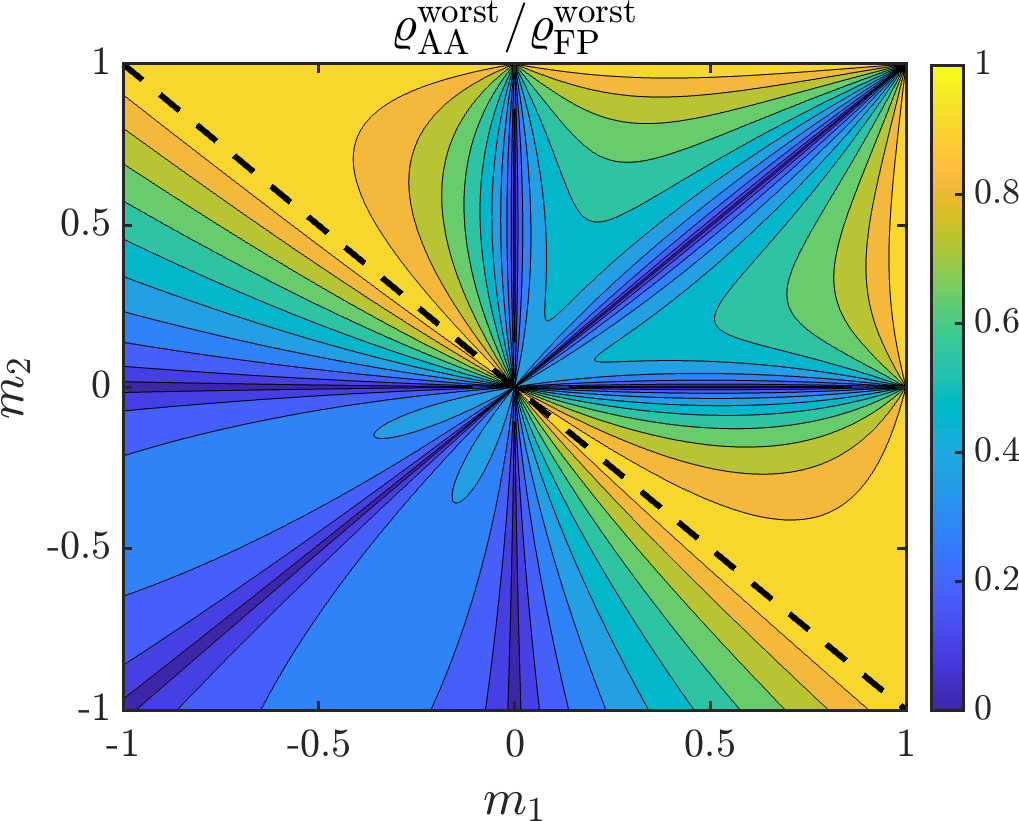}
}
\caption{Convergence of rAA(1) for symmetric $M \in \mathbb{R}^{2 \times 2}$ with eigenvalues $m_1$ and $m_2$.
\textbf{Top:} Cross-sections of the convergence factor \eqref{eq:rAA1-rho-general} as a function of $\varepsilon \in [10^{-2}, 10^2]$ at $(m_1, m_2)$ indicated in the legend.
\textbf{Bottom left:} Worst-case convergence factor for rAA(1) as in  \cref{thm:rho-symm}.
\textbf{Bottom right:} Ratio of the rAA(1) worst-case convergence factor to that of the underlying fixed-point iteration; the dashed black line is the only region where $\varrho_{\rm AA}^{\rm worst} < \varrho_{\rm FP}^{\rm worst}$ does not hold.
\label{fig:symmM-2x2}
}
\end{figure}

\begin{proof}
See \cref{app:symm-2x2-proof}.
\end{proof}
Cross-sections of the convergence factor \eqref{eq:rAA1-rho-general} can be seen in the top panel of \cref{fig:symmM-2x2}.
This result confirms that the rAA(1) asymptotic convergence factor can depend strongly on the initial iterate.
Notice \eqref{eq:rAA1-rho-general} $\to 0$ as $\varepsilon \to 0$ or $\varepsilon \to \infty$, corresponding to $\bm{r}_0 \propto \bm{v}_1$ or $\bm{r}_0 \propto \bm{v}_2$, consistent with \eqref{eq:rAA1-rho-special}.

Next we consider the more general case of the worst-case asymptotic convergence factor for symmetric $M \in \mathbb{R}^{n \times n}$, but first we require the maximum of $\lambda_{ij}(\varepsilon)$.
\begin{lemma} \label{lem:max-lambda}

Let $\lambda_{ij}(\varepsilon)$ be as in \eqref{eq:lambda}.
Then,
\begin{align} \label{eq:lambda-max}
\max \limits_{\varepsilon \in \mathbb{R}} \lambda_{ij}
(\varepsilon)
=
\lambda_{ij}
\left(
\pm \sqrt{\left| \frac{m_j (m_i - 1)}{m_i (m_j - 1)} \right| }
\right)
=
\left( \frac{m_i m_j (m_j - m_i)}{|m_i (m_i - 1)| + |m_j (m_j - 1)|} \right)^2.
\end{align}
\end{lemma}

\begin{proof}
See \cref{app:proof-lem-lambda-max}.
\end{proof}
Our main convergence factor result is now as follows.

\begin{theorem} \label{thm:rho-symm}
Let $\varrho_{\rm AA}^{\rm worst}$ be the worst-case root-convergence factor (see \cref{def:conv-fac}) for the rAA(1) iteration \eqref{eq:rAA1-two-step-iter-res} applied to a symmetric matrix $M \in \mathbb{R}^{n \times n}$ with eigenvalues $\{ m_i \}_{i = 1}^n$, $n \geq 2$.
Then, $\varrho_{\rm AA}^{\rm worst}$ is lower bounded as
\begin{align} \label{eq:symm-rho-worst}
\varrho_{\rm AA}^{\rm worst}
\geq
\max \limits_{\substack{i,j \in \{1, \ldots, n\} \\ i \neq j}} 
\left( \frac{m_i m_j (m_j - m_i)}{|m_i (m_i - 1)| + |m_j (m_j - 1)|} \right)^{1/2}.
\end{align}
Moreover, \eqref{eq:symm-rho-worst} holds with equality when $n = 2$.
\end{theorem}
A plot of this worst-case convergence factor on the space of $2 \times 2$ matrices is shown in the bottom left-hand panel of \cref{fig:symmM-2x2}.
We also note it follows as a consequence of results in \cite[Section 7]{Both-etal-2019} that if $M$ satisfies $\rho(M) < 1$, and $m_{i^*} \neq - m_{j^*}$, with $(i^*, j^*)$ the maximizing indices in \eqref{eq:symm-rho-worst}, then the right-hand side of \eqref{eq:symm-rho-worst} is strictly smaller than $\rho(M)$.

\begin{proof}
Recalling that $\varrho^{\rm worst}_{\rm AA} := \max_{\bm{r}_0} \varrho_{\rm AA}( \bm{r}_0 )$ is the worst-case convergence factor over all non-zero initial residuals, the lower bound \eqref{eq:symm-rho-worst} must hold if there is at least one $\bm{r}_0$ for which $\varrho_{\rm AA}( \bm{r}_0 )$ is equal to the right-hand side of \eqref{eq:symm-rho-worst}.
To this end, let us consider the specific initial rAA(1) residual given by $\bm{r}_0 = \bm{z}_{i^*, j^*}(\varepsilon^*)$, where $(i^*, j^*, \varepsilon^*)$ is the $(i, j, \varepsilon)$ triplet maximizing the left-hand side of \eqref{eq:lambda-max}, and $\bm{z}_{ij}(\varepsilon)$ is defined in \cref{thm:eigvec}.
Then, according to \cref{thm:eigvec}, since $\bm{r}_0$ is a linear combination of two eigenvectors of $M$, every fourth rAA(1) residual will be periodic, $\bm{r}_{4(k+1)} = \lambda_{i^*, j^*}(\varepsilon^*) \bm{r}_{4k} = \big[ \lambda_{i^*, j^*}(\varepsilon^*) \big]^k \bm{r}_{0}$. 
As such, we have for all $k = 0, 1, \ldots,$
\begin{align}
\frac{\Vert \bm{r}_{4(k+1)} \Vert }{\Vert \bm{r}_{4k} \Vert}
&=
\frac{\Vert X (\bm{r}_{4k}) \bm{r}_{4k} \Vert }{\Vert \bm{r}_{4k} \Vert}
=
\lambda_{i^*, j^*}(\varepsilon^*)
\leq
\max \limits_{\bm{v} \neq \bm{0}} \frac{\Vert X ( \bm{v} ) \bm{v} \Vert}{\Vert \bm{v} \Vert}.
\end{align}
Plugging into the definition in \cref{def:conv-fac} and then applying the same line of arguments as used in the proof in \cref{app:symm-2x2-proof} gives the claim.
\end{proof}

Another interesting result is as follows. By the same argument used in the proof in \cref{app:symm-2x2-proof}, initializing rAA(1) with a vector satisfying ${\bm{r}_0 \propto \bm{v}_{i^*} + \varepsilon \bm{v}_{j^*}}$, with $i^*, j^*$ being the maximizing indices on the right-hand side of \eqref{eq:symm-rho-worst}, results in the asymptotic convergence factor (see \cref{def:conv-fac}) $\varrho_{\rm AA}( \bm{r}_0 ) =  [\lambda_{i^*, j^*}(\varepsilon)]^{1/4}$. Recall that $[\lambda_{i^*, j^*}(\varepsilon)]^{1/4}$ is a continuous function of $\varepsilon \in (- \infty, \infty)$, taking on all values between $0$ and the right-hand side of \eqref{eq:symm-rho-worst}. 
It therefore follows that the asymptotic convergence factor $\varrho_{\rm AA}( \bm{r}_0 )$ of rAA(1) can be any real number between zero and the right-hand side of \eqref{eq:symm-rho-worst}.

Based on numerical evidence (described below), we suspect that \cref{thm:rho-symm} can be strengthened so that the lower bound in \eqref{eq:symm-rho-worst} in fact holds with equality for all $n$ and not only $n = 2$. 
More specifically, this strengthening follows as a consequence of the following conjecture which we believe to be true:
\begin{conjecture} \label{conj:symm-Xvmax}
Let $X$ be the symmetric, vector-dependent matrix in \eqref{eq:X-def} and $\lambda_{ij}(\varepsilon)$ be as in \eqref{eq:lambda}. 
Then,
\begin{align} \label{eq:symm-max}
\max \limits_{\bm{v} \neq \bm{0}} \frac{\Vert X ( \bm{v} ) \bm{v} \Vert}{\Vert \bm{v} \Vert} = \max \limits_{\substack{i,j \in \{1, \ldots, n\} \\ i \neq j}} \max \limits_{\varepsilon \in \mathbb{R}} \lambda_{ij}( \varepsilon ).
\end{align}
\end{conjecture}
Unfortunately, we have not been able to prove this conjecture for $n > 2$, but note that it holds trivially when $n = 2$ because all vectors in $\mathbb{R}^2$ can be written as a linear combination of two eigenvectors of $M$, and, thus, all vectors in $\mathbb{R}^2$ are nonlinear eigenvectors of $X(\bm{v})$ as per \cref{thm:eigvec}.
We provide numerical evidence to support the conjecture in the following two numerical results sections and in Supplementary Materials Section \ref{SMsec:conj-num-res} .
Note that \cref{conj:symm-Xvmax} is a generalization of the result that given a symmetric matrix $B$, $\max_{\bm{v} \neq \bm{0}} \Vert B \bm{v} \Vert / \Vert \bm{v} \Vert = \rho(B)$ with the max attained by the $\bm{v}$ that is the dominant eigenvector of $B$.
%

\subsection{Numerical results}
\label{sec:symm-num-res}
\begin{figure}[h!]
\centerline{
\includegraphics[width=0.475\textwidth]{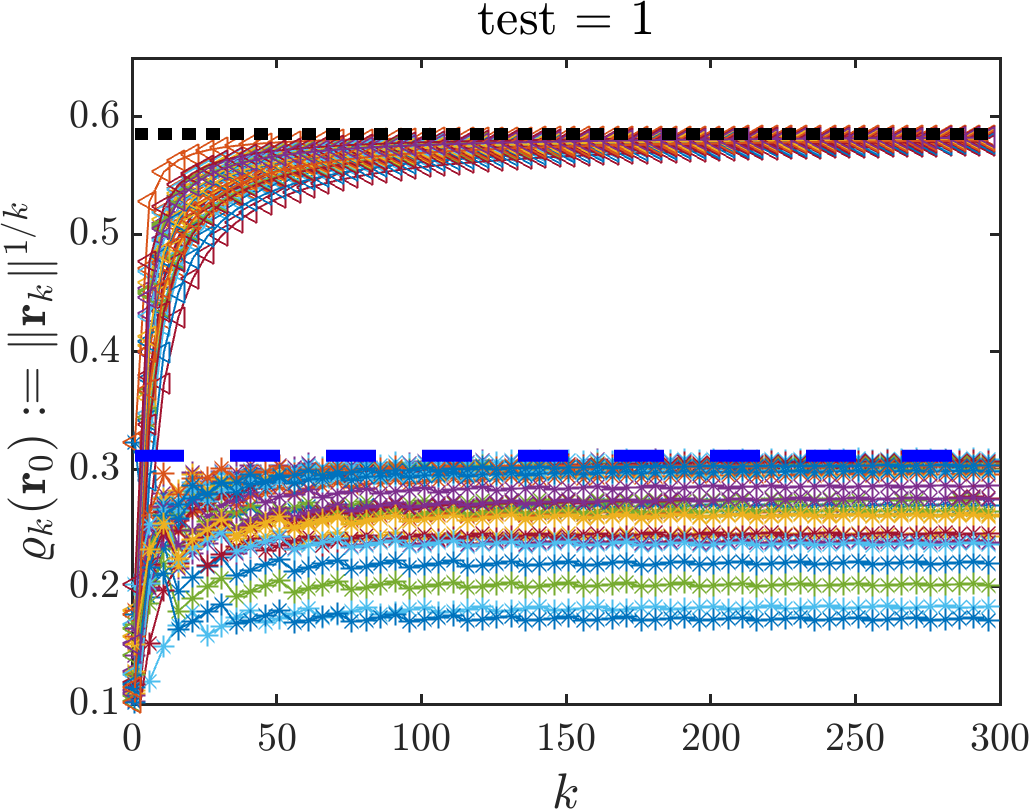}
\hspace{1ex}
\includegraphics[width=0.475\textwidth]{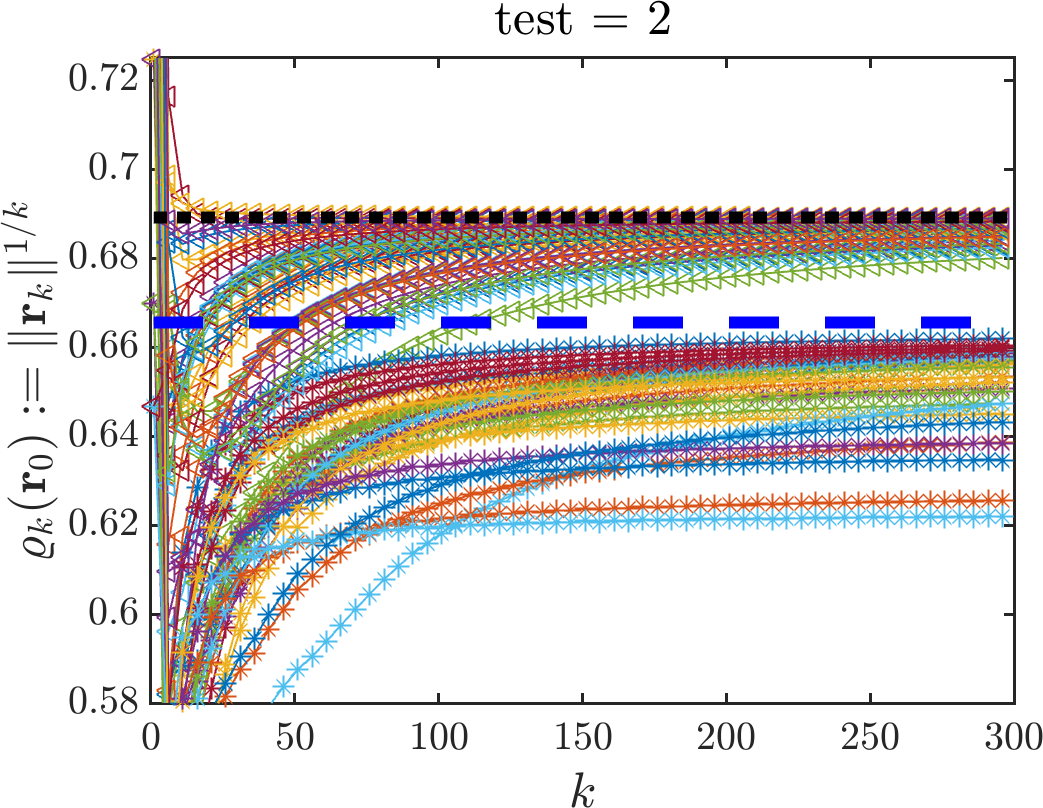}
}
\vspace{1ex}
\centerline{
\hspace{-1.5ex}
\includegraphics[width=0.475\textwidth]{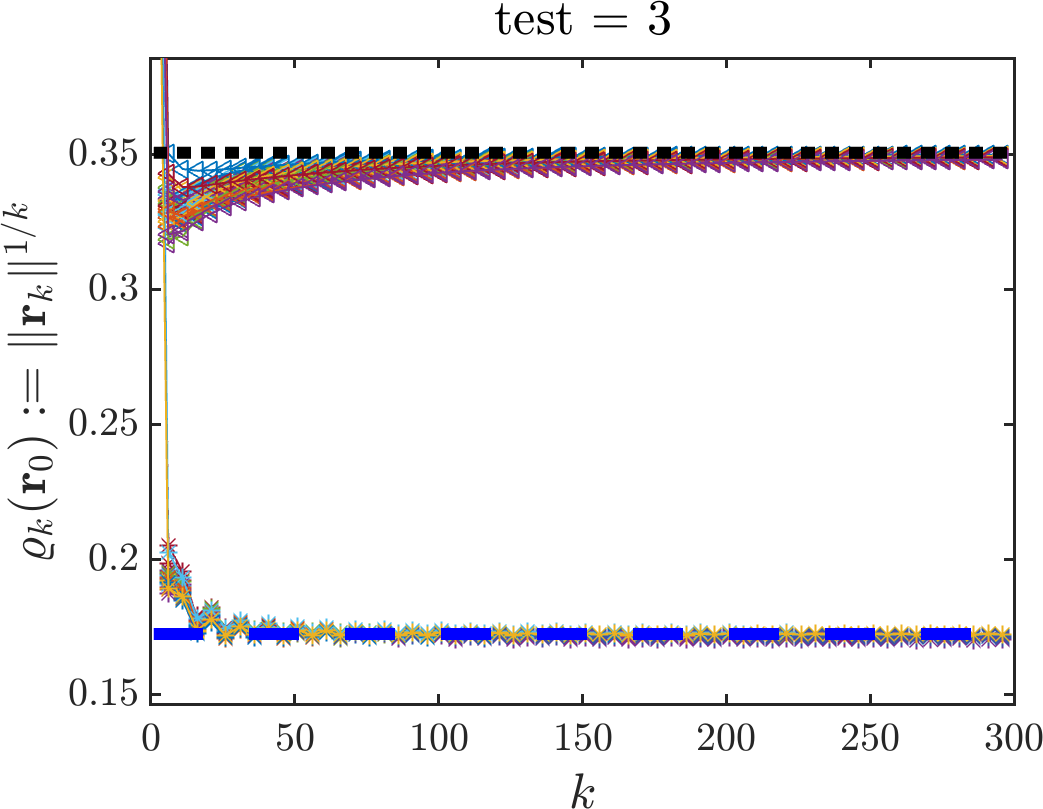}
\hspace{1.5ex}
\includegraphics[width=0.46\textwidth]{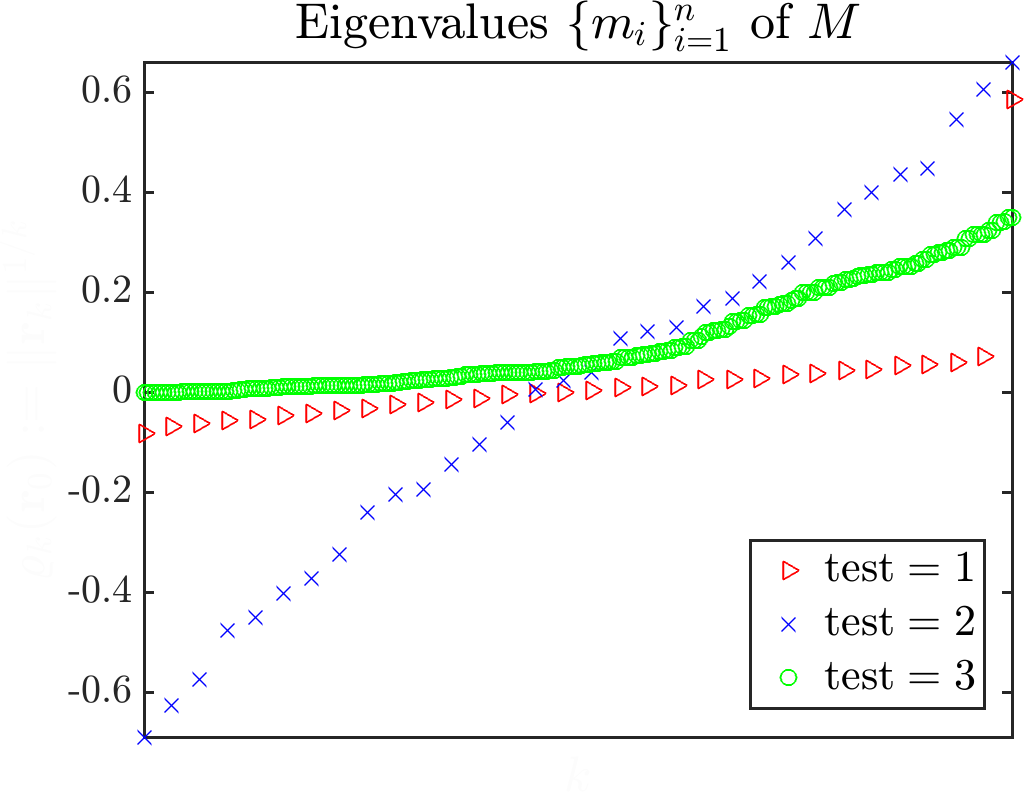}
}
\caption{Supporting numerical evidence for \cref{thm:rho-symm} using three symmetric matrices $M$. 
For each $M$, shown is the convergence factor $\varrho_k(\mathbf{r}_0)$ as a function of $k$ (see \cref{def:conv-fac}).
Triangle markers depict the underlying fixed-point iteration \eqref{eq:PI-iter}, and asterisk markers the rAA(1) iteration \eqref{eq:rAA1-iter}.
For each matrix, each algorithm is initialized with 30 different $\bm{r}_0$ chosen at random.
The thick black dotted line is the worst-case convergence factor for the underlying fixed-point iteration, and the thick blue dashed line is the lower bound for that of the rAA(1) iteration given in \cref{thm:rho-symm}.
The bottom right plot shows the eigenvalues for each test matrix $M$.
\label{fig:symm-rhok}
}
\end{figure}

We now consider three numerical tests that support both \cref{thm:rho-symm} and \cref{conj:symm-Xvmax}, with results shown in \cref{fig:symm-rhok}.
In the first two tests we take symmetric matrices $M \in \mathbb{R}^{32 \times 32}$ with randomly chosen entries.
From the bottom right of \cref{fig:symm-rhok}, observe that in the first test all eigenvalues of $M$ are clustered relatively tightly with a single outlier eigenvalue, while in the second test eigenvalues of $M$ are not particularly clustered.

For the third test we consider a more practical example: Using rAA(1) to accelerate a two-level geometric multigrid algorithm for solving the discretized Poisson problem $-(u_{xx} + u_{yy}) = f(x,y)$ for unknown $u$ subject to zero Dirichlet boundary conditions and function $f$ specified below; see \cite{Trottenberg2001,Briggs2000}.
Discretizing this partial differential equation (PDE) results in the linear system $A_h \bm{u} = \bm{f}$ with the symmetric positive definite (SPD) matrix $A_h$ the standard five-point finite-difference discretization of the Laplacian operator $-[\partial_{xx}(\, \cdot \,) + \partial_{yy}(\, \cdot \,)]$ on a grid with $n_x-1$ unknowns in each direction equispaced by distance $h$.
The vector $\bm{f}$ is the discretization of $f(x,y)$.
Applying a two-level method to solve this linear system yields the following fixed-point iteration:
\begin{align} \label{eq:symmM-mg}
\bm{u}_{k+1}
=
S K S
\bm{u}_{k}
+
B \bm{f},
\quad
K = I - \tfrac{1}{4} P A_{2h}^{-1} P^\top A_h, 
\quad
k = 0, 1, \ldots
\end{align}
Here $S$ is the (symmetric) iteration matrix for the weighted Jacobi method with weight $4/5$, $P$ is bilinear interpolation, and $A_{2h}$ is the analogue of $A_h$ on a grid with $n_x/2-1$ points equispaced by $2h$.
In our numerical example we choose $n_x = 16$, so that the total number of unknowns is $n = (n_x - 1)^2 = 225$.
The matrix $B$ is a certain function of the other matrices involved in iteration \eqref{eq:symmM-mg}.
Notice that the fixed-point iteration \eqref{eq:symmM-mg} is exactly of the form of \eqref{eq:PI-iter}, with solution $\bm{u} = \bm{x}$, forcing term $\bm{b} = B \bm{f}$, and symmetric matrix $M = S K S$.\footnote{Technically, $M$ in \eqref{eq:symmM-mg} is not symmetric because $P A_{2h}^{-1} P^\top$ does not commute with $A_h$. Since $A_h$ is SPD, consider that $M$ is similar to the matrix $\wt{M} = A_h^{1/2} M A_h^{-1/2}$, which, is in fact symmetric (see \cite{Tang-etal-2010} for related ideas), noting that $A_h^{\pm 1/2}$ commute with $S$. 
Since $M$ is equivalent to a symmetric matrix under similarity transform, we anticipate that \eqref{eq:symm-rho-worst} evaluated at the eigenvalues of $M$ will reliably compute the convergence factor of the iteration, with our numerical tests confirming as much.}
Numerical results for the underlying fixed-point iteration \eqref{eq:symmM-mg} and the associated rAA(1) iteration \eqref{eq:rAA1-iter} are shown in the bottom left of \cref{fig:symm-rhok}. In all cases the initial iterates $\bm{u}_0$ are chosen with random entries from the interval $[-1, 1]$, and we take $\bm{f} = \bm{0}$ (corresponding to $f(x,y) = 0$ in the underlying PDE).

The convergence factor plots in \cref{fig:symm-rhok} appear consistent with \cref{thm:rho-symm} since the lower bound \eqref{eq:symm-rho-worst} on the worst-case convergence factor of rAA(1) is realized for all three test matrices.
Moreover, since this lower bound is reached but never exceeded, these tests indicate that the lower bound holds with equality, which is also consistent with \cref{conj:symm-Xvmax} holding.
Moving on, observe that the asymptotic convergence factor of the underlying fixed-point iteration does not depend on the initial residual $\bm{r}_0$, while that for rAA(1) \textit{may} depend strongly on $\bm{r}_0$.
That is, for the first and second tests we see a wide-ranging set of rAA(1) convergence factors for different $\bm{r}_0$, while for the third test we see that all $\bm{r}_0$ result in essentially the same convergence factor.
It is also interesting to consider that in the first and third tests rAA(1) provides significant acceleration while in the second test it provides relatively little.
%

\subsection{Numerical results: A nonlinear extension}
\label{sec:nonlin}

To further demonstrate the predictive power of our theoretical results, in this section we present a numerical test based on relaxing assumption A1 from \cref{sec:ass}: That the underlying iteration function $\bm{q}( \bm{x} )$ is affine. Instead, let us just suppose that $\bm{q}$ has a fixed point at $\bm{x}_*$ and that it is differentiable there.
Heuristically, when the AA iterates $\bm{x}_k$ are sufficiently close to the fixed point $\bm{x}_*$, $\bm{q}(\bm{x}_k)$ will essentially behave as an affine function of $\bm{x}_k$, and thus fall into the class of functions we have already analyzed.
That is, by Taylor expansion we have
$
\bm{q}(\bm{x}_k) 
= 
\bm{q}( \bm{x}_* + [\bm{x}_k - \bm{x}_*]) 
= 
\bm{q}(\bm{x}_*) + \bm{q}'(\bm{x}_*)  [\bm{x}_k - \bm{x}_*] + {\cal O} (\Vert \bm{x}_k - \bm{x}_* \Vert^2)
=
\bm{q}'(\bm{x}_*) \bm{x}_k
+
[I -\bm{q}'(\bm{x}_*) ] \bm{x}_* + {\cal O} (\Vert \bm{x}_k - \bm{x}_* \Vert^2)
\approx
M_* \bm{x}_k + \bm{b}_*,
$
where $M_* := \bm{q}'(\bm{x}_*)$ and $\bm{b}_* := (I - M_*) \bm{x}_*$.
As such, it seems reasonable to expect that when sufficiently close to the fixed point $\bm{x}_*$ of $\bm{q}$, rAA(1) will essentially converge as if it were applied to the affine iteration function $\bm{q}( \bm{x} ) = M_* \bm{x} + \bm{b}_*$.
Based on the tests in the previous section, we therefore anticipate a worst-case convergence factor of $
\max_{i,j \in \{1, \ldots, n\}, i \neq j} 
\left( \frac{m_i m_j (m_j - m_i)}{|m_i (m_i - 1)| + |m_j (m_j - 1)|} \right)^{1/2}
$,
with $\{ m_i \}$ the eigenvalues of $M_*$.
We note that a conjecture similar in spirit relating the convergence of linear and nonlinear windowed AA and GMRES is given in \cite[Conjecture 5.1]{DeSterck-He-2021-stationary-AA}.
Furthermore, Garner et al. \cite{Garner-etal-2023} proved a result analogous to that which we suggest above, wherein they showed that their AA$(m)$ convergence factor bound for $\bm{q}(\bm{x}) = M \bm{x} + \bm{b}$ with $M$ a symmetric matrix extends to nonlinear functions $\bm{q}$ with Jacobian symmetric at the fixed point.

\begin{figure}[b!]
\centerline{
\includegraphics[width=0.475\textwidth]{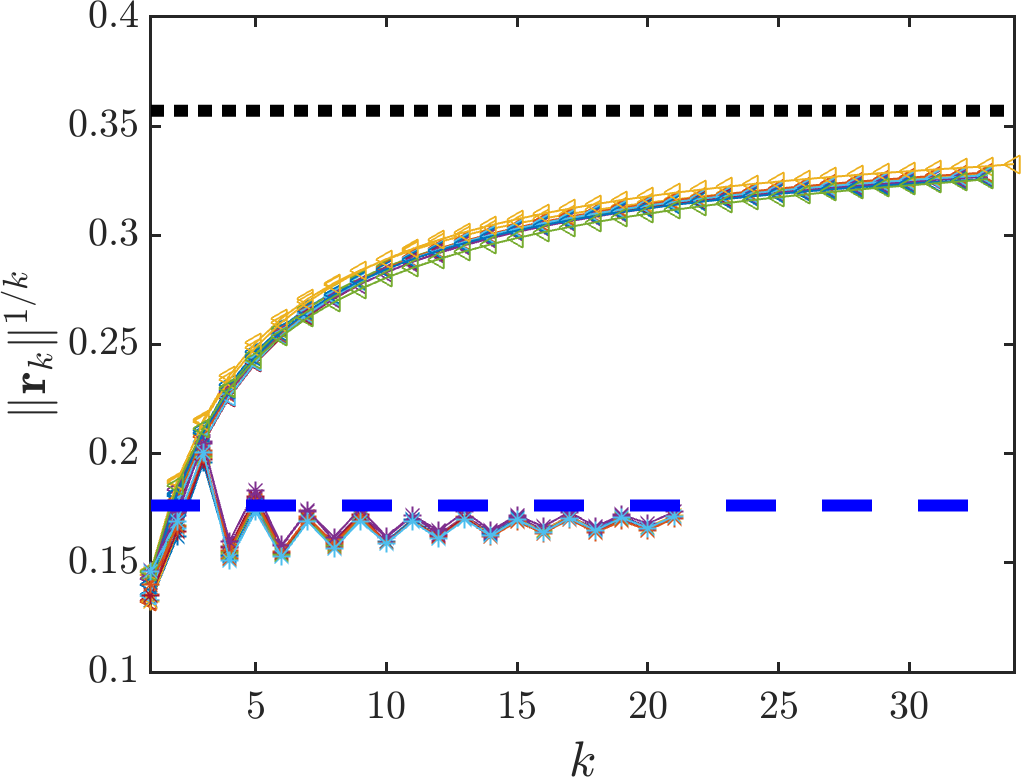}
\hspace{1ex}
\includegraphics[width=0.475\textwidth]{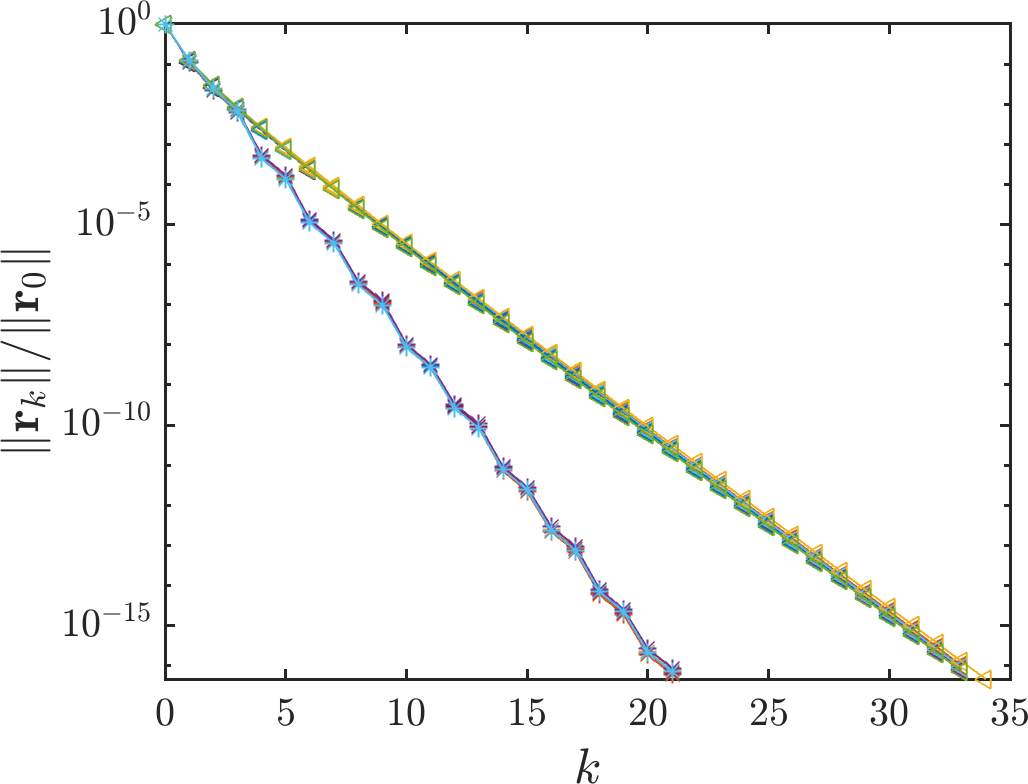}
}
\caption{Supporting numerical evidence for the extension of results from the linear to nonlinear setting by solving a nonlinear elliptic boundary value problem with an inexact Newton algorithm.
Triangle markers correspond to the underlying fixed-point iteration \eqref{eq:nonlin-q-MG}, and asterisk markers to the associated rAA(1) iteration.
\textbf{Left:} Numerically measured root-convergence factor.
The thick broken lines reflect the theoretically predicted worst-case convergence factor for each method.
\textbf{Right:} Relative reduction in residual as a function of iteration.
\label{fig:nonlin-MG}
}
\end{figure}

Given rAA(1)'s success at accelerating the two-grid algorithm in \cref{sec:symm-num-res} here we consider a nonlinear generalization of this problem.
Specifically, we solve the nonlinear elliptic boundary value problem $-(u_{xx} + u_{yy} ) + \gamma u \exp( u ) = f(x,y)$ for unknown $u$ subject to zero Dirichlet boundary conditions and function $f$ specified below.
To discretize this problem we use standard second-order accurate finite differences on a grid of $n_x - 1$ interior points in each direction equispaced by a distance $h$, resulting in a total of $n = (n_x - 1)^2$ unknowns.
Denoting the discrete solution vector as $\bm{u} \in \mathbb{R}^n$, the nonlinear algebraic system to be solved is written as $\bm{F}(\bm{u}) = \bm{0} = {\cal A}( \bm{u} ) - \bm{f}$, where $\bm{f}$ is the discretization of $f(x, y)$, and ${\cal A}$ that of the operator $-[\partial_{xx} (\, \cdot \,) + \partial_{yy} (\, \cdot \,) ] + \gamma (\, \cdot \,) \exp(\, \cdot \,)$.
To solve this algebraic system we employ an inexact Newton method, which, given some starting iterate $\bm{u}_0 \in \mathbb{R}^n$, takes the form 
\begin{align} \label{eq:nonlin-q-MG}
\bm{u}_{k+1} = \bm{q}( \bm{u}_k ) = \bm{u}_k - T(\bm{u}_k) \bm{F}( \bm{u}_k ), 
\quad
k = 0, 1, \ldots,
\end{align}
where $T(\bm{u}_k) \approx [\bm{F}'(\bm{u}_k)]^{-1} \in \mathbb{R}^{n \times n}$ is an approximation to the inverse of the Jacobian of $\bm{F}$ at the point $\bm{u}_k$.
Here we take $T(\bm{u}_k)$ as a single two-grid iteration applied to the Jacobian matrix $\bm{F}'(\bm{u}_k)$.
The details of the two-grid method are the same as those used for the linear problem in \cref{sec:symm-num-res} except that the coarse-grid matrix is taken as the Galerkin operator.
See \cite[Chapter 6]{Briggs2000} for further details.

Numerical results are shown in \cref{fig:nonlin-MG}, where we set ${\gamma = 5}$, and ${n_x = 32}$ so that ${n = 961}$.
The discretized right-hand side vector $\bm{f}$ is chosen to have random entries in $[-1, 1]$.
The underlying fixed-point iteration \eqref{eq:nonlin-q-MG} and the associated rAA(1) iteration are run for 10 different starting iterates $\bm{u}_0$, all of which are chosen to have random entries from the interval $[-0.1, 0.1]$.
From the left panel of \cref{fig:nonlin-MG}, the rAA(1) convergence factor does not appear to depend strongly on the initial iterate $\bm{u}_0$, similar to the linear two-grid problem pictured in \cref{fig:symm-rhok}.
Most significantly, we do see that the numerically measured root-convergence factors for both the underlying fixed-point and rAA(1) iterations trend towards their theoretically expected values of $\rho( \bm{q}'(\bm{u}_*) )$, and $
\max_{i,j \in \{1, \ldots, n\}, i \neq j} 
\left( \frac{m_i m_j (m_j - m_i)}{|m_i (m_i - 1)| + |m_j (m_j - 1)|} \right)^{1/2}
$ evaluated at eigenvalues of $\bm{q}'(\bm{u}_*)$, respectively.\footnote{This Jacobian is not symmetric due to non-symmetry of the two-grid matrix $T(\bm{u}_k)$. As for the two-grid iteration matrix considered in \cref{sec:symm-num-res}, this matrix is similar to a symmetric matrix.} Notice the convergence factors do not quite reach a constant before the halting tolerance is reached.
To evaluate these theoretical convergence factors we approximate the matrix $\bm{q}'( \bm{u}_* )$ using a basic finite difference, with the $\ell$th column being approximated as $\big[\bm{q}( \bm{u}_* + \mu \wt{\bm{e}}_{\ell} ) - \bm{q}(\bm{u}_*) \big]/\mu$ where $\mu = 10^{-6}$, and $\wt{\bm{e}}_{\ell}$ is the $\ell$th canonical basis vector.
The right panel of \cref{fig:nonlin-MG} emphasizes the impact of the acceleration provided by rAA(1).

\section{Convergence of rAA(1) for skew symmetric $M$}
\label{sec:skew}

In this section we analyze the residual iteration \eqref{eq:rAA1-two-step-iter-res} under the  assumption that the matrix $M \in \mathbb{R}^{n \times n} $ is skew symmetric: $M^\top = - M$. 
Preliminaries are discussed next in \cref{sec:skew-intro}, with convergence factors computed in \cref{sec:skew-conv}. Numerical results are given in \cref{sec:skew-num-res}, and dependence on the initial iterate is investigated in \cref{sec:skew-r0}.

\subsection{Preliminaries}
\label{sec:skew-intro}

For simplicity of our calculations, let us assume that $n$ is even throughout this section, recalling then that eigenvalues of an even-dimensional skew-symmetric matrix $M \in \mathbb{R}^{n \times n}$ are imaginary and occur in conjugate pairs.
Denote eigenvalues of $M$ as $\{ m_j, \bar{m}_j  \}_{j = 1}^{n/2} = \{ \i |m_j|, -\i |m_j| \}_{j = 1}^{n/2}$ with $\i$ the imaginary unit.
Suppose the eigenvalues are ordered increasingly by magnitude so that $|m_{\min}|:=|m_1| \leq |m_2| \leq \cdots \leq |m_{n/2}| := |m_{\max}| = \rho(M)$.
Let $M = U T U^\top$ be the real Schur decomposition of $M$ with $T \in \mathbb{R}^{n \times n} $ the block diagonal matrix\footnote{If $ U T U^\top$ is the real Schur decomposition for some general matrix $B$, then $T$ would be block upper triangular in general, and only if $B$ is normal would $T$ truncate to be block diagonal.}
\begin{align}
T
=
\begin{bmatrix}
T_1 \\ & \ddots \\ & & T_{n/2}
\end{bmatrix},
\quad
T_{j} 
= 
\begin{bmatrix}
0 & -|m_j| \\
|m_j| & 0
\end{bmatrix}.
\end{align}
The matrix $U \in \mathbb{R}^{n \times n}$ is orthogonal, and is blocked as
\begin{align} \label{eq:Uj-def}
U = [ U_1, \ldots, U_{n/2} ], 
\quad
U_j = [ \bm{u}_j, \, \wt{\bm{u}}_j ] \in \mathbb{R}^{n \times 2}.
\end{align}
Note that $M$ is invariant on the subspaces $\mathrm{span} \{ U_j \}$, $j = 1, \ldots, n/2$, and in particular,
\begin{align} \label{eq:skew-M-invar}
M \big[ \bm{u}_j, \, \wt{\bm{u}}_j \big]  
= 
|m_j| \big[ \wt{\bm{u}}_j , \,  -\bm{u}_j,  \big],
\quad
\mathrm{and}
\quad
M^2 \big[ \bm{u}_j, \, \wt{\bm{u}}_j \big]  
= 
-|m_j|^2 \big[ \bm{u}_j, \, \wt{\bm{u}}_j \big].
\end{align}

From \eqref{eq:alpha-def} recall the coefficient 
\begin{align} \label{eq:alpha-skew}
\alpha( \bm{z} )  
=
\frac{\la A\bm{z}, \bm{z} \ra}{\la A\bm{z}, A\bm{z} \ra }
=
\frac{\la (A\bm{z}), A^{-1} (A\bm{z}) \ra}{\la (A\bm{z}), (A\bm{z}) \ra }
=
\frac{\la (A\bm{z}), Y (A\bm{z}) \ra}{\la (A\bm{z}), (A\bm{z}) \ra },
\end{align}
$Y := \tfrac{1}{2}( A^{-1} + A^{-\top} )$.
Observe $(I - M)(I + M) Y 
= 
\tfrac{1}{2} (I - M)(I + M)\big[ (I - M)^{-1} + (I + M)^{-1} \big] 
=
\tfrac{1}{2} \big[ (I + M) + (I - M) \big] 
= I$.
Hence, $Y = [(I - M)(I + M)]^{-1} = (I - M^2)^{-1} = (I + M M^{\top})^{-1}$.
Since $M M^{\top}$ is symmetric positive semi-definite, $Y$ is SPD, with eigenvalues $\{( 1 + |m_j|^2 )^{-1} \}_{j = 1}^{n/2}$ each of algebraic multiplicity two.
Since $Y$ and $Y^{-1}$ share eigenvectors, the eigenvectors of $Y$ are those of $M^2$. Hence, from \eqref{eq:skew-M-invar} observe that any $\bm{z} \in \mathrm{span}( U_{j} )$ is an eigenvector of $M^2$, and thus of $Y$ with corresponding eigenvalue $( 1 + |m_j|^2 )^{-1}$.

\begin{lemma} \label{lem:skew-alpha-properties}

Suppose $M$ is skew symmetric, $\bm{0} \neq \bm{z} \in \mathrm{span} \{ U_j \}$, with $U_j$ as in \eqref{eq:Uj-def}, and that $\alpha$ is as in \eqref{eq:alpha-skew}.
Then,
\begin{align} \label{eq:skew-alpha-special}
\alpha( \bm{z} ) = \frac{1}{1 + |m_j|^2}.
\end{align}
Otherwise, for $\bm{0} \neq \bm{v} \in \mathbb{R}^{n}$ arbitrary we have
\begin{align} \label{eq:skew-alpha-interval}
\alpha( \bm{v} ) 
\in 
\left[ \frac{1}{1 + |m_{\max}|^2}, 
\frac{1}{1 + |m_{\min}|^2}
\right]
\subseteq (0, 1].
\end{align}

\end{lemma}

\begin{proof}
From the right-most expression in \eqref{eq:alpha-skew} if $A \bm{z}$ is an eigenvector of $Y$, then $\alpha( \bm{z} )$ is the corresponding eigenvalue of $Y$.
If $\bm{z} \in \mathrm{span} \{ U_j \}$, then $A \bm{z} \in \mathrm{span} \{ U_j \}$ since $A = I - M$, and, hence, by the above discussion regarding eigenvectors and eigenvalues of $Y$, $\alpha( \bm{z} ) = (1 + |m_j|^2)^{-1}$.
Consider \eqref{eq:skew-alpha-interval}: Since $A$ is full rank, the range of the Rayleigh quotient between $Y$ and $A \bm{v}$ is $[\lambda_{\min}( Y ), \lambda_{\max}( Y ) ]$, and it was just shown that the eigenvalues of $Y$ are $\{( 1 + |m_j|^2 )^{-1} \}_{j = 1}^{n/2}$.
\end{proof}

\subsection{Convergence factor}
\label{sec:skew-conv}

In \cref{cor:skew-worst-case} below we establish the worst-case convergence factor, but first we require the results in \cref{lem:skew-special-case,lem:skew-upper-bound}.

\begin{lemma}
\label{lem:skew-special-case}
Suppose $M$ is skew symmetric, that $\bm{0} \neq \bm{z} \in \mathrm{span} \{ U_j \} $ with $U_j$ as in \eqref{eq:Uj-def}, and that ${\cal R}$ is the residual propagator in \eqref{eq:calR-def}.
Then, 
\begin{align}
{\cal R}( \bm{z} ) \in \mathrm{span} \{ U_j \},
\quad
\textrm{and}
\quad
\frac{\Vert {\cal R}( \bm{z} ) \Vert}{\Vert \bm{z} \Vert}
= 
\frac{|m_j|^2}{\sqrt{1 + |m_j|^2}}.
\end{align}

\end{lemma}

\begin{proof}
Write $\alpha \equiv \alpha( \bm{z} )$. We have,
${\cal R} ( \bm{z} ) = M( I - \alpha A ) \bm{z} = (1 - \alpha) M \bm{z} + \alpha M^2 \bm{z} $.
Since $U_j$ is an invariant subspace of $M$, it is of the polynomial $(1 - \alpha)M + \alpha M^2$, and, thus, ${\cal R}( \bm{z} ) \in \mathrm{span} \{ U_j \}$. \\
Now consider $\Vert {\cal R} ( \bm{z} ) \Vert^2$. 
Recalling from \eqref{eq:skew-M-invar} that $M^2 \bm{z} = - |m_j|^2 \bm{z}$ for any $\bm{z} \in \mathrm{span}(U_j)$, we have
\begin{align}
\Vert {\cal R} ( \bm{z} ) \Vert^2 
&= 
\la (1 - \alpha) M \bm{z} - \alpha |m_j|^2 \bm{z}, (1 - \alpha) M \bm{z} - \alpha |m_j|^2 \bm{z} \ra,
\\
&= 
(1 - \alpha)^2 \la M\bm{z}, M\bm{z} \ra + \alpha^2 |m_j|^4 \la \bm{z}, \bm{z} \ra,
\end{align}
with the second equality following since $\bm{z}^\top M \bm{z} = 0$ for any vector $\bm{z}$.
Simplifying we find $ \Vert {\cal R} ( \bm{z} ) \Vert^2 = |m_j|^2 \big[ \alpha^2 |m_j|^2 + (1 - \alpha)^2 \big] \Vert \bm{z} \Vert^2$. 
The claim then follows by substituting in $\alpha = 1/( 1 + |m_j|^2 )$, as per \eqref{eq:skew-alpha-special}, and simplifying.
\end{proof}

\begin{lemma}
\label{lem:skew-upper-bound}

Suppose $M$ is skew symmetric, and that ${\cal R}$ is the residual propagator in \eqref{eq:calR-def}.
Then,
\begin{align} \label{eq:skew-upper-bound}
\max \limits_{  \bm{v} \neq \bm{0} }
\frac{\Vert {\cal R}( \bm{v} ) \Vert}{\Vert \bm{v} \Vert}
\leq
\frac{|m_{\max}|^2}{\sqrt{1 + |m_{\max}|^2}}.
\end{align}

\end{lemma}

\begin{proof}
For any $\bm{v} \neq \bm{0}$ we have
\begin{align}
\Vert {\cal R}( \bm{v} ) \Vert^2
=
\Vert M (I - \alpha( \bm{v} ) A ) \bm{v} \Vert^2
\leq
|m_{\max}|^2 \Vert (I - \alpha( \bm{v} ) A ) \bm{v} \Vert^2.
\end{align}
Since $ A \bm{v} \perp (I - \alpha( \bm{v} ) A ) \bm{v}$ we have
\begin{align}
\Vert (I - \alpha( \bm{v} ) A ) \bm{v} \Vert^2
=
\la 
\bm{v},
(I - \alpha( \bm{v} ) A ) \bm{v}
\ra
=
\Vert \bm{v} \Vert^2 - \alpha( \bm{v} ) \la \bm{v}, A \bm{v} \ra
=
\Vert \bm{v} \Vert^2 ( 1 - \alpha(\bm{v}) )
\end{align}
noting that the symmetric part of $A = I - M$ is $I$.
From \cref{lem:skew-alpha-properties} recall that $\alpha( \bm{v} ) \in (0, 1]$ for all $\bm{v} \neq \bm{0}$.
Hence, for any $\bm{v} \neq \bm{0}$ we have
\begin{align}
\frac{\Vert {\cal R}( \bm{v} ) \Vert^2}{\Vert \bm{v} \Vert^2}
\leq
| m_{\max} |^2 \Big( 1 - \min_{\bm{w} \neq \bm{0}} \alpha(\bm{w}) \Big).
\end{align}
Plugging in $\min_{\bm{w}} \alpha(\bm{w}) = \frac{1}{1 + |m_{\max}|^2}$ from \cref{lem:skew-alpha-properties} and simplifying gives the claim.
\end{proof}

\begin{corollary}
Let $\varrho^{\rm worst}_{\rm FP}$ be the worst-case root-average convergence factor (see \cref{def:conv-fac}) of the fixed-point iteration \eqref{eq:PI-iter-res} applied to a skew symmetric matrix $M$ with spectral radius $\rho(M)$.
Let $\varrho^{\rm worst}_{\rm AA}$ be the worst-case root-average convergence factor of the associated rAA(1) iteration \eqref{eq:rAA1-two-step-iter-res}.
Then,
\label{cor:skew-worst-case}
\begin{align} \label{eq:rho-skew-wc}
\varrho^{\rm worst}_{\rm AA}
=
\frac{\rho(M)}{\big[ 1 + \rho^2(M) \big]^{1/4}}
< 
\rho(M) 
= 
\varrho^{\rm worst}_{\rm FP}.
\end{align}
As such, the rAA(1) iteration converges for arbitrary $\bm{r}_0  \in \mathbb{R}^n$ if and only if 
\begin{align} \label{eq:skew-mmax-limit}
\rho(M) 
<
\sqrt{\tfrac{1}{2}(1 + \sqrt{5})} \approx 1.272.
\end{align}

\end{corollary}
See the left panel of \cref{fig:rho-wc-skew} for a plot comparing $\varrho^{\rm worst}_{\rm AA}$ and $\varrho^{\rm worst}_{\rm FP}$. 
In the right panel of \cref{fig:rho-wc-skew} we show the number of iterations $k_*$ required to reduce the error from its initial value by a factor $10^{-\nu_*}$ assuming a root-average convergence factor equal to $\varrho^{\rm worst}$, i.e., the $k_*$ such that $\Vert \bm{e}_{k_*} \Vert / \Vert \bm{e}_0 \Vert = \big( \varrho^{\rm worst} \big)^{k_*} = 10^{-\nu_*}$.
Considering the right plot, the reduction in iterations from using rAA(1) over the underlying fixed-point iteration is arguably negligible for small to moderate values of $\rho(M)$, e.g., $\rho(M) \lesssim 0.5$.

\begin{figure}[h!]
\centerline{
\includegraphics[width = 0.475\textwidth]{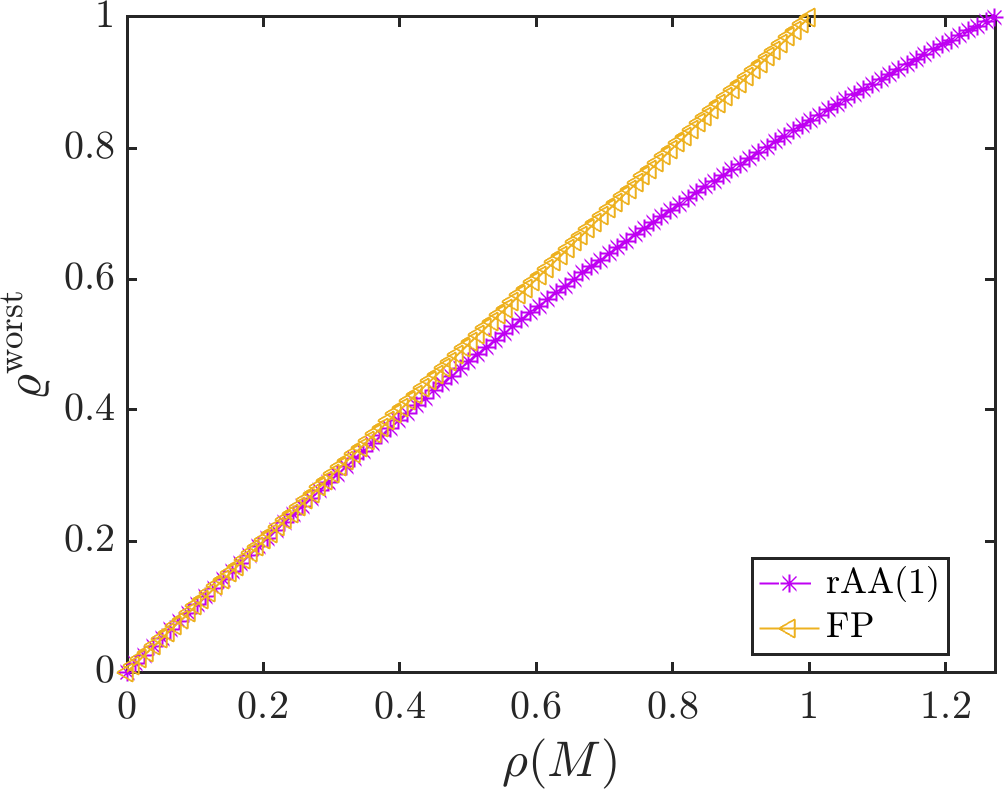}
\hspace{1ex}
\includegraphics[width = 0.47\textwidth]{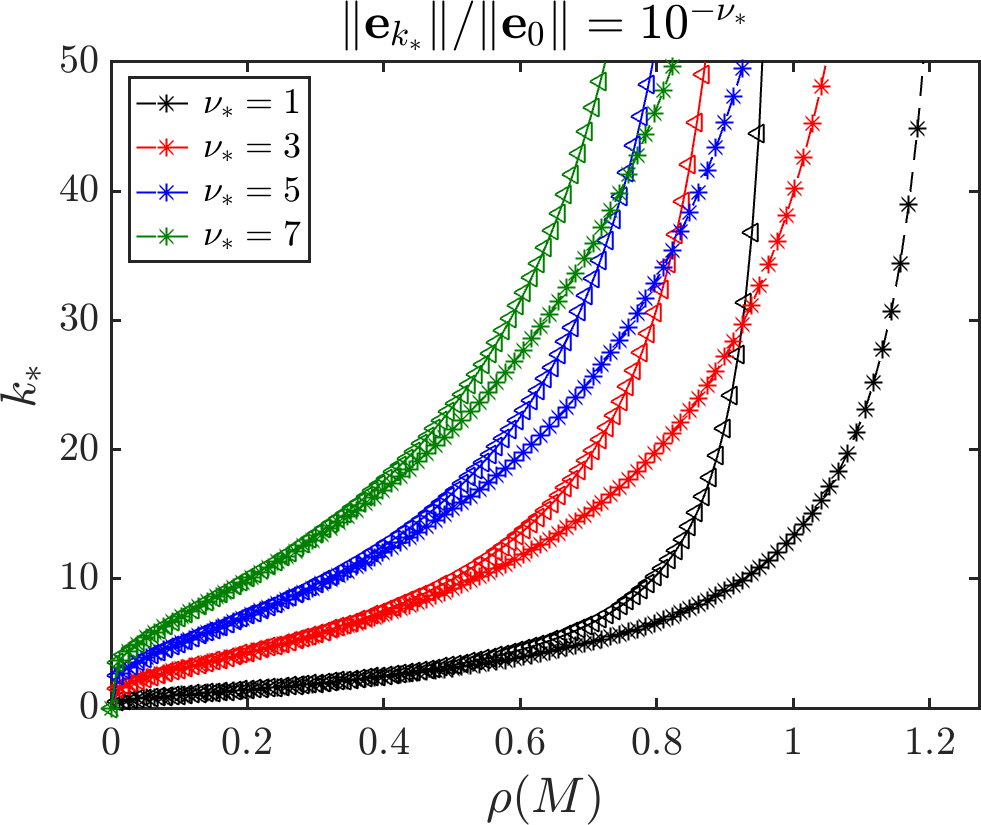}
}
\caption{For skew-symmetric matrices $M$ with spectral radius $\rho(M)$.
\textbf{Left:} Worst-case root-linear convergence factor, as per \cref{cor:skew-worst-case}.
\textbf{Right:} For a given $\rho(M)$, the number of iterations $k_*$ required to reduce the error by a factor of $10^{-\nu_*}$, with values of $\nu_*$ indicated in the legend. 
Asterisks markers represent AA, and triangle markers the underlying fixed-point iteration.
\label{fig:rho-wc-skew}
}
\end{figure}

\begin{proof}
Recall that $\varrho^{\rm worst}_{\rm FP} = \rho(M)$ is shown in \eqref{eq:rho-PI-wc}.
Otherwise, to prove $\varrho^{\rm worst}_{\rm AA}$ in \eqref{eq:rho-skew-wc} we show that there are vectors for which the upper bound in \eqref{eq:skew-upper-bound} achieves equality.
To this end, let us consider an initial rAA(1) residual that satisfies $\bm{r}_{0} \in \mathrm{span} \{ U_{n/2} \} \neq \bm{0}$. 
Then, from \cref{lem:skew-special-case} the corresponding subsequent residuals must also be in this subspace: $\bm{r}_{2\ell + 2} = {\cal R}( \bm{r}_{2\ell} ) \in \mathrm{span} \{ U_{n/2} \}$ for all $\ell = 0, 1, 2, \ldots$.
By combining \cref{lem:skew-special-case,lem:skew-upper-bound} and recalling that $|m_{n/2}| = |m_{\max}|$, it therefore holds for any $\bm{r}_{0} \in \mathrm{span} \{ U_{n/2} \} \neq \bm{0}$ that
\begin{align} \label{eq:skew-worst-case-quotient-aux}
\max \limits_{  \bm{v} \neq \bm{0} }
\frac{\Vert {\cal R}( \bm{v} ) \Vert}{\Vert \bm{v} \Vert}
\leq
\frac{|m_{\max}|^2}{\sqrt{1 + |m_{\max}|^2}}
=
\frac{|m_{n/2}|^2}{\sqrt{1 + |m_{n/2}|^2}}
=
\frac{\Vert {\cal R} (\bm{r}_{2 \ell}) \Vert}{\Vert \bm{r}_{2 \ell} \Vert}
\quad
\ell = 0, 1, 2, \ldots, 
\end{align}
That is, the quantity $\frac{\Vert {\cal R}( \bm{v} ) \Vert}{\Vert \bm{v} \Vert}$ is maximized whenever $\bm{v} \in \mathrm{span} \{ U_{n/2} \} \neq \bm{0}$.

Now let us consider arbitrary initial rAA(1) residuals $\bm{r}_0 \neq \bm{0}$, i.e. not just those in $\mathrm{span} \{ U_{n/2} \}$.
Recall from \cref{def:conv-fac} $\varrho( \bm{r}_0 ) =\limsup_{k \to \infty} \varrho_k( \bm{r}_0 )$, where $\varrho_k(\bm{r}_0) = \Vert \bm{r}_k \Vert^{1/k}$, and that $\varrho^{\rm worst} := \max_{\bm{r}_0 \neq \bm{0}} \varrho( \bm{r}_0 )$.
To simplify calculations suppose $k$ is even, then 
\begin{align*}
\max_{\bm{r}_0 \neq \bm{0}}
\frac{
\varrho_k( \bm{r}_0 )
}{\Vert \bm{r}_0 \Vert^{1/k}}
&=
\max_{\bm{r}_0 \neq \bm{0}}
\left(
\prod \limits_{\ell = 0}^{k/2-1} \frac{\Vert \bm{r}_{2\ell+2} \Vert}{\Vert \bm{r}_{2\ell} \Vert}
\right)^{1/k}
=
\left(
\prod \limits_{\ell = 0}^{k/2-1} 
\max_{\bm{r}_{2 \ell} \neq \bm{0}}
\frac{\Vert {\cal R}(\bm{r}_{2\ell}) \Vert}{\Vert \bm{r}_{2\ell} \Vert}
\right)^{1/k}
\\
&=
\left(
\prod \limits_{\ell = 0}^{k/2-1} 
\frac{|m_{\max}|^2}{\sqrt{1 + |m_{\max}|^2}}
\right)^{1/k}
=
\frac{|m_{\max}|}{(1 + |m_{\max}|^2)^{1/4}}.
\end{align*}
Plugging into the definition of $\varrho^{\rm worst}$ gives the result claimed result for $\varrho^{\rm worst}_{\rm AA}$.
Solving for $\varrho^{\rm worst}_{\rm AA} < 1$ yields \eqref{eq:skew-mmax-limit}.
\end{proof}
%

\subsection{Numerical results}
\label{sec:skew-num-res}

Here we provide numerical results in support of the theory from the previous section. 
Our test problems arise in the context of numerically solving the linear conservation law $u_t = [\alpha(x) u]_x$ for a given wave-speed function $\alpha$, and with unknown function $u$ subject to periodic boundary conditions in space, and an initial condition in time.
Consider discretizing this PDE in space using a skew-symmetric derivative operator $\mathrm{D}  \in \mathbb{R}^{n \times n}$; that is, for a periodic, differentiable function $a$, $(\mathrm{D}  \bm{a})_i \approx a'(x_i)$ where $\bm{a} = [a(x_1), \ldots, a(x_n)]^\top$, and $x_i$ are, e.g., mesh or collocation points. 
The PDE then becomes the system of ordinary differential equations $\bm{u}' = \mathrm{D}  \Lambda \bm{u}$, $\Lambda = \diag_i(\alpha(x_i))$, which we discretize with the Crank-Nicolson method. To advance the solution from time $t_n$ to $t_{n+1}$ we must then solve the linear system $(I - \tfrac{1}{2} \delta t \mathrm{D}  \Lambda) \bm{u}^{n+1} = (I + \tfrac{1}{2} \delta t \mathrm{D}  \Lambda) \bm{u}^{n}$.
For readability, let us write $\bm{u} := \bm{u}^{n+1}$, and $\bm{f}:= (I + \tfrac{1}{2} \delta t \mathrm{D}  \Lambda) \bm{u}^{n}$.
Then, to solve the linear system $(I - \tfrac{1}{2} \delta t \mathrm{D}  \Lambda) \bm{u} = \bm{f}$ we employ the following simple fixed-point iteration:
\begin{align} \label{eq:skew-PI}
\bm{u}_{k+1} = \frac{1}{2} \delta t \mathrm{D} \Lambda \bm{u}_{k} + \bm{f}, \quad k = 0, 1, \ldots
\end{align}
Notice that this iteration is exactly of the form of \eqref{eq:PI-iter}, with solution $\bm{x} = \bm{u}$, forcing term $\bm{b} = \bm{f}$ and skew-symmetric matrix $M = \tfrac{1}{2} \delta t \mathrm{D} \Lambda$.\footnote{Analogously to the two-grid example in \cref{sec:symm-num-res}, $M$ is not technically skew symmetric because $\mathrm{D}$ and $\Lambda$ do not commute unless $\alpha(x) \equiv $ constant.
However, since $\alpha(x) > 0$, $\Lambda$ is SPD, and we have that $M$ is similar to the skew-symmetric matrix $\wt{M} = \Lambda^{1/2} M \Lambda^{-1/2} = \tfrac{1}{2} \delta t \Lambda^{1/2} \mathrm{D}  \Lambda^{1/2}$. See, e.g., \cite[Example 13.17]{Hesthaven2017} for related ideas.}
We consider two skew-symmetric differentiation matrices $\mathrm{D}$. The first uses a basic second-order central finite difference, and the second uses the Fourier collocation spectral method, see e.g., \cite[Chapter 13.2]{Hesthaven2017}.
These two tests use $\delta t = 1.7 h$ and $\delta t = 0.85 h$, for $h$ the spacing between finite-difference nodes and collocation points, respectively.
The wave-speed is $\alpha(x) = \tfrac{1}{2} [1 + \cos^2(x)]$ and the spatial domain $x \in (0, 2 \pi)$.

\begin{figure}[t!]
\centerline{
\includegraphics[width=0.475\textwidth]{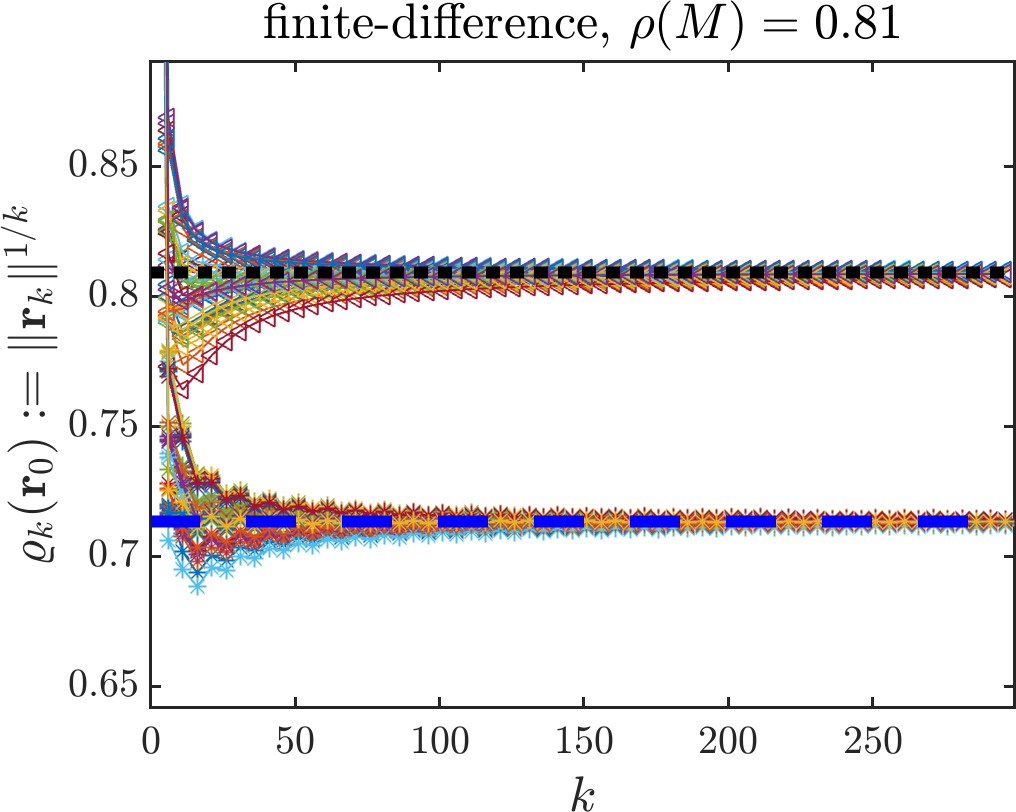}
\hspace{2ex}
\includegraphics[width=0.475\textwidth]{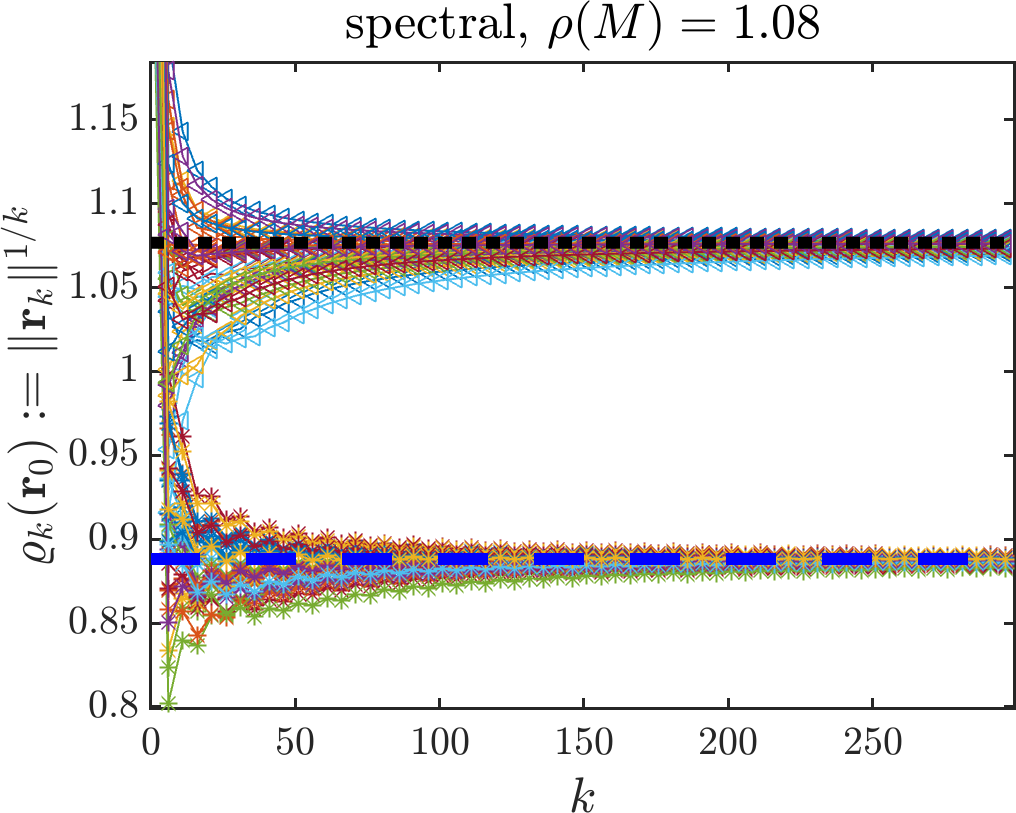}
}
\caption{
For two matrices $M = \tfrac{1}{2} \delta t \mathrm{D}  \Lambda$ similar to skew symmetric matrices, shown is the convergence factor $\varrho_k(\mathbf{r}_0)$ as a function of the iteration index $k$ that limits to the asymptotic convergence factor as ${k \to \infty}$ (see \cref{def:conv-fac}).
Triangle markers depict the underlying fixed-point iteration \eqref{eq:PI-iter}, and asterisk markers the rAA(1) iteration \eqref{eq:rAA1-iter}.
For both problems, each algorithm is initialized with 30 different $\bm{r}_0$ chosen at random.
\textbf{Left:} $\mathrm{D} $ is a 2nd-order accurate finite-difference discretization using $n = 64$ points.
\textbf{Right:} $\mathrm{D} $ is a Fourier spectral differentiation matrix with $n = 31$ collocation points.
The worst-case asymptotic convergence factor \eqref{eq:rho-PI-wc} for the underlying fixed-point iteration is shown in each plot as a thick, black dotted line, and that of \eqref{eq:rho-skew-wc} for the rAA(1) iteration is shown as the thick, blue dashed line.
\label{fig:skew-advection}
}
\end{figure}

In \cref{fig:skew-advection} we plot the convergence factor for the underlying fixed-point iteration \eqref{eq:skew-PI} and that for the associated rAA(1) iteration. We consider 30 different $\bm{u}_0$ vectors, all with random entries drawn from $[-1, 1]$, and we take {$\bm{f} := (I + \tfrac{1}{2} \delta t \mathrm{D}  \Lambda) \bm{u}^{n} = \bm{0}$} in all cases, corresponding to $\bm{u}^n = \bm{0}$.
In all cases we see that the rAA(1) convergence factor approaches the worst-case asymptotic convergence factor predicted by \eqref{eq:rho-skew-wc} of \cref{cor:skew-worst-case}.
Notice that the right-hand panel in \cref{fig:skew-advection} provides an example where the underlying fixed-point iteration does not converge (the convergence factor is larger than unity), while rAA(1) does.
We note that some other examples in literature have shown that AA may converge on problems where the underlying (nonlinear) iteration function $\bm{q}(\bm{x})$ is not a contraction, at least not globally \cite{Lott-etal-2012-noncontractive,Pollock-Rebholz-2021-noncontractive}.

Interestingly, the asymptotic convergence factor of rAA(1) does not appear to depend on $\bm{r}_0$.
Additional numerical tests (see Supplementary Materials Section \ref{SMsec:skew-additional}) for other skew symmetric $M$ indicate that this is a general trend and is not specific to the $M$ considered in \cref{fig:skew-advection}. This is despite the fact that when applied to symmetric matrices $M$, the rAA(1) convergence factor \textit{may} strongly depend on $\bm{r}_0$ (see \cref{sec:symm-num-res}).
We provide a theoretical explanation for this observation in the next section.
%

\subsection{Independence of convergence factor on initial iterate}
\label{sec:skew-r0}

What (the proof of) \cref{lem:skew-rho0-r0} below shows is that residual iteration \eqref{eq:rAA1-two-step-iter-res} for $M$ skew symmetric behaves analogously to a non-normalized power iteration for the dominant subspace $\mathrm{span}(U_j)$ of $M$.
Thus, the rAA(1) iteration behaves analogously to the underlying fixed-point iteration \eqref{eq:PI-iter-res} which is the non-normalized power iteration $\bm{r}_k = M^k \bm{r}_0$.
This explains the observation from the previous section wherein the asymptotic convergence factor did not appear to depend on $\bm{r}_0$.
Note that the proof of \cref{lem:skew-rho0-r0} requires the following auxiliary result.

\begin{lemma} \label{lem:skew-g-function}
Suppose $M$ is skew symmetric.
Let $\alpha$ be as in \eqref{eq:alpha-def}, $\bm{z}_j \in \mathrm{span}(U_j)$ with $U_j$ as in \eqref{eq:Uj-def}, and $\bm{0} \neq \bm{v} \in \mathbb{R}^n$ be arbitrary.
Then, 
\begin{align}
\Vert M[I - \alpha( \bm{v} ) A] \bm{z}_j \Vert
&=
g(|m_j|, \bm{v}) 
\Vert \bm{z}_j \Vert,
\end{align}
where
$
g(|m_j|, \bm{v})
:=
|m_j|
\sqrt{
\big[ |m_j|\alpha(\bm{v}) \big]^2 + \big[ 1 -  \alpha( \bm{v} ) \big]^2
}
$
is strictly increasing with $|m_j|$.
\end{lemma}
 
\begin{proof}
Write $\alpha \equiv \alpha( \bm{v} )$. Then,
$
\Vert M[I - \alpha A] \bm{z}_j \Vert^2 
= 
\Vert M [ (1 - \alpha ) + \alpha M ] \bm{z}_j \Vert^2 
= 
\Vert (1 - \alpha ) M \bm{z}_j  - \alpha |m_j|^2 \bm{z}_j \Vert^2
$.
Recalling $\la \bm{z}_j, M \bm{z}_j \ra = 0$ (since the symmetric part of $M$ is the zero matrix) and using \eqref{eq:skew-M-invar} gives
$\Vert M[I - \alpha A] \bm{z}_j \Vert^2 
=
\la (1 - \alpha ) M \bm{z}_j, (1 - \alpha ) M \bm{z}_j \ra
+
\la \alpha |m_j|^2 \bm{z}_j, \alpha |m_j|^2 \bm{z}_j \ra
 =
|m_j|^2 
\big[ 
(1 - \alpha)^2
+
\alpha^2 |m_j|^2
\big]
\Vert \bm{z}_j \Vert^2
= 
\big[ g(|m_j|, \bm{v}) \big]^2 \Vert \bm{z}_j \Vert^2$.
\end{proof}

\begin{lemma}
\label{lem:skew-rho0-r0}
Consider applying rAA(1) to skew symmetric matrix $M$.
Suppose $\bm{r}_0 = \sum_{j = 1}^{j^*} \bm{z}_j \neq \bm{0}$, with $\bm{z}_j \in \mathrm{span}( U_j )$ and $U_j$ as in \eqref{eq:Uj-def}, and $ 1 \leq j^* \leq n/2$ the largest $j$ for which $\Vert \bm{z}_j \Vert \neq 0$.
Assume the eigenvalues $\pm \i|m_{j^*}|$ of $M$ have algebraic multiplicity one.
Then, the rAA(1) residual $ \bm{r}_{k}$ satisfies
\begin{align}
\lim \limits_{k \to \infty} \bm{r}_{k} 
\in 
\mathrm{span}( U_{j^*} ).
\end{align}
\end{lemma}
Before proving this lemma, let us use it to re-examine our discussion from \cref{sec:con-and-R-dif} on the continuity and differentiability ${\cal R}$. 
From \cref{lem:R-cont-dif}, recall 
\begin{align} \label{eq:DR-copy}
\mathfrak{D} {\cal R}(\bm{v}, \bm{d})
=
M [I - \alpha( \bm{v} ) A] \bm{d}
- 
\la \alpha'( \bm{v} ), \bm{d} \ra  M A \bm{v},
\end{align}
with $\alpha'$ given in \eqref{eq:alpha-prime}, is the directional derivative of ${\cal R}$ at the point $\bm{v} \neq \bm{0}$ in the direction $\bm{d}$.
Recall further from \cref{lem:R-cont-dif} that ${\cal R}$ is differentiable at $\bm{v}$ iff $A \bm{v}$ is an eigenvector of $Y$, and from \cref{sec:skew-intro} that any $\bm{z} \in \mathrm{span}( U_{j} )$ is an eigenvector of $Y$.
Observe that if $\bm{z} \in \mathrm{span}( U_{j} )$, then $A \bm{z} \in \mathrm{span}( U_{j} )$ since $A = I - M$.
Hence, when $M^\top =- M$, ${\cal R}$ is differentiable at any point $\bm{z} \in \mathrm{span}( U_{j} )$ with Jacobian ${\cal R}'( \bm{z} ) = M [I - \alpha( \bm{z} ) A]$.
From \cref{lem:skew-rho0-r0} recall that the rAA(1) residual vector satisfies $\lim_{k \to \infty} \bm{r}_k \in \mathrm{span}( U_{j^*} )$ for some $j^*$.
As such we see that, while the rAA(1) residual propagator ${\cal R}( \bm{r}_k )$ is not (necessarily) differentiable for $k \in \mathbb{N}_0$, it is differentiable as $k \to \infty$ since the non-differentiable term in \eqref{eq:DR-copy} vanishes along the path taken by the method as $\bm{r}_k \to \bm{0}$.
Moreover, the (square root of the) spectral radius of the Jacobian matrix $\lim_{k \to \infty}{\cal R}'( \bm{r}_k ) = M [I - \alpha( \mathrm{span}( U_{n/2} ) ) A]$ is exactly the worst-case asymptotic convergence factor from \cref{cor:skew-worst-case}, as is consistent with \cref{thm:ostrowski}.

\begin{proof}
The arguments here work along the lines of those for the power method for computing the dominant eigenpair of a diagonalizable matrix, see e.g., \cite[Section 4.4.1]{Demmel1997}.
The first step is to develop a recursive expression for the $k$th residual vector $\bm{r}_k$.
To this end, let us change notation slightly from that used in the body of the lemma to write $\bm{r}_0$ as
$
\bm{r}_0 = \Vert \bm{z}^0_{j^*} \Vert \sum_{j = 1}^{j^*} \frac{\bm{z}^0_j}{\Vert \bm{z}^0_{j^*} \Vert}
$
with $\bm{z}_j^0 = \bm{z}_j$.
The second residual may now be written
\begin{align}
\begin{split}
\bm{r}_2 
&=
M[ I - \alpha( \bm{r}_0 ) A ] 
\left(
\Vert \bm{z}^0_{j^*} \Vert
\sum_{j = 1}^{j^*}
\frac{\bm{z}^0_j}{\Vert \bm{z}^0_{j^*} \Vert}
\right)
=
\Vert \bm{z}^0_{j^*} \Vert
\sum_{j = 1}^{j^*}
\underbrace{
\left(  
M[ I - \alpha( \bm{r}_0 ) A ] \frac{\bm{z}^0_j}{\Vert \bm{z}^0_{j^*} \Vert}
\right)}_{=: \bm{z}^2_j}
\\
&=
\Vert \bm{z}^0_{j^*} \Vert
\sum_{j = 1}^{j^*} \bm{z}^2_j
=
\Vert \bm{z}^0_{j^*} \Vert
\Vert \bm{z}^2_{j^*} \Vert 
\sum_{j = 1}^{j^*}  \frac{\bm{z}^2_j}{\Vert \bm{z}^2_{j^*} \Vert}.
\end{split}
\end{align}
Note that $\bm{z}^2_j \in \mathrm{span}(U_j)$ since $M[ I - \alpha( \bm{r}_0 ) A ]$ is a polynomial in $M$ so it is invariant on the subspace $\mathrm{span}(U_j)$.
Notice that $\bm{r}_2$ has the same structure as $\bm{r}_0$, in the sense that it is a linear combination of vectors from the subspaces $\mathrm{span}(U_j)$, $j = 1, \ldots, j^*$. 
Thus, by an inductive argument we may express $\bm{r}_k$, $k = 2, 4, 6, \ldots$, as
\begin{align} \label{eq:rk-zk-sum}
\bm{r}_k
&=
\zeta_k
\sum_{j = 1}^{j^*}  \frac{\bm{z}^k_j}{\Vert \bm{z}^k_{j^*} \Vert}, 
\quad
\mathrm{where}
\quad
\bm{z}^k_j 
&:= 
M[ I - \alpha( \bm{r}_{k-2} ) A ] \frac{\bm{z}^{k-2}_j}{\Vert \bm{z}^{k-2}_{j^*} \Vert} \in \mathrm{span}(U_j),
\end{align}
and $\zeta_k
:=
\prod_{\ell = 0, 2, \ldots}^{k} \Vert \bm{z}_{j^*}^{\ell} \Vert.$
\\ Next we analyze the size of the addends $\frac{\bm{z}^k_j}{\Vert \bm{z}^k_{j^*} \Vert}$ of $\bm{r}_k$ in \eqref{eq:rk-zk-sum}.
Plugging in the definition of $\bm{z}^k_j$ gives, for $k \geq 2$,
\begin{align}
\frac{\Vert \bm{z}^k_j \Vert}{\Vert \bm{z}^k_{j^*} \Vert}
=
\frac{
\frac{\Vert M[I - \alpha( \bm{r}_{k-2} ) A] \bm{z}^{k-2}_j \Vert}{\Vert \bm{z}^{k-2}_{j^*} \Vert}
}{
\frac{\Vert M[I - \alpha( \bm{r}_{k-2} ) A] \bm{z}^{k-2}_{j^*} \Vert}{\Vert \bm{z}^{k-2}_{j^*} \Vert}
}
=
\frac{
\Vert M[I - \alpha( \bm{r}_{k-2} ) A] \bm{z}^{k-2}_j \Vert
}{
\Vert M[I - \alpha( \bm{r}_{k-2} ) A] \bm{z}^{k-2}_{j^*} \Vert
}
=
c_{j, k-2}  
\frac{\Vert \bm{z}^{k-2}_j \Vert}{\Vert \bm{z}^{k-2}_{j^*} \Vert},
\end{align}
with $c_{j, k-2} := g(|m_j|, \bm{r}_{k-2}) / g(|m_{j^*}|, \bm{r}_{k-2})$, and $g$ defined in \cref{lem:skew-g-function}.
Noting that
$\frac{\Vert \bm{z}^k_j \Vert}{\Vert \bm{z}^k_{j^*} \Vert} 
=
c_{j, k-2}  
\frac{\Vert \bm{z}^{k-2}_j \Vert}{\Vert \bm{z}^{k-2}_{j^*} \Vert}$ holds independent of $k$, we can recursively apply this identity to compute $\frac{\Vert \bm{z}^{k-2}_j \Vert}{\Vert \bm{z}^{k-2}_{j^*} \Vert}$, and then $\frac{\Vert \bm{z}^{k-4}_j \Vert}{\Vert \bm{z}^{k-4}_{j^*} \Vert}$ and so on, to give
\begin{align}
\frac{\Vert \bm{z}^k_j \Vert}{\Vert \bm{z}^k_{j^*} \Vert}
=
c_{j, k-2}
\frac{\Vert \bm{z}^{k-2}_j \Vert}{\Vert \bm{z}^{k-2}_{j^*} \Vert}
=
c_{j, k-2}
c_{j, k-4}
\frac{\Vert \bm{z}^{k-4}_j \Vert}{\Vert \bm{z}^{k-4}_{j^*} \Vert}
=
\cdots
=
\Bigg(
\prod_{\ell = 0, 2, \ldots}^{k-2}
c_{j, \ell}
\Bigg)
\frac{\Vert \bm{z}^{0}_j \Vert}{\Vert \bm{z}^{0}_{j^*} \Vert}.
\end{align}
From \cref{lem:skew-g-function}, $g(|m_j|, \bm{r}_{\ell})$ is a strictly increasing function of $|m_j|$ regardless of the value of $\bm{r}_{\ell}$.
Hence, $c_{j, \ell} < 1$ for all $\ell$ when $j < j^*$, and, thus
\begin{align}
\lim \limits_{k \to \infty}
\frac{\Vert \bm{z}^k_j \Vert}{\Vert \bm{z}^k_{j^*} \Vert}
=
\frac{\Vert \bm{z}^{0}_j \Vert}{\Vert \bm{z}^{0}_{j^*} \Vert}
\lim \limits_{k \to \infty}
\Bigg(
\prod_{\ell = 0, 2, \ldots}^{k-2}
c_{j, \ell}
\Bigg)
=
0
\quad 
\mathrm{when}
\quad 
j < {j^*}.
\end{align}
That is, as $k \to \infty$ in \eqref{eq:rk-zk-sum} the magnitude of the first $j^* -1$ addends $\frac{\bm{z}^k_j}{\Vert \bm{z}^k_{j^*} \Vert}$, $j = 1 ,\ldots, j^* - 1$, goes to zero.
On the other hand, the last addend always has magnitude one since $\frac{\Vert \bm{z}^k_j \Vert}{\Vert \bm{z}^k_{j^*} \Vert} = 1$ when $j = j^*$.
Hence, taking the limit of $\bm{r}_k$ in \eqref{eq:rk-zk-sum} gives
\begin{align}
\lim \limits_{k \to \infty}
\bm{r}_k
=
\lim \limits_{k \to \infty} \zeta_k
\frac{\bm{z}^k_{j^*}}{\Vert \bm{z}^k_{j^*} \Vert}
\in
\mathrm{span}(U_{j^*}).
\end{align}
\end{proof}

\section{Conclusion}
\label{sec:conclusion}

Despite widespread use of AA, little is understood about its ability to accelerate an underlying fixed-point iteration. 
This work sheds light on the convergence behaviour of a restarted AA iteration with restart size of one.
While this algorithm is relatively simple, its convergence behavior is not by any means.

Here we explicitly computed the asymptotic convergence factor of this restarted AA method in the special cases that the underlying fixed-point iteration consists of applying either a symmetric or skew symmetric matrix, proving for each case that the AA iteration strictly converges faster than the underlying fixed-point iteration.
In the symmetric case we proved that this convergence factor may depend strongly on the initial iterate, and we showed how this is a manifestation of the there being an infinite number of fixed points to the underlying residual iteration, resulting in a continuum of possible convergence factors.
Conversely, in the skew-symmetric case we proved the convergence factor does not depend on the initial iterate due to it behaving analogously to a power iteration.
Numerical results were used to verify the theory and to demonstrate how it can be extended to describing AA convergence for certain nonlinear fixed-point iterations.
Our future work will focus on the analysis of windowed AA iterations and understanding to what extent the techniques employed here are applicable there.
%

\section*{Acknowledgments}

Critical feedback from anonymous referees is gratefully acknowledged.
Los Alamos Laboratory report number LA-UR-24-33091.

\bibliographystyle{siamplain}
\bibliography{rAA1-convergence.bib}

\appendix

\section{Proofs from \cref{sec:symm}}
\label{app:symm}
To help with the readability of the proof of \cref{thm:eigvec}, we split out from it \cref{lem:Rcv,lem:Rz} below.
Many results here use the following two basic properties of the Rayleigh quotient.
Let $R(S, \bm{z}) = \la \bm{z} , S \bm{z} \ra / \la \bm{z}, \bm{z} \ra$ be the Rayleigh quotient between vector $\bm{0} \neq \bm{z} \in \mathbb{R}^n$ and symmetric matrix $S \in \mathbb{R}^{n \times n}$. Then for $0 \neq c \in \mathbb{R}$ we have
\begin{align} \tag{RQP1} \label{RQprop:1}
R(S, c \bm{z}) = R(S, \bm{z}).
\end{align}
Further, if $(s_1, \bm{y}_1)$ and $(s_2,\bm{y}_2)$ are eigenpairs of $S$, with $\bm{y}_1$ and $\bm{y}_2$ orthogonal, then
\begin{align} \tag{RQP2} \label{RQprop:2}
R(S, \bm{y}_1 + \bm{y}_2) = \frac{s_1 \Vert \bm{y}_1 \Vert^2 + s_2 \Vert \bm{y}_2 \Vert^2}{\Vert \bm{y}_1 \Vert^2 + \Vert \bm{y}_2 \Vert^2}.
\end{align}
%

\subsection{Auxiliary results}
\label{app:symm-aux}

\begin{lemma} \label{lem:Rcv}
Let ${\cal R}$ be the two-step residual propagator in \eqref{eq:calR-def} with  $M$ symmetric.
Let $0 \neq c \in \mathbb{R}$ and $\bm{v}$ be an eigenvector of $M$.
Then, 
\begin{align} \label{eq:Rcv}
{\cal R}(c \bm{v}) = \bm{0}.
\end{align}
\end{lemma}

\begin{proof}

Plugging into \eqref{eq:calR-def} and using that $\bm{v}$ is an eigenvector of $M$ and $A$ with eigenvalues $m = 1 - a$ and $a$, respectively, gives 
\begin{align} \label{eq:Rcv-int-1}
{\cal R}(\bm{c v}) 
= 
c M \big[ I - \alpha(c \bm{v}) A \big] \bm{v}
=
c m \big[ 1 - a \alpha(c \bm{v}) \big] \bm{v}.
\end{align}
Using the definition of $\alpha$ in \eqref{eq:alpha-def} and \eqref{RQprop:1} gives $\alpha(c \bm{v}) = \alpha( \bm{v} ) = 1/a$, since $\bm{v}$ is an eigenvector of $A^{-1}$ with eigenvalue $1/a$. Plugging into \eqref{eq:Rcv-int-1} gives the claim.
\end{proof}

\begin{lemma} \label{lem:Rz}
Let ${\cal R}$ be the two-step residual propagator in \eqref{eq:calR-def} with  $M$ symmetric.
Let $\bm{v}_i$ and $\bm{v}_j$ be two orthogonal eigenvectors of $M$ with eigenvalues $m_i$ and $m_j$, respectively. Denote the corresponding eigenvalues of $A$ by $a_i = 1 - m_i$, and $a_j = 1 - m_j$.
Let $c_i, c_j \in \mathbb{R}$ be non-zero constants, and define the vector $\bm{z}$ as
\begin{align} \label{eq:z-def-1}
\bm{z} := c_i \bm{v}_i + c_j \bm{v}_j 
= 
c_i ( \bm{v}_i + \varepsilon \bm{v}_j),
\quad
\varepsilon := \frac{c_j}{c_i}.
\end{align}
Then, 
\begin{align} \label{eq:Rz}
{\cal R} ( \bm{z} )
=
c_i
\frac{(a_j - a_i) \varepsilon }{a_i^2 + a_j^2 \varepsilon^2} 
\big[
m_i a_j \varepsilon \bm{v}_i
-
m_j a_i \bm{v}_j
\big].
\end{align}
\end{lemma}

\begin{proof}
For readability, let us set $i = 1$ and $j = 2$ in \eqref{eq:z-def-1}.
Since the constants $c_1, c_2$ in \eqref{eq:z-def-1} are arbitrary, without loss of generality, let us assume that  $\Vert \bm{v}_1 \Vert = \Vert \bm{v}_2 \Vert = 1$. 
From \eqref{eq:calR-def} and \eqref{eq:alpha-def} we have that
\begin{align}
{\cal R}(\bm{z}) 
= 
M \big[ I - \alpha(\bm{z}) A \big] \bm{z},
\quad
\alpha( \bm{z} ) = R\big( A^{-1}, A \bm{z} \big).
\end{align}
Plugging $\bm{z} = c_1 ( \bm{v}_1 + \varepsilon \bm{v}_2)$ into ${\cal R}$ and using that $\bm{v}_1$ and $\bm{v}_2$ are eigenvectors gives
\begin{align} \label{eq:Rz-int-0}
{\cal R}(\bm{z}) 
=
c_1 \big[ m_1 \big( 1 - \alpha(\bm{z}) a_1 \big) \bm{v}_1 
+ 
m_2 \big( 1 - \alpha(\bm{z}) a_2 \big) \varepsilon \bm{v}_2 \big].
\end{align}
Now consider $\alpha(  c_1 ( \bm{v}_1 + \varepsilon \bm{v}_2) )$. Making use of \eqref{RQprop:1}, \eqref{RQprop:2}, and the fact that $\bm{v}_1,\bm{v}_2$ are orthogonal eigenvectors of $A$ we have:
\begin{align} \label{eq:beta-z}
\alpha( \bm{z} ) 
&= 
R\big( A^{-1}, c_1 A ( \bm{v}_1 + \varepsilon \bm{v}_2) \big)
=
R\big(A^{-1}, a_1 \bm{v}_1 + a_2 \varepsilon \bm{v}_2 )\big)
=
\frac{a_1 + a_2 \varepsilon^2}{a_1^2 + a_2^2 \varepsilon^2}.
\end{align}
Now let us plug \eqref{eq:beta-z} into the two factors in \eqref{eq:Rz-int-0}, yielding
\begin{align}
1 - \alpha(\bm{z}) a_1
=
\frac{a_2 - a_1}{a_1^2 + a_2^2 \varepsilon^2} a_2 \varepsilon^2,
\quad
1 - \alpha(\bm{z}) a_2
=
-\frac{a_2 - a_1}{a_1^2 + a_2^2 \varepsilon^2} a_1.
\end{align}
Plugging these back into \eqref{eq:Rz-int-0} gives
\begin{align}
{\cal R}(\bm{z}) 
&=
c_1 \left[ 
\frac{a_2 - a_1}{a_1^2 + a_2^2 \varepsilon^2} m_1 a_2 \varepsilon^2 \bm{v}_1 
 -
\frac{a_2 - a_1}{a_1^2 + a_2^2 \varepsilon^2} m_2 a_1 \varepsilon \bm{v}_2 
\right].
\end{align}
Pulling out the common factor of $\frac{a_2 - a_1}{a_1^2 + a_2^2 \varepsilon^2} \varepsilon$ gives the claim \eqref{eq:Rz}.
\end{proof}

\subsection{Proof of \cref{thm:eigvec}}
\label{app:proof-thm-eigvec}

\begin{proof}

For readability, let us set $i = 1$ and $j = 2$ in \eqref{eq:z-def-2}.
For convenience, throughout the proof we also use the eigenvalues of $A$ given by $a_1 = 1 - m_1$, and $a_2 = 1 - m_2$.
From \eqref{eq:Rz} of \cref{lem:Rz} we have that
\begin{align} \label{eq:Rz-int-2}
{\cal R} ( \bm{z} )
=
c_1
\frac{(a_2 - a_1) \varepsilon }{a_1^2 + a_2^2 \varepsilon^2} 
\big[
m_1 a_2 \varepsilon \bm{v}_1
-
m_2 a_1 \bm{v}_2
\big].
\end{align}
We now consider three distinct cases.

\underline{$m_1 = m_2$:} Clearly $m_2 = m_1 \Longleftrightarrow a_2 = a_1 $. As such, if $m_1 = m_2$ then ${\cal R}( \bm{z} ) = \bm{0}$, as claimed in \eqref{eq:Rz-zero}.
We note that this is also valid when $m_1 = m_2 = 0$.
Thus, moving forward we assume that $m_1$ and $m_2$ are distinct.

\underline{$m_1 \neq m_2$ and $m_1 = 0$ or $m_2 = 0$:}
From \eqref{eq:Rz-int-2}, it is the case that ${\cal R} ( \bm{z} ) = c \bm{v}$ where $c$ is a non-zero constant and $\bm{v}$ is an eigenvector of $M$.
By virtue of \eqref{eq:Rcv} from \cref{lem:Rcv}, it must be the case that
${\cal R} ({\cal R} ( \bm{z} )) = X(\bm{z}) \bm{z} = \bm{0}$; note that this result is consistent with \cref{thm:eigvec} because the eigenvalue $\lambda$ in \eqref{eq:lambda} vanishes if either of $m_1$ or $m_2$ are zero.
Thus, moving forward, we assume that $m_1$ and $m_2$ are distinct and non-zero.

\underline{$m_1 \neq m_2$ and $m_1 \neq 0$ and $m_2 \neq 0$:} Recalling that neither $a_1$ nor $a_2$ are zero because then $A$ would not be invertible, from \eqref{eq:Rz-int-2} we must have 
${\cal R} ( \bm{z} )
=
\wt{c}_1 \bm{v}_1
+
\wt{c}_2 \bm{v}_2
$
for two non-zero constants $\wt{c}_1$, $\wt{c}_2$.
We now use this to re-write \eqref{eq:Rz-int-2} as
\begin{align} \label{eq:Rz-int-3}
{\cal R} ( \bm{z} )
=
\underbrace{
c_1
\frac{(a_2 - a_1) m_1 a_2 \varepsilon^2}{a_1^2 + a_2^2 \varepsilon^2} 
}_{\displaystyle =: \wt{c}_1}
\bigg[
\bm{v}_1
+
\underbrace{
\left(
-
\frac{m_2 a_1}{m_1 a_2} \frac{1}{\varepsilon}
\right)
}_{\displaystyle =: \wt{\varepsilon}}
\bm{v}_2
\bigg]
=
\wt{c}_1 ( \bm{v}_1 + \wt{\varepsilon} \bm{v}_2 )
=:
\wt{\bm{z}}
\end{align}
Now we apply ${\cal R}$ to this vector; from \eqref{eq:Rz} of \cref{lem:Rz} we have that
\begin{align} \label{eq:RRz-int-1}
{\cal R} ({\cal R} ( \bm{z} ))
=
{\cal R} ( \wt{\bm{z}} )
=
\wt{c}_1
\frac{(a_2 - a_1) \wt{\varepsilon} }{a_1^2 + a_2^2 \wt{\varepsilon}^2} 
\big[
m_1 a_2 \wt{\varepsilon} \bm{v}_1
-
m_2 a_1 \bm{v}_2
\big].
\end{align}
Clearly $\wt{\varepsilon} \neq 0$, and, thus, by the same argument as previously, we can pull out the factor $m_1 a_2 \wt{\varepsilon}$ to get
\begin{align} \label{eq:RRz-int-2}
{\cal R} ({\cal R} ( \bm{z} ))
=
\wt{c}_1
\frac{(a_2 - a_1) m_1 a_2 \wt{\varepsilon}^2 }{a_1^2 + a_2^2 \wt{\varepsilon}^2} 
\bigg[
\bm{v}_1
+
\left( 
-
\frac{m_2 a_1}{m_1 a_2 } \frac{1}{\wt{\varepsilon}}
\right)
\bm{v}_2
\bigg].
\end{align}
From the definition of $\wt{\varepsilon}$ in \eqref{eq:Rz-int-3}, note that $-\frac{m_2 a_1}{m_1 a_2 } \frac{1}{\wt{\varepsilon}} = \varepsilon$.
Plugging this into \eqref{eq:RRz-int-2} and replacing $\wt{c}_1$ in \eqref{eq:RRz-int-2} with its definition in \eqref{eq:Rz-int-3} results in
\begin{align} \label{eq:RRz-int-3}
X(\bm{z}) \bm{z}
=
{\cal R} ({\cal R} ( \bm{z} ))
=
\frac{(a_2 - a_1) m_1 a_2 \varepsilon^2}{a_1^2 + a_2^2 \varepsilon^2}
\frac{(a_2 - a_1) m_1 a_2 \wt{\varepsilon}^2 }{a_1^2 + a_2^2 \wt{\varepsilon}^2}
\cdot 
\Big[
c_1
(\bm{v}_1
+
\varepsilon
\bm{v}_2)
\Big].
\end{align}
Recalling from \eqref{eq:z-def-2} that $\bm{z} = c_1
(\bm{v}_1
+
\varepsilon
\bm{v}_2)$,
we thus have that $\bm{z}$ is an eigenvector of $X(\bm{z})$ as claimed in \cref{thm:eigvec}, where the associated eigenvalue $\lambda$ is the leading constant in \eqref{eq:RRz-int-3}.\\
We now simplify the expression for the eigenvalue in \eqref{eq:RRz-int-3}.
First we write
\begin{align} \label{eq:lambda-int-1}
\lambda
=
\frac{(a_2 - a_1) m_1 a_2 \varepsilon^2}{a_1^2 + a_2^2 \varepsilon^2}
\frac{(a_2 - a_1) m_1 a_2 \wt{\varepsilon}^2 }{a_1^2 + a_2^2 \wt{\varepsilon}^2}
=
\frac{(a_2 - a_1)^2}{a_1^2 + a_2^2 \varepsilon^2}
\frac{(m_1 a_2 \varepsilon \wt{\varepsilon})^2}{a_1^2 + a_2^2 \wt{\varepsilon}^2}.
\end{align}
Plugging in $\wt{\varepsilon} = -\frac{m_2 a_1}{m_1 a_2} \frac{1}{\varepsilon}$ to the second factor we find
\begin{align}
\frac{(m_1 a_2 \varepsilon \wt{\varepsilon})^2}{a_1^2 + a_2^2 \wt{\varepsilon}^2}
=
\frac{(m_2 a_1)^2}{\displaystyle a_1^2 + \frac{m_2^2 a_1^2}{\varepsilon^2 m_1^2} }
=
\frac{(m_1 m_2)^2}{\displaystyle m_1^2 + \frac{m_2^2}{\varepsilon^2}   }.
\end{align}
Plugging this into \eqref{eq:lambda-int-1} and replacing $a_1 = 1 - m_1$ and $a_2 = 1 - m_2$ gives \eqref{eq:lambda}.
\end{proof}

\subsection{Proof of \cref{lem:max-lambda}}
\label{app:proof-lem-lambda-max}

\begin{proof}
For readability, set $i = 1$, $j = 2$.
Using \eqref{eq:lambda} we may express $\lambda(\varepsilon)$ as
$
\lambda( \varepsilon ) = c / g(\varepsilon),
$
where $c = (m_1 m_2)^2 (m_2 - m_1)^2$ is a constant, and the function $g(\varepsilon)$ is
\begin{align} \label{eq:g-def}
g(\varepsilon)
=
\Big[ 
(m_1-1)^2 +  \varepsilon^2 (m_2 - 1 )^2
\Big]
\bigg[ 
m_1^2 + \frac{m_2^2}{\varepsilon^2}
\bigg].
\end{align}
Critical points of $\lambda$ occur where $\lambda'(\varepsilon) = - \frac{c g'( \varepsilon)}{[g(\varepsilon)]^2} = 0$. Since $g > 0$, it suffices to find the roots of $g'( \varepsilon) = 0$.
By the product rule, we have
\begin{align}
g'( \varepsilon)
=
2 \varepsilon (m_2 - 1 )^2
\bigg[ 
m_1^2 + \frac{m_2^2}{\varepsilon^2}
\bigg]
-
\Big[ 
(m_1-1)^2 +  \varepsilon^2 (m_2 - 1 )^2
\Big]
\frac{2 m_2^2}{\varepsilon^3}.
\end{align}
Setting this equal to zero and multiplying through by $\varepsilon^3/2$ and then cancelling common terms results in the equation
$
\varepsilon^4
=
\left( \tfrac{m_2 (m_1 - 1)}{m_1 (m_2 - 1)} \right)^2.
$
Taking the fourth root of both sides and discarding the two imaginary solutions results in the critical values $\varepsilon = \pm \sqrt{\left| \frac{m_2 (m_1 - 1)}{m_1 (m_2 - 1)}\right| }$.
Noting from \eqref{eq:g-def} that $g(\varepsilon)$ is symmetric about $\varepsilon = 0$, and that $\lim_{\varepsilon \to \infty} g( \varepsilon ) = \lim_{\varepsilon \to 0} g( \varepsilon ) = \infty$, it must be the case that these two critical points of $g$ are global minima.
Recalling $\lambda = c / g$, global minima of $g$ must correspond to global maxima of $\lambda$.
Plugging these critical values into $g$ and simplifying results in \eqref{eq:lambda-max}.
\end{proof}

\subsection{Proof of \cref{thm:2x2}}
\label{app:symm-2x2-proof}

\begin{proof}
Result \eqref{eq:rAA1-rho-special} follows from \cref{lem:Rcv} in the case that $\bm{r}_0$ is parallel to an eigenvector.
The more general result \eqref{eq:rAA1-rho-general} follows from the more general nonlinear eigenvalue result of \cref{thm:eigvec}. 
Specifically, taking $\bm{r}_0$ equal to $\bm{z}$ in \cref{thm:eigvec} tells us that the residual vectors of the rAA(1) iteration will be four-periodic with
\begin{align} \label{eq:r4j-ratios}
\bm{r}_{4j}
=
\lambda( \varepsilon )
\bm{r}_{4(j-1)}
\quad
\Longrightarrow
\quad
\frac{\Vert \bm{r}_{4j} \Vert}{\Vert \bm{r}_{4(j-1)} \Vert} 
= 
\lambda( \varepsilon ),
\quad
j = 1, 2, \ldots
\end{align}
Now consider $\varrho_{k}(\bm{r}_0)$ from \cref{def:conv-fac}; for $k \geq 4$ we have:
\begin{align}
\label{eq:rho_k}
\varrho_{k}(\bm{r}_0) 
=
\Vert \bm{r}_0 \Vert^{1/k}
\left( \frac{\Vert \bm{r}_1 \Vert}{\Vert \bm{r}_0 \Vert} 
\,
\frac{\Vert \bm{r}_2 \Vert}{\Vert \bm{r}_1 \Vert}
\cdots 
\frac{\Vert \bm{r}_{k-1} \Vert}{\Vert \bm{r}_{k-2} \Vert}
\,
\frac{\Vert \bm{r}_k \Vert}{\Vert \bm{r}_{k-1} \Vert} \right)^{1/k}
=
\big( \varrho_{\ell} \Vert \bm{r}_0 \Vert \big)^{1/k}
\left( 
\prod \limits_{j = 1}^{\ell}
\frac{\Vert \bm{r}_{j} \Vert}{\Vert \bm{r}_{j-1} \Vert} 
\right)^{1/k}
\end{align}
where we have defined the shorthands
\begin{align}
\ell
:=
k - k \textrm{ mod } 4,
\quad
\textrm{and}
\quad
\varrho_{\ell}
:=
\begin{cases}
1, \quad &k \textrm{ mod } 4 = 0,
\\[1ex]
\frac{\Vert \bm{r}_k \Vert}{\Vert \bm{r}_{k-1} \Vert},
\quad
&k \textrm{ mod } 4 = 1,
\\[1ex]
\frac{\Vert \bm{r}_{k-1} \Vert}{\Vert \bm{r}_{k-2} \Vert}
\frac{\Vert \bm{r}_k \Vert}{\Vert \bm{r}_{k-1} \Vert},
\quad
&k \textrm{ mod } 4 = 2,
\\[1ex]
\frac{\Vert \bm{r}_{k-2} \Vert}{\Vert \bm{r}_{k-3} \Vert}
\frac{\Vert \bm{r}_{k-1} \Vert}{\Vert \bm{r}_{k-2} \Vert}
\frac{\Vert \bm{r}_k \Vert}{\Vert \bm{r}_{k-1} \Vert},
\quad
&k \textrm{ mod } 4 = 3.
\end{cases}
\end{align}
Since $\ell$ is divisible by four, we judiciously telescopically cancel some terms in \eqref{eq:rho_k}:
\begin{align}
\left( 
\prod \limits_{j = 1}^{\ell}
\frac{\Vert \bm{r}_{j} \Vert}{\Vert \bm{r}_{j-1} \Vert} 
\right)^{1/k}
=
\left( 
\prod \limits_{j = 1}^{\ell / 4}
\frac{\Vert \bm{r}_{4j} \Vert}{\Vert \bm{r}_{4(j-1)} \Vert} 
\right)^{1/k}
=
\left( 
\big[ \lambda( \varepsilon ) \big]^{1/4}
\right)^{\ell/k}
\end{align}
where we used \eqref{eq:r4j-ratios} to plug in $\lambda( \varepsilon )$.
Now we take the limit of $k \to \infty$, noting that $1/k \to 0$, and $\ell / k \to 1$, such that $\varrho( \bm{r}_0 ) := \lim_{k \to \infty} \varrho_k( \bm{r}_0 ) = \big[ \lambda( \varepsilon ) \big]^{1/4}$.
Plugging in the expression for $\lambda( \varepsilon )$ given in \eqref{eq:lambda} yields \eqref{eq:rAA1-rho-general}.
\end{proof}


\newpage

\setcounter{section}{0}
\setcounter{equation}{0}
\setcounter{figure}{0}
\setcounter{table}{0}
\setcounter{page}{1}
\makeatletter
\renewcommand{\thesection}{SM\arabic{section}}
\renewcommand{\theequation}{SM\arabic{equation}}
\renewcommand{\thefigure}{SM\arabic{figure}}
\renewcommand{\thetable}{SM\arabic{table}}
\renewcommand{\thepage}{SM\arabic{page}}

\thispagestyle{plain} 

\headers{{SUPPLEMENTARY MATERIALS: Asymptotic convergence of rAA(1)}}{O. A. Krzysik, H. De Sterck, A. Smith}

\begin{center}
    \textbf{\normalsize\MakeUppercase{
    Supplementary Materials:
    Asymptotic convergence of restarted Anderson acceleration for certain normal linear systems}} 
    \vspace{6ex}
\end{center}

\section{Periodicity of rAA($m$) residuals}
\label{SMsec:rAAm}

As mentioned at the end of Section \ref{sec:symm-eig}, in our numerical tests we often observe that residuals of rAA($m$) are approximately periodic with period $2(m+1)$.
To obtain an rAA($m$) iteration, i.e., restarted AA with a restart window of $m \geq 0$, we take the general AA iteration \eqref{eq:AA} with memory parameter $m_k = k \, \mathrm{mod} \, (m+1)$, which means that every $m+1\geq1$ iterations no previous memory is used in computing the corresponding iterate. 

As a numerical demonstration of this phenomenon, we take a randomly generated symmetric matrix $M \in \mathbb{R}^{14 \times 14}$ and we run the above rAA($m$) iteration with underlying fixed-point function $\bm{q}(\bm{x}) = M \bm{x}$.
For a fixed restart size $m$ and a fixed initial iterate $\bm{x}_0$ we generate a sequence of rAA($m$) residual vectors $\{ \bm{r}_k \}_{k = 0}$ and compute the following function of this sequence:
\begin{align} \label{eq:sm-angle}
\left| 1 - \frac{\langle \bm{r}_{k}, \bm{r}_{k+2(m+1)} \rangle}{\Vert \bm{r}_k \Vert \Vert \bm{r}_{k + 2(m+1)} \Vert} \right|
=
\left| 1 - \cos \theta_{k, k + 2(m+1)} \right|, \quad k = 0, 1, 2, \ldots, 
\end{align}
with $\theta_{k, k + 2(m+1)}$ the angle between the vectors $\bm{r}_{k}$ and $\bm{r}_{k + 2(m+1)}$.
\Cref{fig:rAAm-periodicity} shows plots of this function for $m \in \{1, 3, 5, 7\}$, where for each $m$ we run the iteration for four randomly chosen initial iterates $\bm{x}_0$.
Supposing the residual vectors are non-zero, notice that the quantity in \eqref{eq:sm-angle} is zero if and only if the rAA($m$) residual vectors are $2(m+1)$-periodic because then $\theta_{k, k + 2(m+1)} = 0$.
Considering the plots, notice that overall \eqref{eq:sm-angle} is small for $k$ large enough, indicating that the residual vectors are indeed approximately $2(m+1)$ periodic.
Moreover, in most cases the function \eqref{eq:sm-angle} appears to decrease with $k$, indicating that as a given rAA($m$) iteration proceeds the residuals get closer and closer to being parallel.
It is also interesting to note that in the case of $m=1$ this function tends to zero much faster and more ``cleanly'' than in the other cases.
Without further study the reason for this qualitative difference for $m > 1$ remains unclear.
E.g., rounding errors may play a role in this difference given that we are attempting to accurately measure functions of numerically small residuals.

\begin{figure}[t!]
\centerline{
\includegraphics[width=0.475\textwidth]{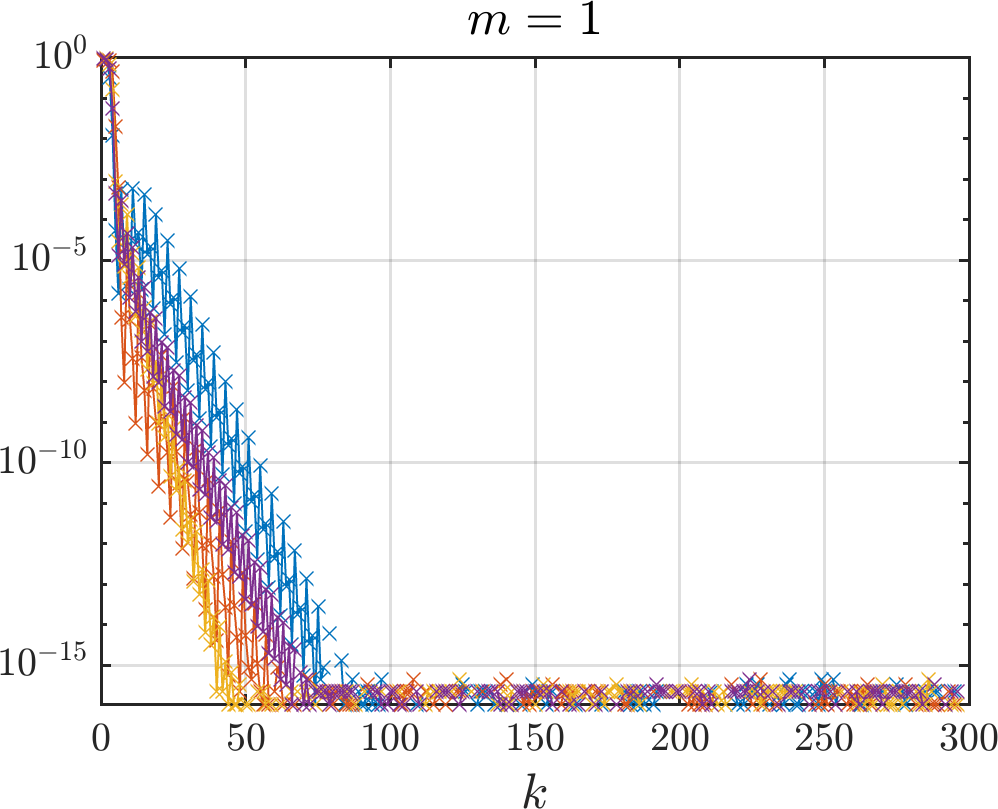}
\hspace{1ex}
\includegraphics[width=0.475\textwidth]{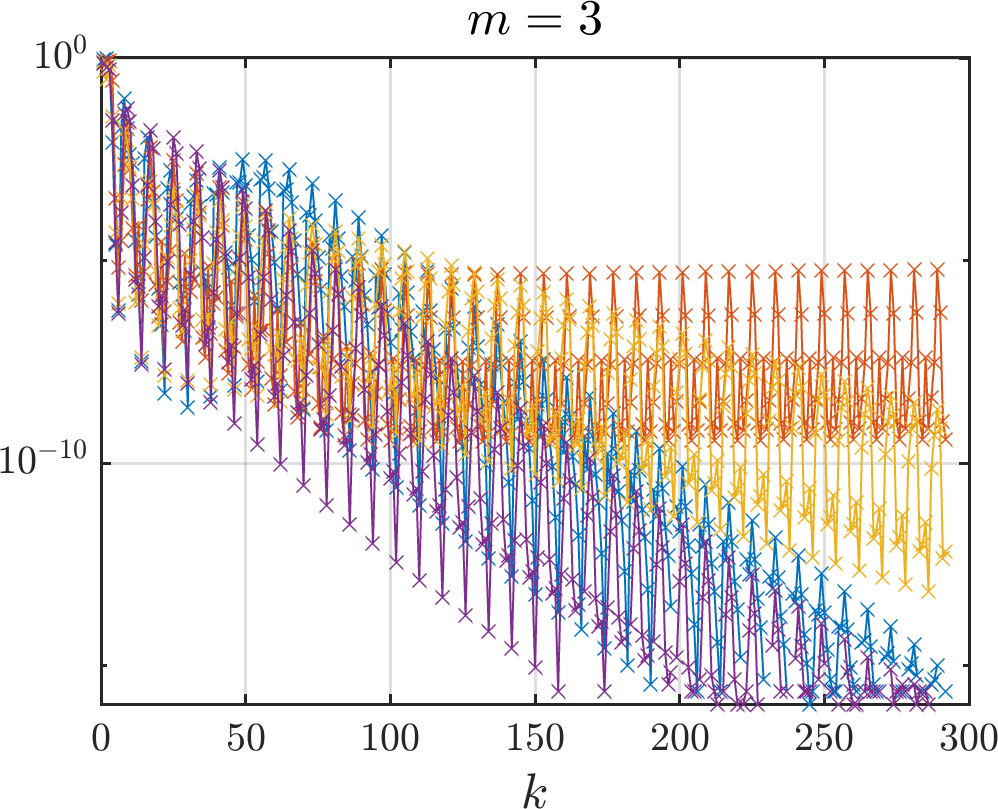}
}
\vspace{1ex}
\centerline{
\hspace{-1.5ex}
\includegraphics[width=0.475\textwidth]{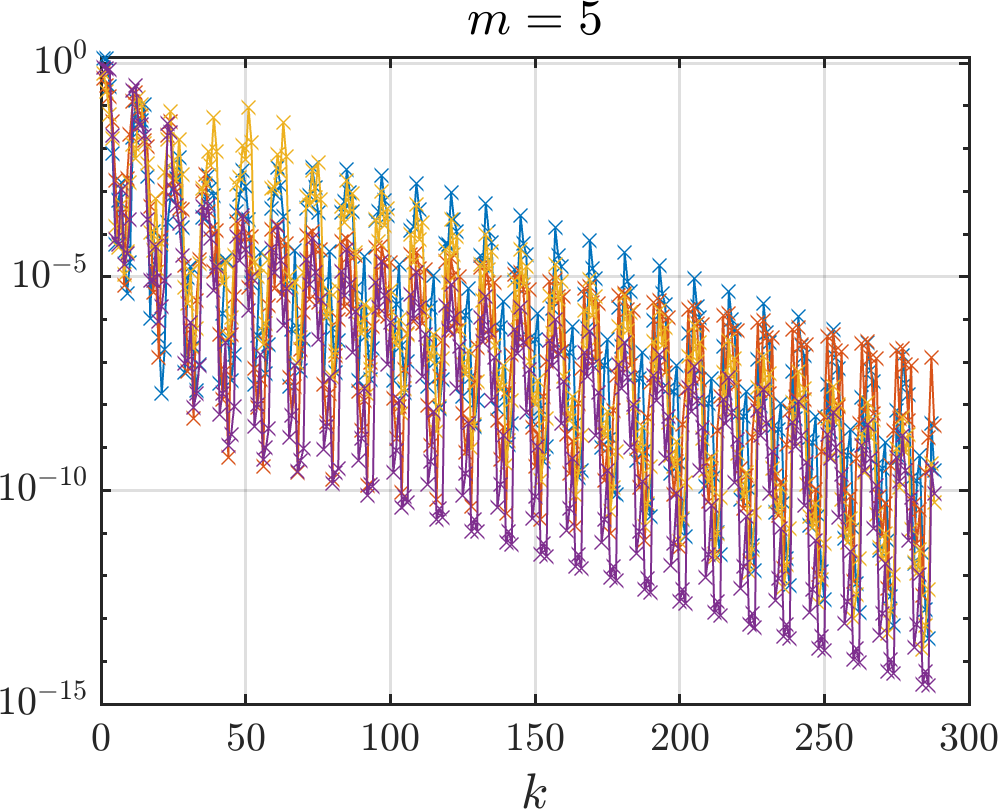}
\hspace{1.5ex}
\includegraphics[width=0.475\textwidth]{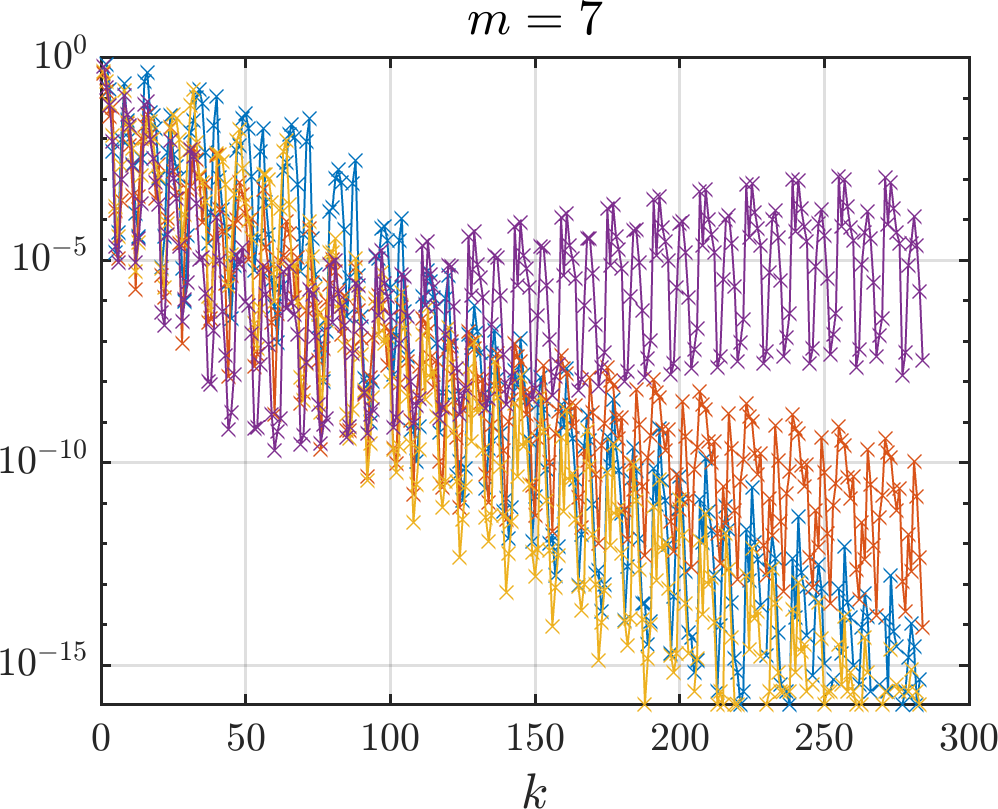}
}
\caption{For $m \in \{1, 3, 5, 7\}$, the function \eqref{eq:sm-angle} is shown which is a proxy for the angle between rAA($m$) residual vectors $2(m+1)$ iterations apart. For each $m$, four different sequences of residual vectors are shown corresponding to four different initial iterates $\bm{x}_0$.
Note that the same set of four initial iterates $\{ \bm{x}_0 \}$ is used for all $m$.
\label{fig:rAAm-periodicity}
}
\end{figure}
%

\newpage
\section{Supporting numerical evidence for conjectures}
\label{SMsec:conj-num-res}

\subsection{Conjecture \ref{conj:X-fp-conv}}

In \cref{SMfig:symm-asymptotic-conj} we present numerical evidence for Conjecture  \ref{conj:X-fp-conv}. Here we show for the first test problem in Section \ref{sec:symm-num-res} and Figure \ref{fig:symm-rhok} that as $k \to \infty$ the residual vector tends to a linear combination of two eigenvectors of $M$.
Further numerical tests (not shown here) for other matrices $M$, such as the two-grid operator \eqref{eq:symmM-mg}, indicate that while Conjecture  \ref{conj:X-fp-conv} likely holds, the residual does not necessarily appear to converge to a linear combination of just two eigenvectors.

\begin{figure}[h!]
\centerline{
\includegraphics[width=0.475\textwidth]{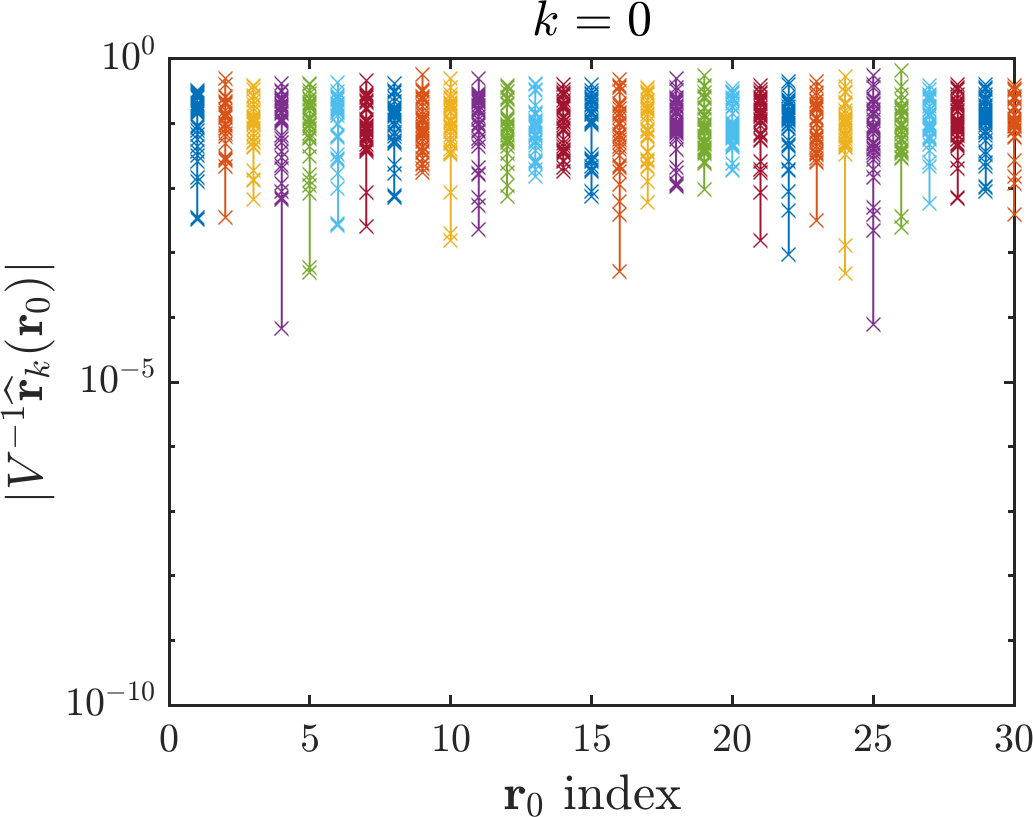}
\hspace{1ex}
\includegraphics[width=0.475\textwidth]{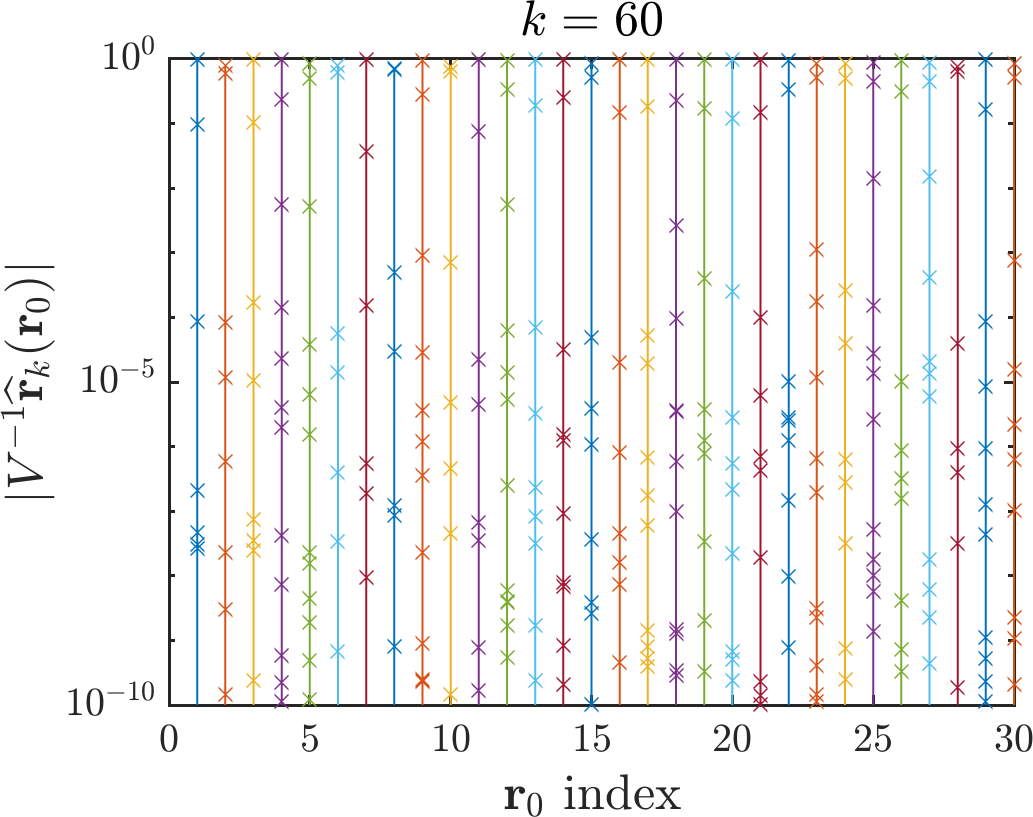}
}
\vspace{1ex}
\centerline{
\includegraphics[width=0.475\textwidth]{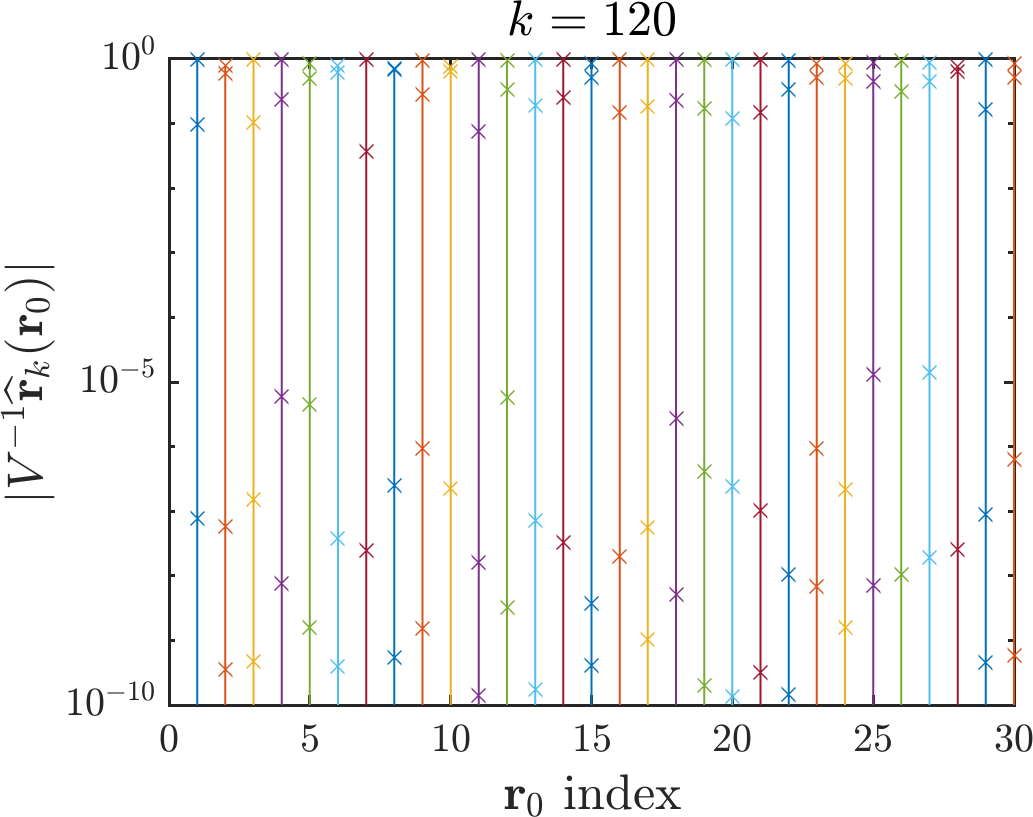}
\hspace{1ex}
\includegraphics[width=0.475\textwidth]{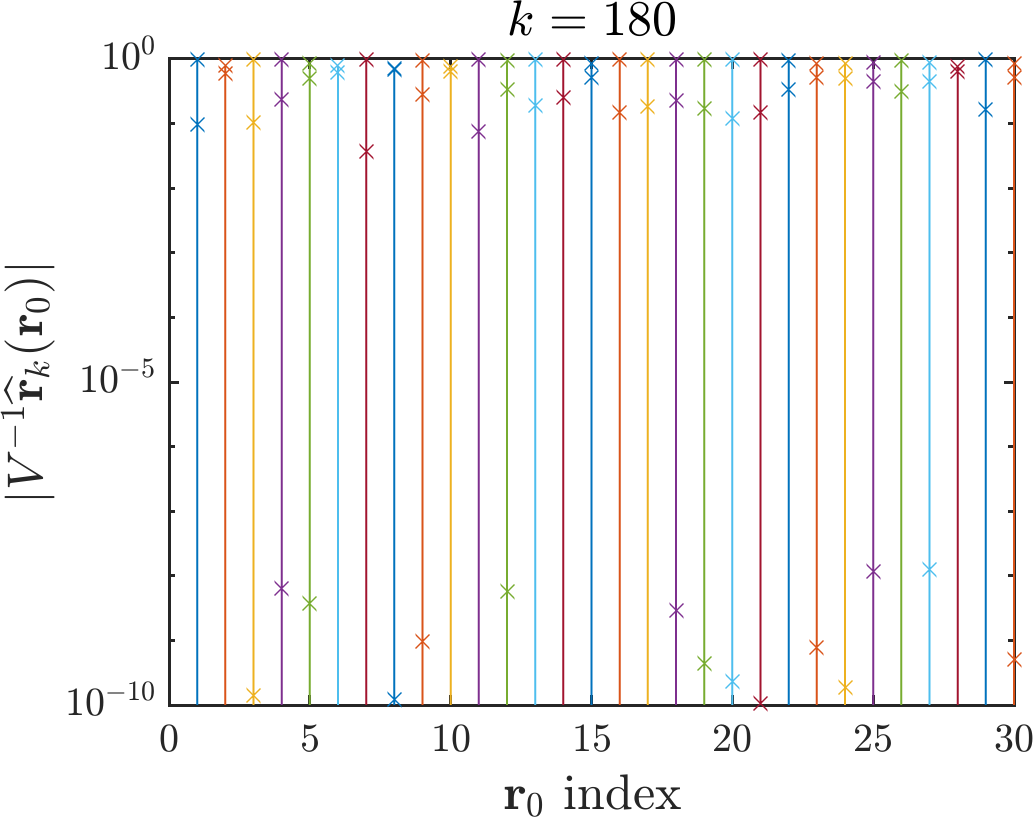}
}
\caption{Supporting numerical evidence for Conjecture \ref{conj:X-fp-conv}. If $V \in \mathbb{R}^{n \times n}$ is a matrix of orthogonal eigenvectors of $M$, and $\wh{\bm{r}}_k := \bm{r}_k / \Vert \bm{r}_k \Vert$, then coefficients in the vector $V^{-1} \wh{\bm{r}}_k$ give the decomposition of the residual in the eigenvector basis of $M$.
In the plots $|V^{-1} \wh{\bm{r}}_k|$ denotes the absolute values of the elements in the vector $V^{-1} \wh{\bm{r}}_k$.
Observe that as $k \to \infty$, it appears that for a given $\bm{r}_0$ all but two of these coefficients tend to zero.
\label{SMfig:symm-asymptotic-conj}
}
\end{figure}

\subsection{Conjecture  \ref{conj:symm-Xvmax}}

In \cref{fig:symm-Xv-conj} we present further supporting numerical evidence for Conjecture \ref{conj:symm-Xvmax}, in addition to that already shown in Figure \ref{fig:symm-rhok}.
We consider a sequence of symmetric matrices $M_{\ell} \in \mathbb{R}^{n(\ell) \times n(\ell)}$ for $\ell = 1, \ldots, 16$, with $n(\ell) = 5 \ell$.
Each $M_{\ell}$ is diagonal, with elements chosen randomly in the interval $[-1, 1]$.
For each $M_{\ell}$, we compute the (fourth root of the) ratio in question from Conjecture \ref{conj:symm-Xvmax}, i.e., $\big( \Vert X_{\ell}(\bm{v}) \bm{v} \Vert / \Vert \bm{v} \Vert \big)^{1/4}$, over a set of 10,000 test vectors $\{ \bm{v} \}$.
To choose these test vectors we consider two strategies: (i) each $\bm{v}$ is chosen randomly, i.e., with random elements drawn from $[-1, 1]$; (ii) for each of the randomly chosen $\bm{v}$ from the first strategy, we iterate 10 times $\bm{v} \gets X_{\ell}(\bm{v}) \bm{v}$, and then we compute $\big( \Vert X_{\ell}(\bm{v}) \bm{v} \Vert / \Vert \bm{v} \Vert / \big)^{1/4}$.
For both sets of test vectors $\{ \bm{v} \}$ we see that the conjecture is not violated since the green markers never exceed the magenta markers.
For the second strategy (right panel) we see that the green markers push right up to the magenta markers, consistent with Conjecture \ref{conj:symm-Xvmax}.
The fact that in the left-hand panel no vectors (other than in the $\ell = 1$ test) come close to attaining the conjectured maximum of $\big( \Vert X_{\ell}(\bm{v}) \bm{v} \Vert / \Vert \bm{v} \Vert / \big)^{1/4}$ suggests that such approximately maximizing vectors are not dense in the set of random vectors.
On the other hand, the fact that the maximum is approximately attained by many vectors in the middle plot suggests that such vectors are relatively denser in the set of vectors that are approximate fixed points of the iteration $\bm{v} \gets X_{\ell}(\bm{v}) \bm{v}$.
%
%

\begin{figure}[h!]
\centerline{
\includegraphics[width=0.475\textwidth]{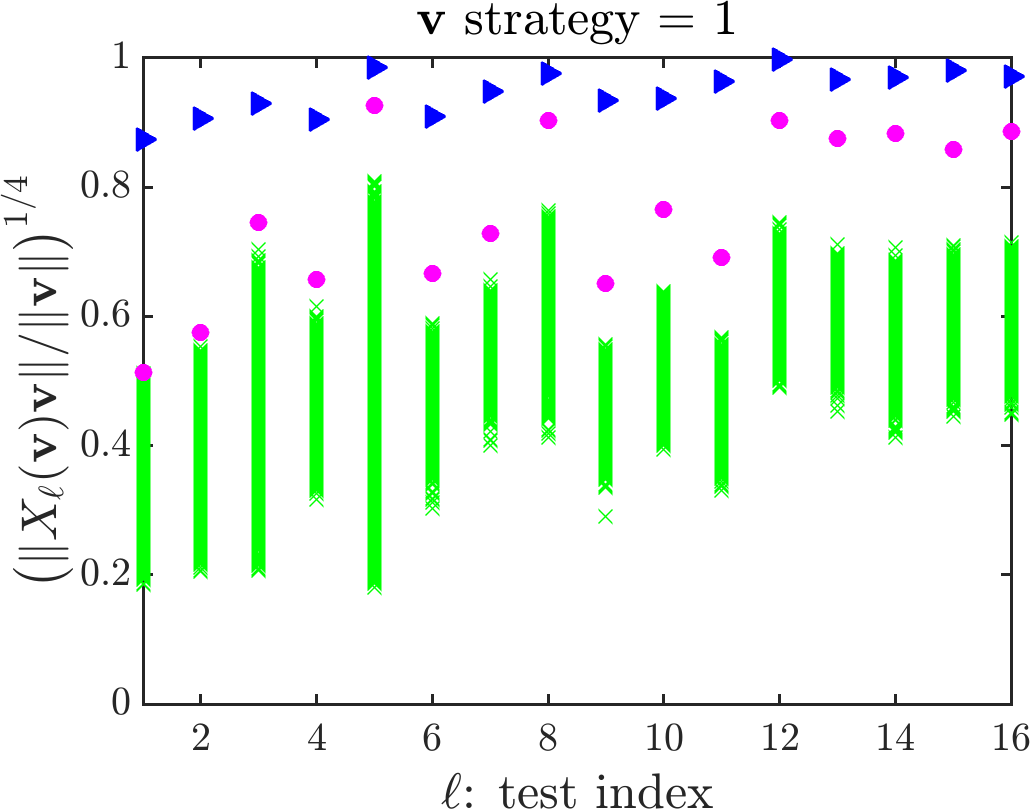}
\hspace{1ex}
\includegraphics[width=0.475\textwidth]{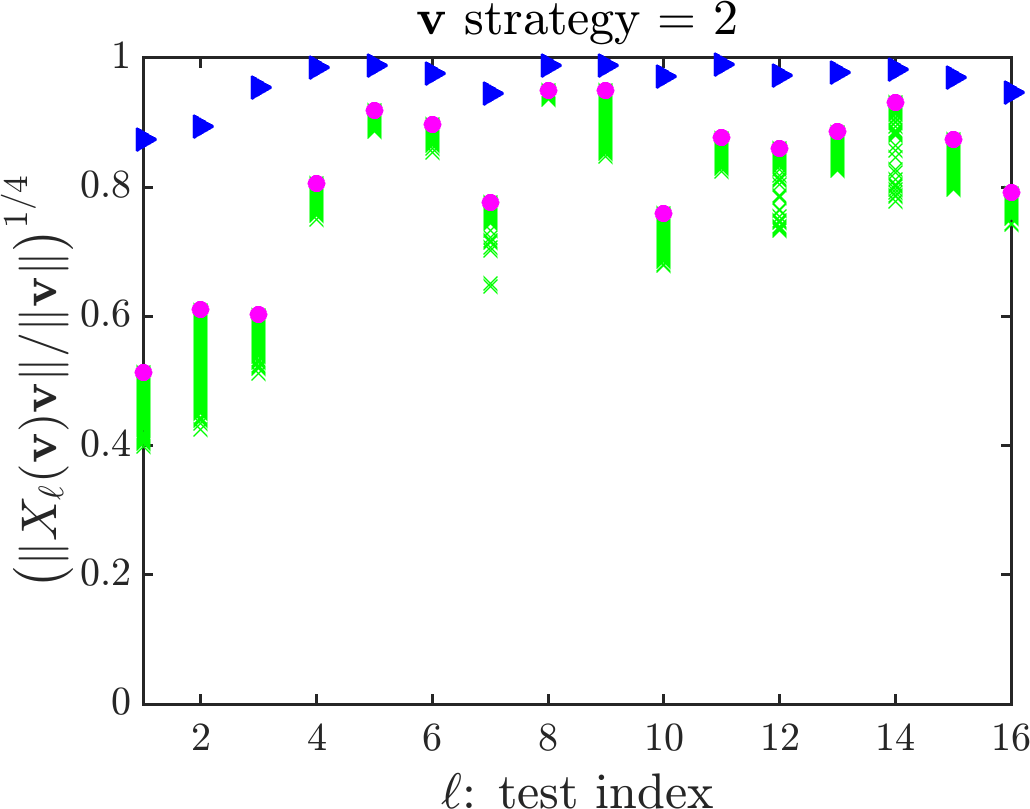}
}
\vspace{1ex}
\centerline{
\includegraphics[width=0.475\textwidth]{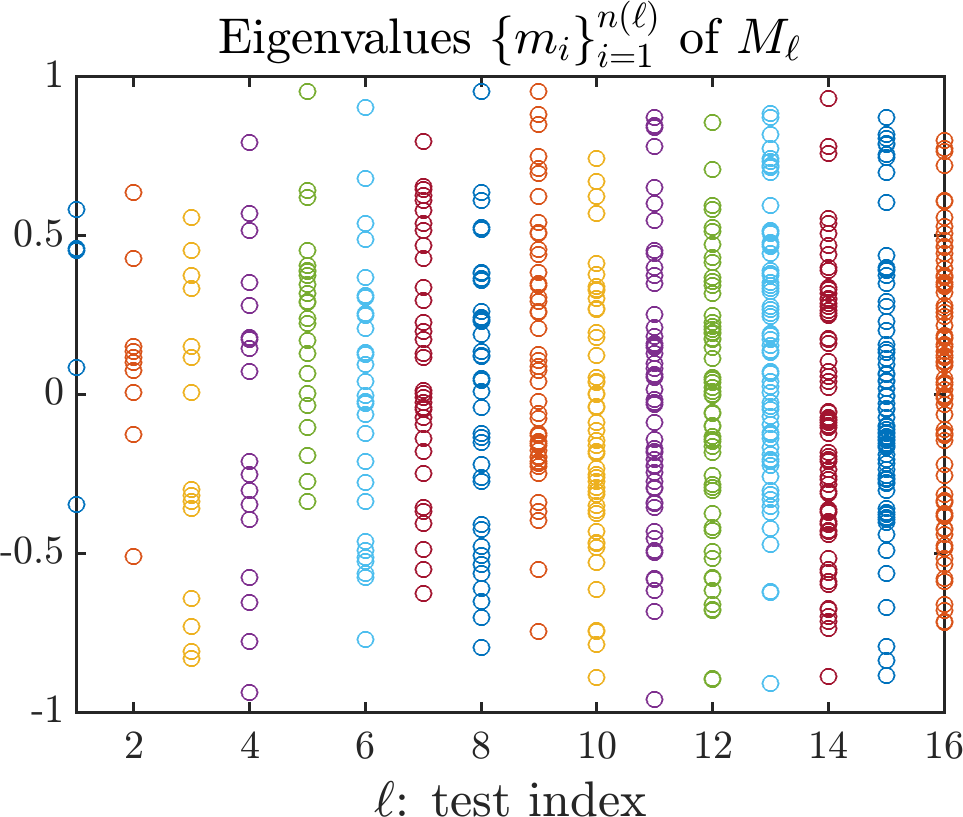}
}
\caption{Supporting numerical evidence for Conjecture \ref{conj:symm-Xvmax}.
Each $\ell$ corresponds to a different symmetric matrix $M_{\ell} \in \mathbb{R}^{n(\ell) \times n(\ell)}$.
The markers in the top row are as follows.
Blue triangle markers correspond to $\Vert M_{\ell} \Vert^{1/4}$.
Magenta circle markers correspond to the conjectured maximum of $\big( \Vert X_{\ell}(\bm{v}) \bm{v} \Vert / \Vert \bm{v} \Vert \big)^{1/4}$.
For a given $\ell$, green markers correspond to values of $\big( \Vert X_{\ell}(\bm{v}) \bm{v} \Vert / \Vert \bm{v} \Vert \big)^{1/4}$ over 10,000 test vectors $\bm{v}$.
\textbf{Top left:} Each test vector $\bm{v}$ is chosen randomly.
\textbf{Top right:} Each of the 10,000 test vectors $\bm{v}$ from the left plot is iterated 10 times under $\bm{v} \gets X_{\ell}(\bm{v}) \bm{v}$ and the resulting vector is used to evaluate $\big( \Vert X_{\ell}(\bm{v}) \bm{v} \Vert /  \Vert \bm{v} \Vert \big)^{1/4}$.
\textbf{Bottom:} Eigenvalue distribution of $M_{\ell}$ for each $\ell$.
\label{fig:symm-Xv-conj}
}
\end{figure}
%

\clearpage
\section{Additional numerical tests in the skew-symmetric setting}
\label{SMsec:skew-additional}

Here we provide additional numerical tests to those already considered in Section \ref{sec:skew-num-res}.
Numerically computed rAA(1) convergence factors and those of the underling fixed-point iteration are shown in \cref{fig:skew-M-additional} for three different skew-symmetric matrices $M$.
In the 1st, 2nd, and 3rd tests, the matrices $M \in \mathbb{R}^{n \times n}$ are of dimension $n = 10, 14$, and 22 respectively.
These matrices were chosen as the skew-symmetric part of matrices output from MATLAB's \texttt{rand}($n,n$).
\begin{figure}[h!]
\centerline{
\includegraphics[width=0.475\textwidth]{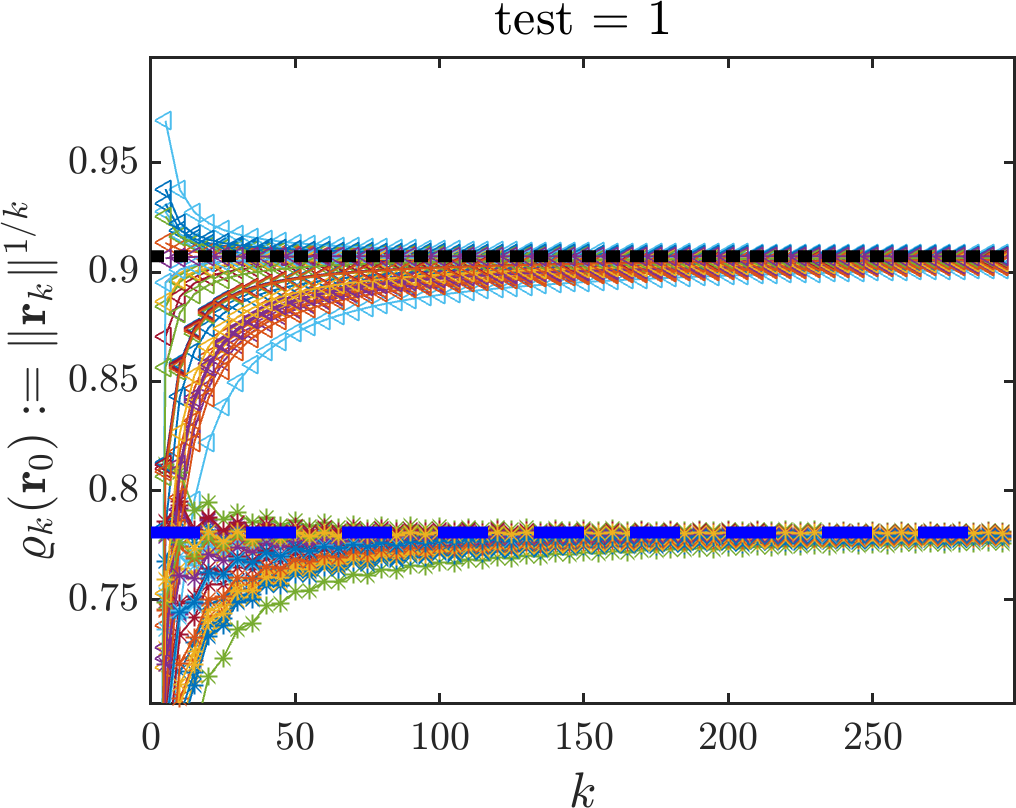}
\hspace{1ex}
\includegraphics[width=0.475\textwidth]{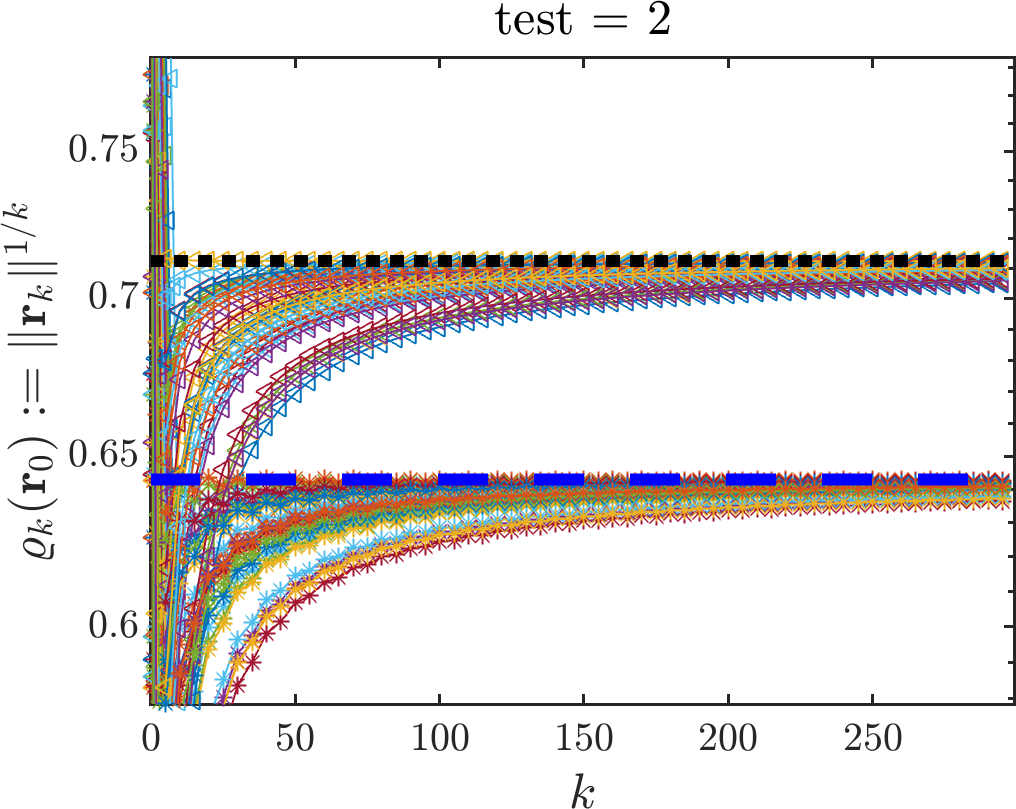}
}
\vspace{1ex}
\centerline{
\hspace{-1.5ex}
\includegraphics[width=0.475\textwidth]{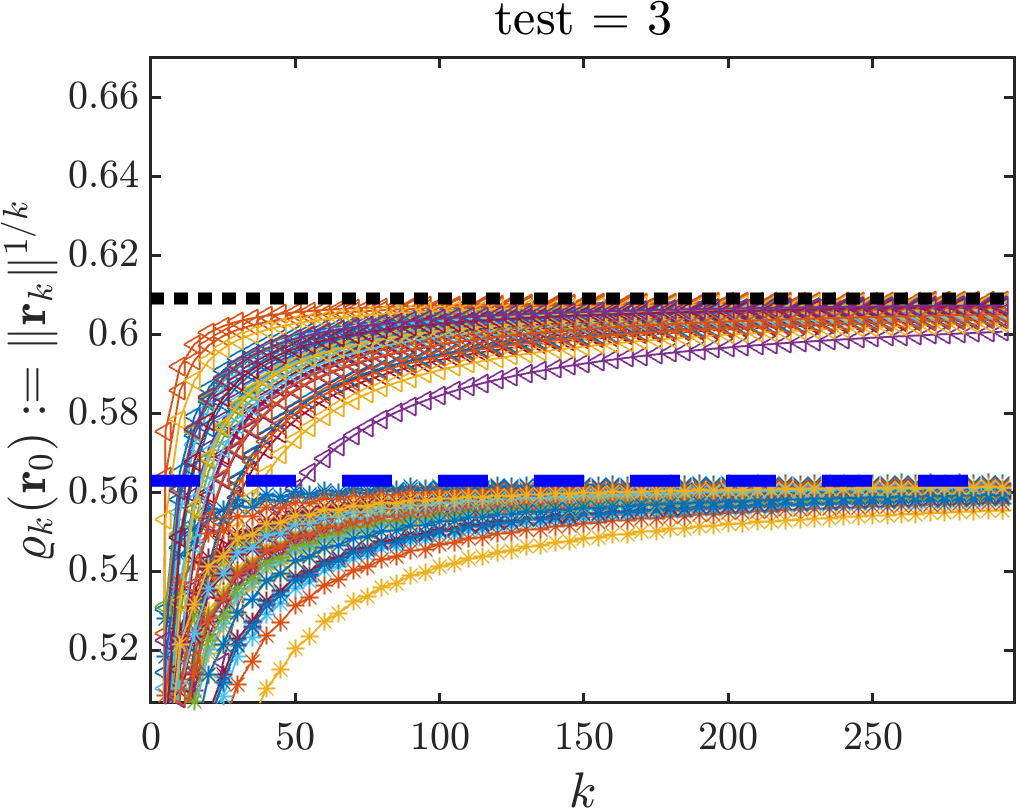}
\hspace{1.5ex}
\includegraphics[width=0.46\textwidth]{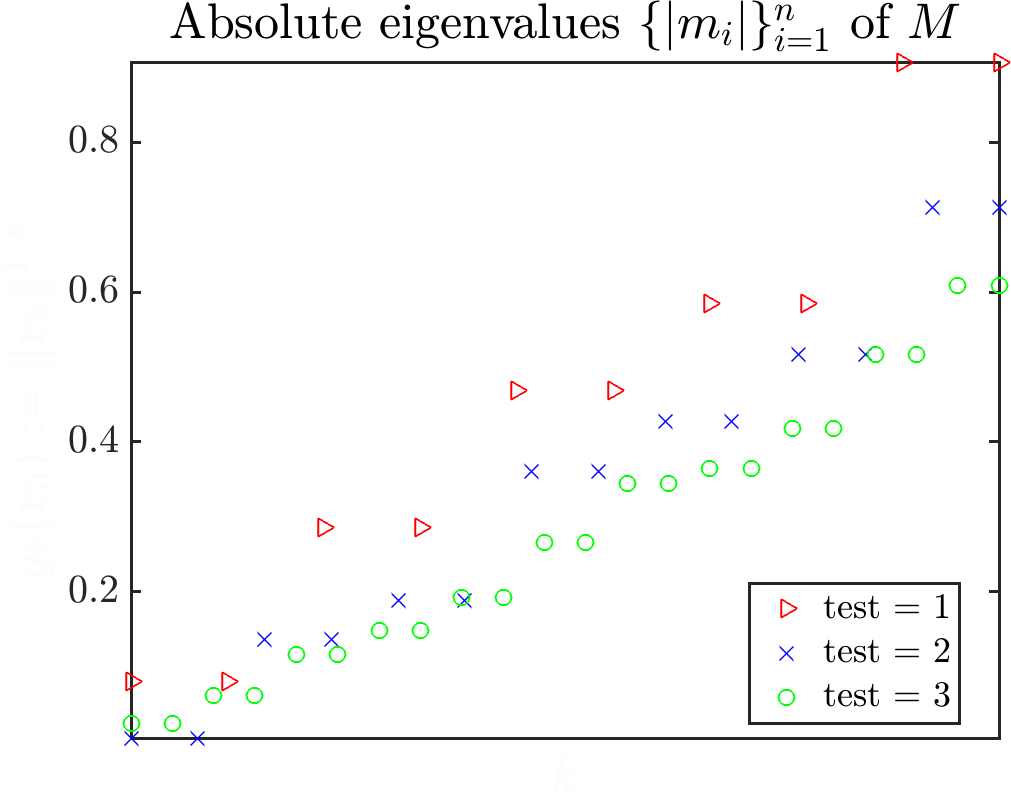}
}
\caption{Additional numerical tests in the case of skew-symmetric $M$.
For each $M$, shown is the convergence factor $\varrho_k(\mathbf{r}_0)$ as a function of $k$ (see Definition \ref{def:conv-fac}).
Triangle markers depict the underlying fixed-point iteration \eqref{eq:PI-iter}, and asterisk markers the rAA(1) iteration \eqref{eq:rAA1-iter}; note that the vector $\bm{b}$ in \eqref{eq:PI-iter} is zero.
For each matrix, each algorithm is initialized with 30 different $\bm{r}_0$ chosen at random.
The worst-case asymptotic convergence factor \eqref{eq:rho-skew-wc} for the underlying fixed-point iteration is shown in each plot as a thick, black dotted line, and that of \eqref{eq:rho-skew-wc} for the rAA(1) iteration is shown as the thick, blue dashed line.
The bottom right plot shows the absolute value of the eigenvalues for each test matrix $M$.
\label{fig:skew-M-additional}
}
\end{figure}

\end{document}

%% file: rAA1-SHARED.tex
\usepackage{lipsum}
\usepackage{amsfonts}
\usepackage{graphicx}
\usepackage{epstopdf}
\usepackage{algorithmic}
\ifpdf
  \DeclareGraphicsExtensions{.eps,.pdf,.png,.jpg}
\else
  \DeclareGraphicsExtensions{.eps}
\fi

\newsiamthm{conjecture}{Conjecture}
\crefname{conjecture}{Conjecture}{Conjectures}
\Crefname{conjecture}{Conjecture}{Conjectures}

\headers{Asymptotic convergence of rAA(1)}{O. A. Krzysik, H. De Sterck, A. Smith}

\title{Asymptotic convergence of restarted Anderson acceleration for certain normal linear systems\thanks{A published version of this preprint is available at \url{https://doi.org/10.1137/24M1672262}.}}

\author{Oliver A. Krzysik\thanks{Department of Applied Mathematics, University of Waterloo, Waterloo, Ontario, Canada. \textit{Present address:} Theoretical Division, Los Alamos National Laboratory, Los Alamos, New Mexico, USA
  (\email{okrzysik@lanl.gov}, \url{https://orcid.org/0000-0001-7880-6512}).}
\and Hans De Sterck\thanks{Department of Applied Mathematics, University of Waterloo, Waterloo, Ontario, Canada} 
  (\email{hans.desterck@uwaterloo.ca}).
  \and Adam Smith\thanks{Department of Applied Mathematics, University of Waterloo, Waterloo, Ontario, Canada
  (\email{ajsmith@uwaterloo.ca}).}
}

\usepackage{amsopn}
\DeclareMathOperator{\diag}{diag} 

\renewcommand{\i}[0]{\ensuremath{\mathrm{i}}}

\usepackage{xcolor}
\usepackage{bm} 

\usepackage{enumitem} 

\usepackage{comment} 

\usepackage{mathtools} 
\mathtoolsset{centercolon} 

\usepackage{amssymb} 
\usepackage{makecell} 
\usepackage{multirow} 
\usepackage{makecell} 
\usepackage{pifont}      

\usepackage{boldline}  

\usepackage{grffile} 

\Crefname{subsection}{Section}{Sections}
\crefname{subsection}{Section}{Sections}
\crefname{section}{Section}{Sections}

\usepackage{tikz}
\usetikzlibrary{calc}
\usepackage{graphbox}

\usepackage[font=small]{caption}
\usepackage{subcaption}

\newcommand{\ra}[0]{\ensuremath{\rangle}}
\newcommand{\la}[0]{\ensuremath{\langle}}

\DeclareMathOperator*{\argmin}{arg\,min}

\newcommand{\wh}[1]{\widehat{#1}} 								  

\newcommand{\wt}[1]{\widetilde{#1}}